\documentclass[a4paper,reqno,dvipsnames,11pt]{amsart}
\usepackage{amsfonts}
\usepackage{stmaryrd}

\usepackage[eulergreek]{sansmath}
\usepackage{todonotes}

\usepackage{color}
\usepackage{hyperref}

\newcommand{\Diver}{\mathrm{Div}}

\newcommand{\tim}{{\times}}

\newcommand{\Linftyn}{{L^\infty(\Omega; \MD)}}

\newcommand{\calR}{\mathcal{R}}

\usepackage[left=3cm,right=3cm,top=3cm,bottom=3cm,headsep=0.7cm]{geometry}
\usepackage[french, italian, english]{babel}
\usepackage{amsmath,amssymb,amsthm,mathrsfs,esint}
\usepackage{mathtools}
\usepackage{paralist}

\newcommand{\N}{{\mathbb N}}
\newcommand{\R}{{\mathbb R}}
\newcommand{\bbC}{{\mathbf C}}

\newcommand{\C}{{\mathbf C}}
\newcommand{\Q}{{\mathcal Q}}

\newcommand{\EE}{{\mathrm{E}}}

\newcommand{\foraa}{\text{for a.a.}}

\newcommand{\Rn}{{\R}^n}

\newcommand{\ds}{\, \mathrm{d} s}

\newcommand{\Dir}{\mathrm{Dir}}
\newcommand{\Neu}{\mathrm{Neu}}
\newcommand{\Gneu}{\Gamma_{\Neu}}

\newcommand{\hn}{\mathcal{H}^{n-1}}
\newcommand{\ol}{\overline}
\newcommand{\ul}{\underline}

\newcommand{\diver}{\mathrm{div}\,}

\newcommand{\sm}{\setminus}
\newcommand{\rmn}{\mathrm{n}}
\newcommand{\rmE}{\mathrm{E}}
\newcommand{\rmC}{\mathrm{C}}

\newcommand{\Mnn}{{\mathbb{M}^{n{\tim} n}_{\mathrm{sym}}}}
\newcommand{\Sym}{\mathrm{Sym}}

\newcommand{\Mb}{\mathcal{M}_{\mathrm{b}}}

\newcommand{\Lnn}{L^2(\Omega;\Mnn)}

\newcommand{\AC}{\mathrm{AC}}
\newcommand{\BD}{\mathrm{BD}}
\newcommand{\sig}[1]{\mathrm{E}(#1)}
\newcommand{\Lin}{\mathrm{Lin}}
\newcommand{\calQ}{\mathcal{Q}}

\newcommand{\weakto}{\rightharpoonup}

\newcommand{\weaksto}{\stackrel{*}\rightharpoonup}
\newcommand{\Spx}{\mathbf{X}}
\newcommand{\Spxw}{\mathbf{X}_{\mathrm{weak}}}

\newcommand{\ass}{a_{\mathrm{m}}}
\newcommand{\Hs}{H^{\mathrm{m}}}
\newcommand{\MD}{{\mathbb M}^{n{{\tim}}n}_{\mathrm{D}}}
\newcommand{\As}{A_{\mathrm{m}}}
\newcommand{\eps}{\varepsilon}

\newcommand{\sft}{\mathsf{t}}
\newcommand{\sfu}{\mathsf{u}}
\newcommand{\sfz}{\mathsf{z}}
\newcommand{\sfp}{\mathsf{p}}
\newcommand{\sfe}{\mathsf{e}}
\newcommand{\sfs}{\mathsf{s}}
\newcommand{\sfF}{\mathsf{F}}
\newcommand{\sfw}{\mathsf{w}}
\newcommand{\sfvar}{\mathsf{x}}
\newcommand{\serifsigma}{{\sansmath \sigma}}
\newcommand{\sfq}{\mathsf{q}}
\def\dd{\;\!\mathrm{d}} 
\newcommand{\pairing}[4]{ \sideset{_{ #1 }}{_{ #2 }}  {\mathop{\langle #3 , #4
\rangle}}}

\newcommand{\dev}{\mathrm{D}}
\def\calE{{\mathcal E}} \def\calD{{\mathcal D}} \def\calH{{\mathcal H}}

\newcommand{\calK}{\mathcal{K}}

\newcommand{\down}{\downarrow}
 \newcommand{\MbD}{{\Mb(\Omega \cup \partial_D\Omega; \MD)}}
 
 \newcommand{\Gdir}{\Gamma_{\Dir}}
 \newcommand{\calEv}{\calE}
 \newcommand{\Ez}{\Phi}
 \newcommand{\dipname}{\mathrm{H}}
 \newcommand{\aein}{\text{a.e.\ in }}
 
 \newcommand{\Qsp}{\mathbf{Q}}
 \newcommand{\rmD}{\mathrm{D}}
 
 \newcommand{\norm}[2]{\| #1\|_{#2}}
\newcommand{\gBa}{\mathsf{X}}
\newcommand{\gHi}{\mathsf{H}}
\newcommand{\gR}{\mathsf{R}}

\theoremstyle{plain}
\begingroup
\theoremstyle{plain}
\newtheorem{theorem}{Theorem}[section]
\newtheorem{maintheorem}{Theorem}

\newtheorem{proposition}[theorem]{Proposition}
\newtheorem{lemma}[theorem]{Lemma}
\theoremstyle{definition}
\newtheorem{definition}[theorem]{Definition}
\theoremstyle{remark}
\newtheorem{remark}[theorem]{Remark}

\newtheorem{notation}[theorem]{Notation}
\newtheorem{hypothesis}[maintheorem]{Hypothesis}
\endgroup

\newcommand{\slope}[3]{\mathbb{S}_{\mathsf{#1}}(#2,#3)}
\newcommand{\congdist}[3]{\mathbb{D}_{\mathsf{#1}}(#2,#3)}
\newcommand{\congdistname}[1]{\mathbb{D}_{\mathsf{#1}}}
\newcommand{\congdistvisc}[3]{\mathsf{d}_{\mathsf{#1}}(#2,#3)}
\newcommand{\congdistviscname}[1]{\mathsf{d}_{\mathsf{#1}}}
\newcommand{\congdistviscq}[3]{\mathsf{d}^2_{\mathsf{#1}}(#2,#3)}
\newcommand{\congdistq}[3]{\mathbb{D}^2_{\mathsf{#1}}(#2,#3)}

\newcommand{\congdistqtildeu}[3]{\tilde{\mathbb{D}}_{\mathsf{#1}}(#2,#3)}
\newcommand{\Mfunz}[5]{\mathcal{M}_{#1}(#2,#3,#4,#5)}
\newcommand{\Mfunzname}[1]{\mathcal{M}_{#1}}
\newcommand{\REDfunz}[5]{\mathcal{M}_{#1}^{\mathrm{red}}(#2,#3,#4,#5)}
\newcommand{\REDfunzname}[1]{\mathcal{M}_{#1}^{\mathrm{red}}}

\newcommand{\calM}{\mathcal{M}}
\newcommand{\Epp}{\mathcal{E}_{\scriptstyle \tiny \mathrm{PP}}}
\newcommand{\Hpp}{\mathcal{H}_{\scriptstyle \tiny \mathrm{PP}}}
\newcommand{\Dpp}{\mathscr{D}_{\scriptstyle \tiny \mathrm{PP}}}
\newcommand{\Qpp}{\Qsp_{\scriptstyle \tiny \mathrm{PP}}}
\newcommand{\BV}{\mathrm{BV}}

\newcommand{\HV}{\mathrm{HV}} 
\newcommand{\RV}{\mathrm{RV}} 
\newcommand{\ARV}{\mathrm{ARV}} 
\newcommand{\WE}{\mathrm{WE}} 
\newcommand{\FWE}[3]{\mathbf{Work}_{#3}^{#1}[#2]}
\newcommand{\FHV}[1]{\mathbf{HV}_{#1}}

\newcommand{\FREM}[1]{\mathbf{Rem}_{#1}}

\newcommand{\SiKappa}[2]{\Sigma \mathrm{K}_{#1}(#2)}

\newcommand{\DIN}[2]{\mathfrak{I}_{#1}^{#2}}
\newcommand{\subd}{r}
\newcommand{\subdiv}[3]{\subd_{#1,#2}^{#3}}
\newcommand{\cardi}[2]{M_{#1,#2}}
\newcommand{\subdivhat}[3]{\rho_{#1,#2}^{#3}}
\newcommand{\cardihat}[2]{L_{#1,#2}}
\newcommand{\partg}[2]{\mathfrak{s}_{#1}^{#2}}
\newcommand{\cardipart}[1]{\mathfrak{S}_{#1}}
\newcommand{\CV}{\mathrm{V}}
\newcommand{\isoC}{\mathbb{C}}
\newcommand{\sfx}{\mathsf{x}}

\newcommand{\ut}{\mathfrak{u}}
\newcommand{\pt}{\mathfrak{p}}
\newcommand{\loc}{{\mathrm{loc}}}
\DeclareSymbolFont{extraup}{U}{zavm}{m}{n}
\DeclareMathSymbol{\varheart}{\mathalpha}{extraup}{86}
\DeclareMathSymbol{\vardiamond}{\mathalpha}{extraup}{87}

\numberwithin{equation}{section}

\title[Visco-plastic approximation for damage with plasticity]{On the visco-plastic approximation 
\smallskip
\\
of a rate-independent 
 \smallskip
\\
coupled elastoplastic damage model}

\date{\today, revised version.}

\begin{document}

\author{Vito Crismale}
\address{V.\ Crismale, Dipartimento di Matematica “Guido Castelnuovo”, Sapienza Università di Roma, Piazzale Aldo
Moro 2, I--00185 Rome -- Italy}

\email{vito.crismale\,@\,uniroma1.it}

\author{Giuliano Lazzaroni}
\address{G.\ Lazzaroni, Dipartimento di Matematica e Informatica “Ulisse Dini”, Università degli Studi di Firenze, Viale
Morgagni 67/a, I--50134 Florence -- Italy}

\email{giuliano.lazzaroni\,@\,unifi.it}

\author{Riccarda Rossi}
\address{R.\ Rossi, DIMI, Universit\`a degli studi di Brescia. Via Branze 38, I--25133 Brescia -- Italy}
\email{riccarda.rossi\,@\,unibs.it}

\begin{abstract}
 In this paper we study a rate-independent system for the propagation of damage and plasticity.
To construct solutions we resort to approximation in terms of viscous evolutions, where  viscosity  affects \emph{both} damage and plasticity with the same rate.
The main difficulty arises from the fact that the available estimates do not provide sufficient regularity on the limiting 
evolutions
 to guarantee that forces and velocities are in a duality pairing,  hence we cannot use a chain rule for the driving energy.
Nonetheless, via careful techniques we can characterize the limiting rate-independent evolution by means of an energy-dissipation balance, which encodes the onset of viscous effects in the behavior of the system at jumps.
\end{abstract}

\maketitle

\section{Introduction}
In order to predict and prevent degradation and failure of materials, it is crucial to capture the interplay between different phenomena, such as damage and plasticity. In several applications, for instance to load-bearing structures, the propagation of such phenomena is very slow, if compared to the scale of internal oscillations of the body under examination.
 Thus, the system is considered as being in equilibrium at every instant. From a mathematical point of view, this amounts to the concept of \emph{quasistatic}, or \emph{rate-independent}, \emph{evolution}. In turn, rate-independent evolutions are idealized descriptions of processes where some phenomena are neglected, for instance the effects of viscosity. The analysis of quasistatic evolutions and their approximation by viscous evolutions has been  the object  of extensive mathematical literature in recent years. In fact, in order to understand how damage and plasticity grow, it is paramount to analyze their interaction already at a viscous level.

In this paper we study a rate-independent system for the propagation of damage and plasticity, by means of approximation in terms of viscous evolutions, where  viscosity  affects both damage and plasticity.
Specifically, in the setting of linear elasticity the body is determined by its reference configuration $\Omega\subset\R^n$,
$n=2,3$, 
 and 
the \emph{displacement} $u\colon (0,T){\tim} \Omega\to \R^n$, where $(0,T)$ is the time interval
 during which the process is observed.
 The degradation of the material is described by the \emph{damage} variable $z\colon (0,T){\tim} \Omega\to[0,1]$ and by the \emph{plastic strain} $p\colon(0,T){\tim} \Omega\to\MD$, where $\MD$ is the subspace of symmetrix matrices $\Mnn$ with null trace.
Together  with the \emph{elastic strain} $e\colon (0,T){\tim} \Omega \to \Mnn$,  the plastic strain $p$ complies with 
  the kinematic admissibility condition for the \emph{strain} $ \rmE(u)  = \frac{\nabla u + \nabla u^T}{2}$, namely,
 \begin{equation}
 \label{kam-intro}
    \rmE(u)  = e+p   \quad  \text{ in } (0,T){\tim} \Omega .
 \end{equation} 
\begin{subequations}
 \label{RIS-intro}
 The rate-independent system under consideration consists of 
 \begin{itemize}
\item[-] the momentum balance
\begin{align}
\label{mom-balance-intro}
 - \mathrm{div}\,\sigma = f  \quad \text{ in } (0,T){\tim} \Omega , 
 \end{align}
 where 
$\sigma$ is the  \emph{stress} tensor
\begin{equation} \label{stress-intro}
 \sigma =  \bbC(z)e  
 \quad   \text{ in } (0,T){\tim} \Omega , 
\end{equation}
with $\bbC\colon [0,1] \to \R^{n{\tim} n{\tim} n{\tim} n}$ the elastic stress tensor, and
 $f\colon  (0,T){\tim}\Omega\to\R^n$ is a given time-dependent external force; 
\item[-] the flow rule for the damage variable $z$
 \begin{align}
\label{flow-rule-dam-intro}
\partial\mathrm{R}(\dot{z})  + \As (z)+ W'(z) \ni - \tfrac12\bbC'(z)e : e  \quad  \text{ in } (0,T){\tim} \Omega ,
\end{align}
where $\partial \mathrm{R}  : \R \rightrightarrows \R$ denotes the convex analysis subdifferential of 
the function
\begin{equation}
\label{damage-pot-intro}
\mathrm{R}:\R \to [0,+\infty]  \  \text{ defined by } \   \mathrm{R}(\eta):= \left\{ \begin{array}{ll}
\kappa |\eta| & \text{ if } \eta\leq 0,
\\
+\infty & \text{ otherwise},
\end{array}
\right.
\end{equation}
with $\kappa>0$ the toughness of the material.  In \eqref{flow-rule-dam-intro}, the  \emph{nonlocal}  $\mathrm{m}$-Laplacian operator
   $\As \colon \Hs(\Omega ) \to \Hs(\Omega)^*$ features an exponent 
 $\mathrm{m}>\tfrac n2$. Finally, 
  $W\colon[0,1]\to\R$ is a suitable nonlinear, possibly nonsmooth, function; 
  \item[-] the flow rule for the plastic tensor 
 \begin{align}
&
\label{flow-rule-plast-intro}
\partial_{\dot{p}}  \mathrm{H}(z,\dot p) \ni \sigma_{\mathrm{D}}  \quad \text{ in } \Omega {\tim} (0,T),
\end{align}
where $\sigma_\dev$ is  the deviatoric part of the stress tensor $\sigma$, i.e.\ its orthogonal projection on $\MD$,  and $\dipname(z, \cdot)$ is  the density of the plastic dissipation potential.
 \end{itemize}
%
System \eqref{mom-balance-intro}--\eqref{flow-rule-plast-intro} is complemented by 
the boundary conditions 
\begin{equation}
 \label{viscous-bound-cond}
 u=w  \text{ on } (0,T){\tim}\Gdir ,   \qquad 
 \sigma \rmn =g \text{ on }  (0,T){\tim}\Gneu , \qquad \partial_{\rmn} z =0 \text{ on } (0,T){\tim}\partial\Omega , 
  \end{equation}
where $\Gdir$ is   the Dirichlet part of the boundary $\partial \Omega$ and  $w\colon\Omega\to\R^n$ a  time-dependent Dirichlet loading, while 
$\Gneu$ is the Neumann part of $\partial\Omega$ (disjoint from $\Gdir$), $\rmn$ its exterior unit normal, 
and $g\colon\Omega\to\R^n$ an assigned traction.
\par
 \end{subequations}
 The elasto-plastic damage model \eqref{RIS-intro},
which was first proposed and studied in \cite{AMV14,AMV15}, 
 includes rate-independent flow rules for damage and plasticity, both given in terms of threshold conditions: propagation starts when the damage variable or the deviatoric part of the stress reach the boundary of the stability set.
Note that $\mathrm{R}$ is the density of the dissipation potential for damage; hence,  the flow rule \eqref{flow-rule-dam-intro}
 encompasses
 the unidirectionality in the evolution of damage
 through the constraint $\dot z\leq 0$ in $\Omega{\tim} (0,T)$.  In turn,  $\mathrm{H}(z,\dot p)$ is the density of the plastic dissipation potential.  
 The coupling of the system is apparent
 both  from the dependence of the elasticity tensor $\bbC$ on the damage variable, and 
  from the $z$-dependence of $\mathrm{H}$. In particular, along the footsteps of \cite{AMV14,AMV15} we encompass 
softening in the model, which consists in the reduction of the yield stress as plastic deformation proceeds, see \eqref{propsH-2} below, as well as \cite[eq.\ (30)]{AMV14} and \cite{DMDMM08}.  
\par
The main impact of the analysis of such coupled model arises in the study of cohesive fracture.
Indeed, in parts of the material where plastic strain has been cumulated, one may observe nucleation of cohesive cracks and thus the emergence of fatigue phenomena, see  \cite[Section 5]{AMV14}, as well as \cite{Crismale, Crismale-Lazzaroni}, and references therein. 
  \par
  Rate-independent damage processes with plasticity have been extensively studied in recent years. Among  phase-field type models (i.e, featuring 
  the damage parameter and the plastic strain as internal variables), we mention
    e.g.\  \cite{BMR12, BonRocRosTho16,  RouVal17} for damage coupled with  plasticity with hardening,
   \cite{Cri17} for  damage  and strain-gradient plasticity,  \cite{Roub-Valdman2016} also accounting  for damage healing, \cite{MelScaZem19} for finite-strain plasticity with damage,  and \cite{DavRouSte19} for perfect plasticity and damage in a dynamical setting. 
\subsection*{The vanishing-viscosity approach and our results}
   \begin{subequations}
 \label{viscous-intro}
  The existence of \emph{quasistatic evolutions} (or   \emph{energetic solutions}) to the Cauchy problem for system \eqref{RIS-intro}
was proved in \cite{Crismale}. 
In this paper we will  instead address  the following viscous approximation of \eqref{RIS-intro}, for $\eps>0$ given:
 \begin{align}
 & - \mathrm{div}\, \sigma = f   \quad    \text{with $\sigma = \bbC(z) e$} &&  \text{ in } \Omega {\tim} (0,T), 
  \\
&  \partial\mathrm{R}(\dot{z})  +\eps \dot z  + \As (z)+ W'(z) \ni - \tfrac12\bbC'(z)e : e  &&    \text{ in } \Omega {\tim} (0,T),
\\
& \partial_{\dot{p}}  \mathrm{H} (z,\dot{p}) + \eps \dot p  \ni \sigma_{\mathrm{D}}  &&    \text{ in } \Omega {\tim} (0,T), 
 \end{align}
 \end{subequations}
supplemented by the boundary conditions \eqref{viscous-bound-cond}.   System  \eqref{viscous-intro} 
thus pertains to the class of \emph{rate-dependent} damage models, which have also been widely studied. Existence results
for rate-dependent damage processes are indeed available both in the case when plasticity effects are neglected
(starting from the pioneering papers \cite{BS, bss}),  and for models encompassing plasticity and even temperature
\cite{Roub-Tomassetti,Rossi17}. A  hallmark of these rate-dependent  systems is  the  gradient regularization of the 
damage parameter, 
 ensuring sufficient spatial regularity for $z$: typically, $p$-Laplacian operators with $p>n$ feature in the damage flow rule. In systems \eqref{RIS-intro}
 and \eqref{viscous-intro}, along the footsteps of \cite{KRZ13}
 we have instead resorted to a non-local, but \emph{linear},
  operator for analytical reasons.
\par
The existence of a solution $(u_\eps,z_\eps,p_\eps)$ to the Cauchy problem for  \eqref{viscous-intro} can be proven for instance by time discretization, cf.\ Theorem \ref{thm:exist-visc} ahead. 
In the approximated system both flow rules for damage and plasticity feature a viscous regularization, specifically they contain the time derivatives of the damage variable and of the plastic strain.
Such regularization can be tuned through the parameter $\eps>0$. In fact, we study the limit as $\eps\down0$, expecting to find (a version of) the rate-independent system \eqref{RIS-intro} in the limit.
This reflects the fact that,  again, both rate-independent damage and rate-independent plasticity are idealized processes where first-order terms are neglected.
We emphasize that, in \eqref{viscous-intro}, 
  the viscosities  in $p$ and in $z$  \emph{both} vanish with the same rate. 
\par
The strategy to derive \eqref{RIS-intro} from \eqref{viscous-intro} follows the \emph{vanishing-viscosity approach} explored in a wide literature, see \cite{MielkeRoubicek15} and references therein. Exploiting suitable a priori estimates on the approximate solutions, uniform with respect to $\eps$, one then passes to the limit as $\eps\down0$, finding a so-called \emph{Balanced Viscosity ($\BV$) solution} 
 \cite{MRS12,MRS13}
to \eqref{RIS-intro}. This technique is by now  standard: it  is based on   reparameterization and  was pioneered
in \cite{EfeMie06RILS}. Indeed, the reparameterized trajectories, defined in a new interval $[0,S]$,
 are uniformly Lipschitz in the new time scale, and so it is their limiting trajectory; for the system under consideration, a Balanced Viscosity solution is  in fact a quadruple $(\sft,\sfu,\sfz,\sfp)$, Lipschitz as a function of the (artificial) time $s\in [0,S]$,   where the rescaling function $\sft: [0,S]\to [0,T]$ records the  original  process time. Now,  it is to be expected that 
 solutions to system \eqref{RIS-intro}
  jump as functions of the (true) time $t\in [0,T]$:  the parameterized solutions  $(\sft,\sfu,\sfz,\sfp)$  keep track of this 
 in that jumps in the original time correspond to the regime in which $\sft'=0$, namely the function $\sft$ is frozen. The notion of $\BV$ solution then provides additional information on the behavior of the rate-independent system in a jump regime. In fact, it encodes the possibility that, between two stable states, there occurs
either a slow transition, corresponding to quasistatic propagation in the original time scale, or a  fast transition, 
where the system displays viscous behavior. 
All of this is encompassed in the (single) 
 energy-dissipation balance
 \begin{equation}
\label{EDB-intro}
 \begin{aligned}
 &
   \Epp(\sft(s),\sfq(s)) + \int_{0}^{s}
   \calM_0(\sft(r),\sfq(r),\sft'(r),\sfq'(r)) \dd r  
   \\
   & =  \Epp(\sft(0),\sfq(0)) +\int_{0}^{s} \partial_t \Epp (\sft(r), \sfq(r)) \, \sft'(r) \dd r  \qquad \text{for all } s \in [0,S],
   \end{aligned}
\end{equation}
cf.\ Definition \ref{def:parBVsol-PP},
 where $\Epp$ is the (overall) driving energy functional for system \eqref{RIS-intro} and $\sfq$ is a place-holder for the rescaled triple $(\sfu,\sfz,\sfp)$. In fact, in  \eqref{EDB-intro} the energy released between the initial time and a given final time $s$ (in the
 artificial time scale) is balanced by the work of the external forces, and by a term involving the `vanishing-viscosity contact potential' $\calM_0$
 (see \eqref{def:M0} for its explicit definition), 
  which keeps track of the occurrence of 
 slow/fast transitions in  
 in the jump regime.  The main result of this paper, \textbf{Theorem  \ref{mainth:1}}, states the convergence of (reparameterized) viscous trajecctories to a $\BV$ solution $(\sft,\sfu,\sfz,\sfp)$  satisfying 
\eqref{EDB-intro}.  
\par
 When applying  the vanishing-viscosity approach  to our setup, one major analytical difficulty arises already when dealing with the a priori estimates,
  uniform w.r.t.\ the parameter $\eps$, needed for the vanishing-viscosity limit. 
 Indeed, we may  deduce  the first set of  basic estimates for the viscous solutions   from an energy-dissipation balance 
 that is tightly connected with 
  the gradient structure of \eqref{viscous-intro} (cf.\ Proposition \ref{prop:-aprio-est}).
  Relying on such estimates,  we prove convergence as $\eps\down 0$ of the reparameterized viscous trajectories (along a suitable subsequence),  to a quadruple $(\sft,\sfu,\sfz,\sfp)$ 
  that fulfills the inequality $\leq$ in \eqref{EDB-intro}, 
where the energy at the current (artificial) time $s$
 is estimated from above by the initial energy. 
However, the available estimates do \emph{not} provide sufficient regularity on the limiting 
 parameterized curves 
 to guarantee that forces and velocities are in a duality pairing.  Hence we are  \emph{not} able to prove the validity of a chain rule
  for the energy $\Epp$ evaluated along the curve  $(\sft,\sfu,\sfz,\sfp)$; in turn,  chain rules are the  tool   
  usually employed to obtain the  energy inequality $\geq $ in \eqref{EDB-intro}  (hereafter, we shall refer to it as the \emph{lower} energy inequality, as therein the energy evaluated at the current time is estimated from below by the initial energy). 
\par
In order to prove the energy-dissipation balance, 
in this paper 
we have to carry out finer arguments, 
based on the  \emph{sole} validity of those estimates,
  uniform w.r.t.\ the viscosity parameter, that are deducible from the viscous energy-dissipation balance.
This  is 
 in the spirit of  the analysis 
 carried out in \cite{MRS13}, whose results are, however, not directly  applicable  here 
 because in the present functional setup we do not have a suitable  chain rule at our disposal,
 cf.\ Remark \ref{rmk:sth-missing}. 
 Rather, the techniques advanced  here  in fact  revisit the approach developed in \cite{DalDesSol11}, see also 
\cite{BabFraMor12}. More precisely,
  we are going to obtain 
 the desired  lower energy inequality 
 by exploiting the fact  that the all  terms featuring in \eqref{EDB-intro}
  are the limit of suitable Riemann sums, defined on carefully chosen subdivisions of the time interval
 that distinguish between the quasistatic regime and the jump regime, and by using  that an approximate form of the lower energy inequality holds along the intervals 
 associated with the partitions. A remarkable difference between this work and the previous contributions is that now the chain rule is not available even in the jump regime, which forces us to perform a discrete approximation also there, cf.\ Remark~\ref{rmk:differencetoBFM}. 
  \par
 We finally mention that the vanishing-viscosity approximation of system  \eqref{RIS-intro} has already been addressed by means of  different viscously regularizing systems, whose structure
allowed for enhanced estimates on the 
reparameterized viscous trajectories, leading to  better 
regularity for the limiting curve and, ultimately, allowing  for the validity of  the chain rule.
First of all, 
in \cite{Crismale-Lazzaroni}   \eqref{RIS-intro} was approximated by an elasto-plastic damage system in which   the viscous regularization only involved the damage variable, while 
the evolution of the plastic variable was kept  rate-independent; in particular,
 following this approach,
 the possible  emergence of viscous behavior for the plastic strain  at 
  jumps is \emph{not} included in the mathematical description of the evolution.
  \par
An alternative  approximation scheme was studied in \cite{Crismale-Rossi19}  in a setting with multiple rates (in the spirit of the approach in \cite{MRS14, MiRo23}). In fact, 
in the viscous system  addressed in  \cite{Crismale-Rossi19},  the momentum equation was also viscously regularized, 
and the   displacement and the plastic strain were set to  converge to elastic equilibrium and   rate-independent evolution,  respectively, 
at a faster rate than the  damage variable. Furthermore,
the vanishing-viscosity regularization  was 
combined with a  vanishing-hardening approximation of the flow rule 
for  technical reasons related to the validity of the aforementioned \emph{enhanced} a priori estimates. The notion of $\BV$ solution  in  \cite{Crismale-Rossi19} thus enjoys better regularity properties than 
those obtained in the forthcoming   Theorem   \ref{mainth:1}. In particular, 
relying on the validity of the chain rule for the driving energy functional, in \cite{CLR} we have
provided a characterization of the $\BV$ solutions from  \cite{Crismale-Rossi19}
in terms of a system of subdifferential inclusions  that 
illustrates in a precise way the possible onset of viscous behavior for the system at jumps. 
Such a differential characterization, tightly related to the validity of a chain rule, 
seems to be out of reach for the present notion of $\BV$ solution.
However, exploiting that very same differential 
characterization, in \cite{CLR} we have proved that the $\BV$ solutions  of \cite{Crismale-Rossi19} `essentially' coincide  (after an initial transition layer)  with those from  \cite{Crismale-Lazzaroni},
 where viscous regularization in $p$ was \emph{not} present. 
\par 
For this reason, in this paper we consider the alternative scheme \eqref{viscous-intro}, which is a natural regularization of 
the rate-independent system \eqref{RIS-intro}. 
Indeed, system \eqref{viscous-intro} reflects the fact that plasticity and damage are tightly connected \cite{AMV14,AMV15}: loosely speaking, 
setting $\eps=0$ formally leads to removing viscous perturbation both in plasticity and damage (which is not the case with the multi-rate regularization chosen in \cite{CLR}). From this viewpoint, a single-rate regularization as the one adopted in this paper seems more adequate. 
On the other hand, due to poor regularity and lack of chain rule, it remains an open problem to further characterize the $\BV$ solutions
fulfilling \eqref{EDB-intro}   e.g.\  by means of a differential characterization
as in \cite{CLR}, see also  Remark \ref{rmk:sth-missing} ahead.
 Thus we leave open the question if the solutions found in the present paper are different from those of
\cite{CLR} and \cite{Crismale-Lazzaroni}, in particular if in this case the presence of the viscous regularization in $p$ could be effectively detected in the limit. 
Anyhow, we maintain that the techniques developed in this paper, as a revisitation of  \cite{DalDesSol11,BabFraMor12} 
are interesting on their own   
 and could be useful in other contexts where the lack of a chain rule in any regime is a hindrance from the analysis of $\BV$ solutions.

\subsection*{Plan of the paper}
In Section \ref{s:2} we settle most of  the notation and preliminary results for our analysis. In Section \ref{s:3}, after introducing the reparameterization technique for the vanishing-viscosity
limit (Sec.\ \ref{ss:3.1})
 and the 
energetics for the perfectly plastic system (Sec.\ \ref{ss:3.2-PP}), we give the notion of \emph{(parameterized) Balanced Viscosity} solution to the system for perfect plasticity and damage
and state the convergence of  (reparameterized) viscous trajectories to a $\BV$ solution in Theorem \ref{mainth:1}.
Its proof is carried out throughout three sections. Indeed, Sec.\ \ref{s:4} contains the compactness arguments and the proof of the upper energy-dissipation inequality. Its converse is proved in Sections \ref{s:5}--\ref{s:6}. 
\par
Finally, the Appendix collects some auxiliary results used in proofs scattered throughout the paper. In particular, 
\begin{enumerate}
\item  in Appendix \ref{ss:proofLemVito} we carry out the proof of a technical result used in Sec.\ \ref{ss:5-step2};
\item in Appendix \ref{ss:appD} we prove a compactness result based on an  elliptic regularity estimate extending \cite[Thm.\ 9.1]{DalDesSol11};
\item  Appendix \ref{ss:gen-duality} contains a  convex analysis result; 
\item in Appendix \ref{ss:appE} we adapt a result from \cite{FraGia2012} to obtain strong convergence for the stresses;
\item we 
settle some results concerning the approximation of the external loadings in Appendix \ref{ss:5-prelim}.
\end{enumerate}

\setcounter{tocdepth}{1}  

\section{Preliminaries}
\label{s:2}
First of all, let us fix some notation that will be used throughout the paper. 
\begin{notation}[General notation  and preliminaries] 
\label{not:1.1}
\upshape
 Given a Banach space $\Spx$, we will
 denote by 
$\pairing{}{\Spx}{\cdot}{\cdot}$  both 
  the duality pairing between $\Spx^*$ and $\Spx$ and that between $(\Spx^n)^*$ and $\Spx^n$; we  will   just write $\pairing{}{}{\cdot}{\cdot}$ for the inner Euclidean product in $\R^n$. 
  Analogously, we  will
 indicate  by $\| \cdot \|_{\Spx}$ the norm in $\Spx$ and, most  often, use the same symbol for the norm in $\Spx^n$, while we will  just write $|\cdot| $ for the Euclidean norm in $\R^n$. 
We  will
  denote by  
 $B_r(0)$  the open ball of radius $r$, centered at $0$, in  $\R^n$.  For the Lebesgue measure of a set $A\subset \R^n$ we will use both notations $|A|$
 and $\mathcal{L}^n(A)$. 
\par
We  will
 denote  by $\Mnn$ the space of the symmetric  $(n{{\tim}}n)$-matrices, and by $\MD$ the subspace of the deviatoric matrices   with null trace. 
 Any  $\eta \in \Mnn$ can be written as 
$
\eta = \eta_\dev+ \frac{\mathrm{tr}(\eta)}n I
$,
with $\eta_\dev$ the orthogonal projection of $\eta$ into $\MD$;  $\eta_\dev$ will be  referred to  as the deviatoric part of $\eta$.  
 We  write   $\Sym(\MD;\MD)$ for   the set of symmetric endomorphisms on $\MD$.
\par
We  will
 often use  the short-hand notation $\| \cdot \|_{L^p}$, $1\leq p<+\infty$, for the $L^p$-norm on the space $L^p(O;\R^m)$, with $O$ a measurable subset of $\R^n$,  and analogously we  will    write  $\| \cdot \|_{H^1}$. Furthermore,  we  will
 denote  by $\Mb(O;\R^m)$ the space of  bounded Radon measures on $O$ with values in $\R^m$. 
\par
Let $v: \Omega{\tim} (0,T) \to \R$ be  differentiable w.r.t.\ time a.e.\ on $\Omega{\tim} (0,T)$. We  will
  denote by $\dot{v}:\Omega{\tim} (0,T)\to\R$ its (almost everywhere defined) partial time derivative. However, as soon as we consider $v$ as a (Bochner) function from $(0,T)$ with values in a suitable Lebesgue/Sobolev space $\Spx$
 (possessing the Radon-Nikod\'ym property),  and $v$ is in  the space $\AC ([0,T];\Spx)$,  we  will
  denote by $v':(0,T) \to \Spx$ its  (almost everywhere defined)  time derivative.
\par
 Finally, we will
 use the symbols
$c,\,c',\, C,\,C'$, etc., whose meaning may vary even within the same   line,   to denote various positive constants depending only on
known quantities. 
\end{notation}
\paragraph{\bf Functions of bounded deformation.}
The space  $\BD(\Omega)$ of  \emph{functions of bounded deformation}  is    defined by
\begin{equation}
\BD(\Omega): = \{ u \in L^1(\Omega;\R^n)\, : \ \sig{u} \in  \Mb(\Omega;\Mnn)  \},
\end{equation}
 where $\Mb(\Omega;\Mnn) $ is   the space of  bounded Radon measures on $\Omega$ with values in $\Mnn$,  with norm $\| \lambda\|_{\Mb(\Omega;\Mnn) }: = |\lambda|(\Omega)$, and  $|\lambda|$ the variation of the measure. 
By the Riesz representation theorem, 
 $\Mb(\Omega;\Mnn) $ can be identified with the dual of the space $\mathrm{C}_0(\Omega;\Mnn )$. The space $ \BD(\Omega)$ is endowed with the graph norm
\[
\| u \|_{\BD(\Omega)}: = \| u \|_{L^1(\Omega;\R^n)}+ \| \sig{u}\|_{\Mb(\Omega;\Mnn) },
\]
which makes it a Banach space. In fact,  $\BD(\Omega)$ is the dual of a normed space, cf.\ \cite{Temam-Strang80}.
\par
In addition to the strong convergence induced by $\norm{\cdot}{\BD(\Omega)}$,  the duality 
from \cite{Temam-Strang80} 
defines  a notion of weak$^*$ convergence on $\BD(\Omega)$: 
a sequence $(u_k)_k $ converges weakly$^*$ to $u$ in $\BD(\Omega)$ if $u_k\weakto u$ in $L^1(\Omega;\R^n)$ and $\sig{u_k}\weaksto \sig{u}$ in $ \Mb(\Omega;\Mnn)$; 
every bounded sequence in $\BD(\Omega)$ has  a weakly$^*$ converging subsequence. 
The space $\BD(\Omega)$ is contained in $L^{n/(n{-}1)}(\Omega; \R^n)$; every bounded sequence in $\BD(\Omega)$ has    a subsequence converging weakly in $L^{n/(n{-}1)}(\Omega;\R^n)$ and strongly in $L^{p}(\Omega;\R^n)$  for every $1\leq p <\frac n{n-1}$. 
\par
 Finally, we recall that for  every $u \in \BD(\Omega)$  the trace $u|_{\partial\Omega}$ is well defined as an element in $L^1(\partial\Omega;\R^n)$, and that (cf.\ \cite[Prop.\ 2.4, Rmk.\ 2.5]{Temam83}) a Poincar\'e-type inequality holds:
\begin{equation}
\label{PoincareBD}
\exists\, C>0  \ \ \forall\, u \in \BD(\Omega)\, : \ \ \norm{u}{L^1(\Omega;\R^n)} \leq C \left( \norm{u}{L^1(\Gamma_\Dir;\R^n)} + \norm{\sig u}{\Mb(\Omega;\Mnn)}\right)\,.
\end{equation}

\subsection{Assumptions}
\label{ss:2.1}
We now detail the standing assumptions for our analysis.
\begin{hypothesis}[The reference configuration]
\label{hyp:domain}
\sl
We suppose that 
 $\Omega \subset \R^n$,  $n\in \{2,3\}$,  is a bounded  Lipschitz  domain
 satisfying 
 the so-called \emph{Kohn-Temam condition}, namely $\partial \Omega = \Gdir \cup
\Gneu \cup \Sigma $ with 
 \begin{compactitem}
\item[-] $\Gdir, \,\Gneu, \, \Sigma$ pairwise disjoint;
\item[-]
 $\Gdir$ and $\Gneu$ relatively open in $\partial\Omega$, and $\Sigma = \partial\Gdir = \partial \Gneu$ their relative boundary in $\partial\Omega$;
\item[-] $\Sigma$ of class $\mathrm{C}^2$ and $\calH^{n-1}(\Sigma)=0$, and  $\partial\Omega$ Lipschitz and of class $\mathrm{C}^2$ in a neighborhood of $\Sigma$.
  \end{compactitem}  
  \end{hypothesis}
We  will
  work with the  space
\[
H_\Dir^1(\Omega;\R^n) := \left\{ u\in H^1(\Omega;\R^n) \colon u=0 \ \text{on}\ \Gdir \right\}\,.
\]
\subsubsection*{A divergence operator}
Any $\sigma \in  \Lnn $ such that $\mathrm{div}\, \sigma  \in L^2(\Omega;\R^n) $ induces
 the distribution $[\sigma \rmn]$  defined by 
\begin{equation}\label{2809192054}
\langle [\sigma \rmn] , \psi \rangle_{\partial\Omega}:= \langle \mathrm{div}\,\sigma , \psi \rangle_{L^2} + \langle \sigma , \rmE(\psi) \rangle_{L^2}  \qquad \text{for every }  \psi \in H^1(\Omega;\R^n).
\end{equation}
 By \cite[Theorem~1.2]{Kohn-Temam83} and \cite[(2.24)]{DMDSMo06QEPL} we have that $[\sigma \rmn] \in H^{-1/2}(\partial \Omega; \R^n)$; 
moreover, if $\sigma \in \mathrm{C}^0(\overline{\Omega};\Mnn)$, then  the distribution $[\sigma \rmn]$ fulfills $[\sigma \rmn] = \sigma \rmn$,
where the right-hand side is 
the standard pointwise product of the matrix $\sigma$ and the normal vector $\rmn$ in $\partial \Omega$. 
\par
 For the treatment of the perfectly plastic system for damage, it will be crucial to   work with the space 
\begin{equation}
\label{space-sigma-omega}
\Sigma(\Omega):= \{ \sigma \in \Lnn \colon  \mathrm{div}\,\sigma  \in L^n(\Omega;\R^n) , \ \sigma_\dev \in L^{\infty}(\Omega;\MD) \}. 
\end{equation} 
Furthermore, 
our choice of external loadings
 will ensure that 
the stress fields $\sigma$  that we consider have,  at equilibrium,  the additional property that $[\sigma \rmn] \in L^\infty(\Omega;\R^n)$ and $\sigma \in \Sigma(\Omega)$. Therefore,  any of such fields
induces a functional $-\Diver\,\sigma \in \BD(\Omega)^*$ via 
\begin{equation}\label{2307191723}
\langle -\Diver\,\sigma, v \rangle_{\BD(\Omega)} :=\langle - \mathrm{div}\,\sigma,   v \rangle_{L^{\frac{n}{n-1}}(\Omega; \R^n )}  + \langle [\sigma \rmn], v \rangle_ {L^{1}(\Gneu; \R^n)}  \qquad \text{for all } v \in \BD(\Omega). 
\end{equation}
With slight abuse of notation, we shall denote by $ -\Diver\,\sigma$ also the restriction of the above functional to $H^1(\Omega;\R^n)$.

\begin{hypothesis}[The elasticity tensor]
\label{hyp:elast}
\sl 
We assume that  $\bbC : [0,+\infty) \to   \Lin(\Mnn;\Mnn)  $
 fulfills  the following conditions:  
\begin{subequations}
\label{eq:C}
\begin{align}
&
 \bbC \in  \mathrm{C}^{1,1}([0,+\infty);  \Lin(\Mnn;\Mnn))\,,
 \label{spd}
\\
&
\exists\, \gamma_1,\, \gamma_2>0  \ \ \forall\, z \in [0,+\infty)  \ \forall\, \xi \in \Mnn \, : \quad  \gamma_1 |\xi |^2  \leq \bbC(z) \xi : \xi \leq \gamma_2 |\xi |^2  \,, \label{C2}
\\
&
 \forall\, z_o \in (0,1) \ 
\exists\, \alpha_{\bbC}>0 \, \  \ \forall\, z \in [z_o,1] \ \ \forall\, \xi \in \Mnn\, : \ \   \bbC'(z) \xi : \xi \geq \alpha_{\bbC}|\xi|^2\,.    \label{C3} 
\end{align}
 Furthermore,  we require that 
\begin{equation}
\label{C_dMdSS}
\begin{aligned} 
&
\exists\, \CV \in  \mathrm{C}^{1,1}([0,+\infty);[0,+\infty))\,,
\quad \CV \text{ bounded from below, } \inf_{z\in [0,+\infty)}\CV(z) \geq c_{\CV}>0 \,,
 \\
 & \qquad  \exists\, \isoC  \in  \Lin(\Mnn;\Mnn) \text{ positive definite, symmetric, \emph{isotropic}, such that}
  \\
  &  \qquad \qquad   \qquad \qquad   \qquad \qquad   \qquad \qquad    \forall\, z \in  [0,+\infty)\, : \
  \bbC(z)  = \CV(z) \, \isoC \,,
  \end{aligned}
\end{equation}
where  \emph{isotropic} implies that $\isoC A = 2\mu {A} + \lambda \mathrm{tr}({A}) {I} $ for all ${A}  \in \Mnn$, with $\lambda,\, \mu>0  $ the Lam\'e constants. 
\end{subequations}
\end{hypothesis}
\begin{remark}[On condition \eqref{C_dMdSS}]
\upshape
In view of the required structure 
 $  \bbC(z)  = \CV(z) \, \isoC$, 
conditions \eqref{spd}--\eqref{C3} could be reformulated   in terms of the function $\CV$. Nonetheless, we have opted for stating  \eqref{spd}--\eqref{C3}  in general,
regardless of \eqref{C_dMdSS}, because the latter structure condition will be exploited \emph{solely} in the proof of Proposition \ref{prop:regolaritaellittica} ahead,
cf.\ also Remark \ref{rmk:purtroppo}. 
In  turn,  Proposition \ref{prop:regolaritaellittica} will be used to prove the lower energy-dissipation inequality, i.e., $\geq $ in \eqref{EDB-intro}. 
\end{remark}

\begin{hypothesis}[Stored energy for damage]
\label{hyp:stored}
\sl
The stored energy for damage encompasses
  a gradient regularizing contribution, 
featuring the bilinear form  $\ass : \Hs(\Omega) {\tim}\Hs(\Omega) \to \R$,
\[
\ass(z_1,z_2): =
 \int_\Omega \int_\Omega\frac{\big(\nabla z_1(x) - \nabla
   z_1(y)\big)\cdot \big(\nabla z_2(x) - \nabla
   z_2(y)\big)}{|x-y|^{n  + 2 (\mathrm{m} - 1)}}\dd  x \dd y \quad \text{ with }  \mathrm{m} \in \left( \frac n2,2\right) \,,
\]
and 
an additional integral  term, with 
density $W$ satisfying 
\begin{subequations}
\begin{align}
&
W\in \rmC^{2}((0,+\infty); \R^+)\cap  \rmC^0([0,+\infty); \R^+{\cup}\{+\infty\})\,,\label{W1}\\
&
s^{2n} W(s)\rightarrow +\infty \text{ as }s\rightarrow 0^+\,, \label{W2}
\end{align}
\end{subequations}
where $W\in \rmC([0,+\infty); \R^+\cup\{+\infty\})$ means that  $W(0)=\infty$  and 
$W(z)\to+\infty$ if $z\to0$ as prescribed by \eqref{W2}. 
\end{hypothesis}


We  will
  denote by 
$\As: \Hs(\Omega) \to \Hs(\Omega)^*$ the operator associated with the bilinear form  $\ass$,  viz.\ 
\[
 \langle \As(z), w \rangle_{\Hs(\Omega)}  := \ass(z,w) \quad \text{for
every $z,\,w \in \Hs(\Omega)$.}
\]
We recall that $\Hs(\Omega)$ is a Hilbert space with the inner product 
\[
\pairing{}{\Hs(\Omega)}{z_1}{z_2}: = \int_\Omega z_1 z_2 \dd x + \ass(z_1,z_2).
\]
 Since we assume that   
$m>\tfrac n2$,  we have the compact embedding $\Hs(\Omega)\Subset \mathrm{C}(\overline\Omega)$. 
\begin{hypothesis}[Plastic dissipation] 
\label{hyp:plast-diss}
\sl The plastic dissipation density  $H:[0,+\infty)  {\tim}  \MD \to [0,+\infty)$  is continuous and enjoys the following properties:
\begin{subequations}
\label{propsH}
\begin{align}
&
\label{propsH-3}
\pi \mapsto H(z,\pi) \text{ is convex and $1$-positively homogeneous for all } z \in [0,1]\,,
\\
&
\label{propsH-2}
0\leq H(z_2,\pi) - H(z_1,\pi)  \qquad \text{for all } 0 \leq z_1\leq z_2   \text{ and all } \pi \in \MD \text{ with } |\pi|=1\,,
\\
&
\label{propsH-2+1/2}
\exists\, C_K>0 \ \forall\, z_1,\, z_2 \in [0,+\infty) \ \forall\,  \pi \in \MD 
\quad |H(z_2,\pi) - H(z_1,\pi) |\leq  C_K |\pi| |z_2{-}z_1| \,, 
\\
&
\label{propsH-4}
\exists\, \bar{r},\, \bar{R}>0 \ \ \forall\, (z,\pi) \in [0,+\infty)  {\tim}  \MD\, : \quad  \bar{r} |\pi| \leq H(z,\pi) \leq \bar{R}|\pi|\,.
\end{align}
\end{subequations}
\end{hypothesis}
\begin{remark}[Constraint sets]
\label{rmk:equivalent}
We point out that a  damage-dependent plastic dissipation density fulfilling \eqref{propsH} can be obtained as support function
$H(z,\pi): = \sup_{\sigma \in K(z)} \sigma {:} \pi$
 associated with a family  $(K(z))_{z\in [0,+\infty)}$ of closed and convex constraint sets in 
 $ \MD$
 fulfilling 
  \[
  \begin{aligned}
  &
\exists\, C_K>0 \quad \forall\, z_1,\, z_2 \in   [0,+\infty)\, :   \qquad d_{\mathscr{H}} (K(z_1),K(z_2)) \leq C_K  |z_1{-}z_2|, 
\\
  &
  \exists\,  0<\bar{r} <  \overline{R} \quad   \forall\, 0\leq z_1\leq z_2  \, : \qquad B_{\bar r} (0) \subset K(z_1)\subset K(z_2) \subset   B_{\overline{R}} (0),  
\end{aligned}
\]
with $d_{\mathscr{H}}$ the Hausdorff distance.  
\par
In fact, any dissipation density as in  \eqref{propsH} arises from a family of constraint sets with the above properties: it suffices to set 
$K(z) := \partial_\pi H(z,0)$, with $\partial_\pi H : \MD \rightrightarrows \MD$ the subdifferential of $H$ w.r.t.\ its second variable.  
\end{remark}
\par
 %
\begin{hypothesis}[Forces and data]
\label{hyp:data}
\sl
We  consider initial data 
\begin{equation}
\label{init-data} 
\begin{aligned}
&
u_0 \in  H_\Dir^1(\Omega;\R^n),  
\\
&  z_0 \in \Hs(\Omega) \text{ with } W(z_0) \in L^1(\Omega) \text{ and } z_0 \leq 1 \text{ in } \ol\Omega,  
\\
 &  p_0 \in L^2(\Omega;\MD). 
 \end{aligned}
\end{equation}
We require that 
  the volume force $f$ and the assigned traction $g$ fulfill  
 \begin{subequations} \label{hyp-data}
  \begin{equation}
\label{forces-u}
f \in H^1(0,T;  L^n(\Omega;\R^n) ) , \quad g \in H^1(0,T; L^\infty(\Gneu;\R^n)).
\end{equation}
Furthermore, as customary for perfect plasticity, 
 we  shall impose  a \emph{uniform safe load condition}, namely that  there exists 
\begin{equation}\label{2909191106}
\rho \in H^1(0,T; \Lnn) \quad\text{ with }\quad\rho_\dev \in H^1 (0,T;\Linftyn)
\end{equation} and there exists $\alpha>0$ such that  for every $t\in[0,T]$  (recall \eqref{2809192054})
\begin{equation}\label{2809192200}
-\mathrm{div}\,\varrho(t)=f(t) \text{ a.e.\ on }\Omega,\quad\qquad [\varrho(t) \rmn]= g(t) \text{ on } \Gneu,
\end{equation}
\begin{equation}
\label{safe-load}
\rho_\dev(t,x) +\xi \in K  \qquad \text{for a.a.}\ x\in\Omega \ \text{and for every}\ \xi\in\Mnn \ \text{s.t.}\ |\xi|\le\alpha.
\end{equation} 
As for the time-dependent Dirichlet boundary condition  $w$, we  assume  that
\begin{equation}
\label{dir-load} 
 w \in H^1(0,T;H^1(\R^n;\R^n)).  
\end{equation}
\end{subequations}
\end{hypothesis}
 By the properties of $W$, the requirement $W(z_0) 
\in L^1(\Omega)$ already encodes the condition $z_0 \geq 0$ in $\overline
\Omega$. 
 The body and surface forces $f$ and $g$ define the total load  function 
\[
F\colon [0,T] \to  \BD(\Omega)^*, \qquad 
\langle F(t), v \rangle_{\BD(\Omega)}: = \langle f(t), v\rangle_{ L^{n/(n{-}1)}(\Omega;\R^n)} + \langle g (t), v \rangle_{L^1(\Gamma_\Neu;\R^n)}
\]
for all $v \in \BD(\Omega).$
Observe that, combining \eqref{forces-u} with     \eqref{2909191106}--\eqref{safe-load}  yields
\begin{equation}
\label{propsF}
F \in H^1(0,T;\BD(\Omega)^*) \quad \text{and} \quad 
-\Diver\,\varrho(t)=F(t) \text{  for all $t \in [0,T]$.}
\end{equation}

\subsection{Energetics for the viscous system} 
The damage dissipation density $ \mathrm{R} : \R \to [0,+\infty] $
 from \eqref{damage-pot-intro} 
 induces
 the dissipation  potential
 \begin{equation}
 \label{damage-dissipation-potential}
 \calR : L^1(\Omega) \to [0,+\infty], \qquad   \calR(\eta) : = \int_{\Omega}\mathrm{R}(\eta(x)) \dd x. 
 \end{equation}
We will work with the subdifferential of the restriction of $\calR$ to $\Hs(\Omega)$, namely with the operator
$\partial_{\Hs} \calR : \Hs(\Omega) \rightrightarrows \Hs(\Omega)^*$ defined by 
\[
\zeta \in \partial_{\Hs} \calR(\eta) \quad  \text{ if and only if }\quad  \calR(\omega)-\calR(\eta)
\geq \pairing{}{\Hs(\Omega)}{\zeta}{\omega{-}\eta} \quad \text{for all } \omega\in \Hs(\Omega)\,.
\]
In what follows, we will simply denote the above subdifferential by $\partial \calR$. 
The viscous system will also feature the viscously perturbed dissipation potential 
 \begin{equation}
 \label{visc-dam-dissipation-potential}
 \calR_\eps : L^1(\Omega) \to [0,+\infty], \qquad   \calR_\eps(\eta) : = \calR(\eta) + \frac\eps 2\| \eta\|_{L^2(\Omega)}^2\,,
 \end{equation}
and  $\partial \calR_\eps :  \Hs(\Omega) \rightrightarrows \Hs(\Omega)^*$  will denote its subdifferential in the duality pairing  $\pairing{}{\Hs(\Omega)}{\cdot}{\cdot}$. 
\par
The \emph{plastic dissipation  potential} $\mathcal{H} \colon \rmC^0(\overline{\Omega}; [0,+\infty))  {\tim}  L^1(\Omega;\MD) 
\to \mathbb{R}$ is defined by
\begin{equation}
\label{plast-diss-pot-visc}
\mathcal{H}(z,\pi):= \int_{ \Omega} H(z(x),\pi(x))\dd x \,.
\end{equation}  
Its convex analysis subdifferential w.r.t.\ the second variable is the operator
\begin{equation}
\label{subdiff-calH}
\begin{aligned}
&
\text{$\partial_\pi \calH \colon \rmC^0(\overline{\Omega}; [0,+\infty)) {\tim} L^1(\Omega;\MD)\rightrightarrows L^\infty(\Omega;\MD)$
defined by}
\\
&
\omega \in \partial_\pi \calH(z, \pi) \quad \text{ if and only if } \quad  \omega(x)\in  \partial_\pi  H(z(x),\pi(x))  \quad \foraa\, x \in \Omega\,.
\end{aligned}
\end{equation} 
Recall that the dissipation potential  density $H$ is associated with a family $(K(z))_{z\in [0,+\infty)}$ of convex subsets of $\Omega$, cf.\ Remark \ref{rmk:equivalent}. 
Then, 
for a given $z\in  \rmC^0(\overline{\Omega})$, we set
\begin{equation}
\label{calK-sets}
\begin{aligned}
\calK_z(\Omega) &  = \{ \omega \in  L^1(\Omega;\MD)\, : \ \omega(x) \in K(z(x)) \ \foraa\, x \in \Omega\}
\\
&  =
  \{ \omega \in  L^1(\Omega;\MD)\, : \ \omega(x) \in \partial_\pi H(z(x),0) \ \foraa\, x \in \Omega\} = \partial_\pi \calH(z,0),
  \end{aligned}
\end{equation}
and observe that  $\calH(z,\cdot)$ is the support function  of the set $\calK_z(\Omega)$, namely 
\begin{equation}
\label{support-calH}
\calH(z,\pi)= \sup_{\omega \in \calK_z(\Omega) } \int_\Omega \omega(x) \colon \pi(x) \dd x \qquad \text{for all } \pi \in L^1(\Omega;\MD)\,.
\end{equation}
\par
We will also work with the viscously perturbed potential
\begin{equation}
\label{viscous-Heps}
\mathcal{H}_\eps \colon \rmC^0(\overline{\Omega}; [0,+\infty))  {\tim}  L^2(\Omega;\MD) \to [0,+\infty), \quad  \mathcal{H}_\eps(z,\pi) : = \mathcal{H}(z,\pi)+\frac\eps2 \|\pi\|_{L^2(\Omega)}^2\,.
\end{equation}
\par

We introduce the stored elastic energy
$\calQ: L^2(\Omega;\Mnn) {\tim} \rmC^0(\overline\Omega) \to \R$ 
\begin{equation}
\label{elastic-energy}
\calQ (e,z) := \frac12 \int_\Omega \bbC(z) e{:} e \dd x\,.
\end{equation}
The energy functional driving the evolution of the \emph{viscous} system is 
\begin{equation}
\label{viscous-energy}
\begin{aligned}
&
\text{$\calEv: [0,T]{\tim} H_{\Dir}^1(\Omega;\R^n) {\tim} \Hs(\Omega){\tim} L^2(\Omega;\Mnn) \to (-\infty,+\infty]$ defined by}
\\
&
\begin{aligned}
&  
 \calEv (t,u,z,p): =  &&  \calQ( e(t),z )  
 +\Ez(z)
 - \langle F(t), u+w(t) \rangle_{H^1(\Omega;\R^n)}
\\
& \text{ with }  && e(t): = \sig{u{+}w(t)} {-}p\,, \qquad   \Ez(z): = \frac12 \ass(z,z)   + \int_\Omega  W(z)  \,\mathrm{d}x\,.
\end{aligned}
\end{aligned}
\end{equation}
 In \eqref{viscous-energy} we have incorporated the boundary loading $w$ in the  elastic energy and in the term with the external force $F$. This reflects the fact that 
we will indeed impose the Dirichlet  condition for the displacements on $\Gdir$, cf.\ \eqref{viscous-bound-cond}, by working with a state variable 
in $ H_{\Dir}^1(\Omega;\R^n) $, so that $u{+}w$ in fact satisfies the desired boundary condition. The new variable $u{+}w$ will thus feature 
in the driving energy functional and in the statement of  our existence theorem for the viscous system, cf.\ Thm.\ \ref{thm:exist-visc} ahead.  
\subsection{Existence and a priori estimates for the viscous problem}
\label{ss:3.2}
\noindent
The state space for the viscous system is 
\[
\Qsp: = H_{\Dir}^1(\Omega;\R^n) {\tim} \Hs(\Omega){\tim} L^2(\Omega;\Mnn)\,.
 \] 
It was  proved in \cite[Lemma 3.3]{Crismale-Rossi19} that for every $t\in [0,T]$ 
the functional $q:= (u,z,p) \mapsto \calEv(t,q)$ is Fr\'echet differentiable on its domain $[0,T]{\tim} \mathscr{D}$, with 
\[
\mathscr{D}= \{ (u,z,p) \in \Qsp  \, : \ \min_{x\in \overline\Omega} z(x)>0 \text{ in }
 \overline{\Omega}\}.
\]
 It also follows from Hypothesis  \ref{hyp:data} that for all $t\in [0,T]$ the function $q\mapsto \calE(t,q)$ is Fr\'echet-differentiable 
 on $\mathscr{D}$, 
 with 
   \begin{subequations}
   \begin{align}
   \label{e:partial-tder}
   & 
   \rmD_q \calEv(t,q) = \left( -\mathrm{Div}\,\sigma(t) -F(t), \As z+W'(z) +\tfrac12 \bbC'(z) e(t){:}e(t), -\sigma_\dev(t) \right)  \in \Qsp^* 
   \intertext{ for all $ (t,q) \in [0,T]{\tim} \Qsp$, where $\sigma_\dev(t) $ is  the deviatoric part of the stress tensor $\sigma(t) = \bbC(z) e(t)$. Furthermore,  for all $q \in H_{\Dir}^1(\Omega;\R^n) {\tim} \Hs(\Omega){\tim} L^2(\Omega;\Mnn)$ the function $t\mapsto \calEv(t,q)$ belongs to  $H^1(0,T)$,
   with } 
   &
   \label{partial-deriv-t}
   \partial_t \calEv(t,q) = \int_\Omega  \sigma(t) : \sig{w'(t)} \dd x -\langle F'(t), u+w(t) \rangle_{H^1(\Omega)} -\langle F(t), w'(t) \rangle_{H^1(\Omega)} 
   \end{align}
   \end{subequations}
   $\foraa\, t\in (0,T)$. 
   Relying on this, it is easy to check  that system \eqref{viscous-intro} reformulates as the generalized gradient system 
 \begin{equation}
\label{1509172300} 
\partial_{q'} \Psi_\eps(q(t), q'(t)) + \rmD_q\calEv (t,q(t)) \ni 0 \qquad \text{in } \Qsp^* \quad \foraa\, t \in (0,T),
\end{equation}
involving the overall dissipation potential (degenerate w.r.t.\ the  variable $u$)
\[
 \Psi_\eps: \Qsp {\tim} \Qsp \to [0,+\infty] \qquad \Psi_\eps(q, q') 
  : =  \calR_\eps(z') + \calH_\eps(z,p') \,.
  \]
  A by now standard argument 
  (cf., e.g., \cite{AGS08,MRS2013}) 
    based on the validity of the chain rule for the driving energy $\calEv$
  shows that 
 a curve $q = (u,z,p)\in H^1(0,T;\Qsp)$
 is a solution to  
  the generalized gradient system \eqref{1509172300}, 
  if and only if  it satifies for every $[s,t]\subset[0,T]$ the energy-dissipation balance
\begin{equation}
\label{endissbal}
\begin{aligned}
&
 \calEv (t,q(t))  + \int_s^t \left( \Psi_\eps(q(r),q'(r)) + \Psi_\eps^*(q(r),{-} \rmD_q\calEv (r,q(r))) \right) \dd r
 \\
 & = \calEv (0,q(0))+\int_0^t \partial_t  \calEv (r,q(r)) \dd r.
 \end{aligned}
\end{equation}
Indeed, \eqref{endissbal} features the Fenchel-Moreau conjugate 
$
\Psi_\eps^*: \Qsp^* \to [0,+\infty)
$ 
of $\Psi_\eps$
which is defined, for 
$\xi = (\eta,\zeta,\omega) \in \Qsp^* = H_{\Dir}^1(\Omega;\R^n)^* {\tim} \Hs(\Omega)^*{\tim} L^2(\Omega;\Mnn)$, by 
%
\begin{equation}
\label{Psi-eps-star}
\begin{aligned}
&
\Psi_\eps^*(q,\xi)   =  \calR_\eps^*( \zeta) + \calH_\eps^*(z,\omega)  
\qquad  \text{with }
\\
& 
 \begin{cases}
 \displaystyle
 \calR_\eps^*( \zeta)  = 
 \frac1{2\eps} \inf_{\gamma \in \partial \calR(0)}   \mathfrak{f}_{L^2}  (\zeta{-}\gamma)^2 \,, & \zeta \in \Hs(\Omega)^*,
    \\[1em]
 \calH_\eps^*(z,\omega)  = 
   \frac1{2\eps} \min_{\rho \in \partial_\pi \calH(z,0)} \| \omega -\rho\|_{L^2(\Omega)}^2 \,, & \omega \in L^2(\Omega;\Mnn),
    \end{cases}
    \end{aligned}
\end{equation}
where
\[
\mathfrak{f}_{L^2}: \Hs(\Omega)^*  \to [0,+\infty]  \text{ is  given by }  \mathfrak{f}_{L^2}  (\eta) = \begin{cases} \|\eta\|_{L^2(\Omega)} & \text{if } \eta \in L^2(\Omega),
\\
+\infty & \text{otherwise}.
\end{cases}
\] 
Now, it can be easily checked that, if $ \calR_\eps^*( \zeta) <+\infty$, then the $\inf$ in its definition is indeed attained and thus 
$ \calR_\eps^*( \zeta) = \tfrac1{2\eps} \min_{\gamma \in \partial \calR(0)}  \|\zeta{-}\gamma\|_{L^2(\Omega)}^2 $. 
In what follows, we will use the short-hand notation 
\begin{equation}
\label{cong-distances}
\begin{aligned}
&
\congdistvisc{z}tq: = 
\begin{cases}
 \min_{\gamma \in \partial \calR(0)}   \mathfrak{f}_{L^2}  ( {-}\rmD_z\calEv (t,q) {-}\gamma )  &   \text{if } \inf_{\gamma \in \partial \calR(0)}    \mathfrak{f}_{L^2}  ( {-}\rmD_z\calEv (t,q) {-}\gamma) <+\infty, 
\\
+\infty &\text{otherwise}, 
\end{cases}
\\
& 
 \congdistvisc{p}tq: =   \min_{\rho \in \partial_\pi \calH(z,0)} \| {-}\rmD_p\calEv (t,q) -\rho\|_{L^2(\Omega)} \,.
 \end{aligned}
\end{equation}
Clearly,  the notation \eqref{cong-distances}  hints at  the fact  that both objects are in fact the distances  of 
${-}\rmD_z \calE$ 
and ${-}\rmD_p \calE$ 
(for ${-}\rmD_z \calE$, one has in fact to consider the  \emph{extended} distance)
from the respective stable sets $\partial\calR(0) \subset \Hs(\Omega)^*$ and  $\partial_\pi \calH(z,0) = \calK_z(\Omega)$ 
(the latter is in fact a subset of $ L^\infty(\Omega;\Mnn)$). 
\par
Slightly adapting
 the arguments from  \cite[Thm.\ 5.1]{Crismale-Rossi19},  
   we prove that the viscous system \eqref{viscous-intro} admits a solution additionally satisfying the energy-dissipation balance
   \eqref{2112200013} below.
\begin{theorem}
 \label{thm:exist-visc}
 Assume Hypotheses \ref{hyp:domain}--\ref{hyp:data}. 
 Then, there exists a triple $q=(u,z,p)$, with
 \begin{equation}
 \label{reg-of-sols}
 \begin{aligned}
 & 
  u \in   H_\Dir^1(0,T;H^1(\Omega;\R^n))\,,   
  \\
  & 
   z \in H^1(0,T;\Hs(\Omega))\,,  \qquad\qquad  0 \leq z(t,x) \leq 1 \quad \text{for every } (t,x)\in [0,T]{\tim}\overline\Omega\,,
   \\
   &
    p \in  H^1(0,T;L^2(\Omega;\MD))\,,
    \end{aligned}
 \end{equation}
  such that $(u{+}w,z,p)$  satisfies
 for almost all $t\in (0,T)$
  system \eqref{viscous-intro}
  in the sense 
  \begin{subequations}
 \label{visc-prob}
  \begin{align}
&
\label{1509172145-cont}
-\mathrm{Div}\big(\bbC(z(t)) e(t)\big) =  F(t)
&& \text{in } H_{\rm Dir}^1(\Omega;\R^n)^*, 
\intertext{with $e(t) = \sig{u(t){+}w(t)} - p(t)$,}
&
\label{1509172300a}
\partial \mathcal{R}_\eps(z'(t)) + \As(z(t)) + W'(z(t)) \ni -\frac12 \bbC'(z)  e(t) \colon e (t) 
 && \text{in } 
\Hs(\Omega)^*, 
\\
&
\label{1509172259}
\partial_{p} H(z(t), p'(t)) + \eps p'(t) \ni  \sigma_{\dev}(t)  && \text{a.e.\ in  } \Omega, 
\end{align}
joint with the initial conditions 
\begin{equation}
\label{initial-conditions-visc}
u(0)=u_0 \text{ in }  H_{\rm Dir}^1(\Omega;\R^n),   \qquad z(0) = z_0 \text{ in } \Hs(\Omega), \qquad p(0) =p_0 \text{ in } L^2(\Omega).
\end{equation}
\end{subequations}
 In fact, \eqref{1509172145-cont} holds at all $t\in [0,T]$. 
Furthermore,  the curve $ q =(u,z,p)$ satisfies the energy-dissipation balance  for every $t \in [0,T]$
\begin{equation}\label{2112200013}
\begin{split}
&\calEv(t,q(t))
+\int_0^t \left( \calR(z'(r)) {+} \calH(z(r),p'(r)) \right) \dd r 
\\
& \qquad \quad + \int_0^t \left(\frac{\eps}{2} \| z'(r)\|_{L^2(\Omega)}^2 {+}  \frac{\eps}{2} \| p'(r)\|_{L^2(\Omega)}^2 \right)  \dd r 
\\
& \qquad \quad 
+ \int_0^t \left(\frac1{2\eps} \congdistviscq zr{q(r)} {+}  \frac1{2\eps} \congdistviscq pr{q(r)}\right)  \dd r \
\\& = \calEv(0,q(0)) + \int_0^t \partial_t \calEv(r,q(r)) \dd r\,.
\end{split}
\end{equation}
\end{theorem}
\begin{proof}[Sketch of the proof]
As mentioned earlier, the proof of Thm.\ \ref{thm:exist-visc} basically relies on the same arguments
devised for 
\cite[Thm.\ 5.1]{Crismale-Rossi19}. Indeed, 
the main difference between 
system \eqref{viscous-intro} and  the viscous regularization  of 
\eqref{RIS-intro}
addressed in \cite{Crismale-Rossi19} resides in the fact that,
 here, no viscosity regularization is considered for 
   the momentum balance. It is then worthwhile to  comment the  temporal regularity $u \in  H^1(0,T;H^1(\Omega;\R^n))  $. Formally, it can be  justified by testing \eqref{1509172145-cont} by $u'$, \eqref{1509172300a}  by $z'$, and \eqref{1509172259} by $p'$. Rigorously, the existence of solutions to \eqref{visc-prob} can be shown by  time discretization as in \cite{Crismale-Rossi19}; the abovementioned estimate can be then performed  on the discrete system.
\end{proof}
\par
Let now $(q_\eps)_\eps  = (u_\eps,z_\eps, p_\eps)_\eps \subset H^1(0,T;\Qsp)$
be a family of solutions to 
the Cauchy problem  \eqref{visc-prob}. 
 The following result  collects the bounds that they enjoy uniformly w.r.t.\ the parameter $\eps$. Such estimates are a direct consequence of 
 the energy-dissipation balance
   (cf., e.g., the arguments in   \cite[Prop. 4.3]{Crismale-Rossi19}). In particular, the strict positivity property \eqref{strict-positivity}
 below derives from the energy estimate \eqref{enrgy-bound} via the growth condition \eqref{W2} on the potential $W$, cf.\ 
 \cite[Lemma 3.3]{Crismale-Lazzaroni}. 
\begin{proposition}
\label{prop:-aprio-est}
There exists a constant  $\widetilde{S}>0$  such that for every $\eps>0$
\begin{align}
& 
\label{enrgy-bound}
\sup_{t\in [0,T]} |\calE(t,q_\eps(t))| \leq  \widetilde{S}\,, 
\\
&
\label{2112200025}
\int_0^T \|z'_\varepsilon(t)\|_{L^1} \,\mathrm{d}t = \frac1{\kappa} \int_0^T \calR(z'(t)) \dd t \leq  \widetilde{S}\,,   
\\
&
\int_0^T \|p'_\varepsilon(t)\|_{L^1} \dd t \leq  \frac1{\overline{r}} \int_0^T \mathcal{H}(z_\varepsilon(t), p'_\varepsilon(t))\, \mathrm{d}t \leq  \widetilde{S}\,, 
\\
& 
 \int_0^T \sqrt{\|z_\eps'(t)\|_{L^2}^2{+} \|p_\eps'(t)\|_{L^2}^2} \,
  \sqrt{\congdistviscq zt{q_\eps(t)} {+}   \congdistviscq pt{q_\eps(t)}}  \dd t \leq  \widetilde{S}\,.   
\end{align}
Moreover, 
\begin{equation}
\label{strict-positivity}
\exists\, m_0 >0 \ \ \forall\, \eps>0 \ \ \forall\, t \in [0,T] \, : \quad \min_{x\in \overline\Omega} z_\varepsilon(t,x) \geq m_0\,.
\end{equation} 
 \end{proposition}
\section{Vanishing-viscosity analysis and main result}
\label{s:3}
\subsection{\bf Reparameterization} 
\label{ss:3.1}
Let   $(q_\eps)_{\eps} $ be a family of solutions to
the Cauchy problem  \eqref{visc-prob}.  Relying on Proposition \ref{prop:-aprio-est},
 we are going to reparameterize them by the 
 \emph{energy-dissipation arclength} %
 $ \widetilde{s}_\eps : [0,T]\to [0,\widetilde{S}_\eps]$ (with $\widetilde{S}_\eps: =  \widetilde{s}_\eps(T)$)
 defined by 
 \begin{subequations}
 \label{reparameterization-viscous}
 \begin{equation}\label{2006191707}
 \begin{aligned}
 \widetilde{s}_{\eps}(t):= \int_0^t \Big( &  1  + \|z_{\eps}'(\tau)\|_{L^1} + \|p_{\eps}'(\tau)\|_{L^1} 
 \\
  & \quad + 
  \sqrt{\|z_\eps'(\tau)\|_{L^2}^2{+}\|p_\eps'(\tau)\|_{L^2}^2}  \,  \sqrt{\congdistviscq{z} {\tau}{q_\eps(\tau)}{+} \congdistviscq{p} {\tau}{q_\eps(\tau)}} 
  \Big) \dd \tau\,.
  \end{aligned}
\end{equation}
   `Energy-dissipation'  refers to the interplay between the dissipation term $\sqrt{\|z'\|_{L^2}^2{+}\|p'\|_{L^2}^2} $
and the term
 $\sqrt{\congdistviscname{z}^2 {+} \congdistviscname{p}^2 }$, which contains the forces  $-\rmD_z \calE$ and $-\rmD_p \calE$.   
 From the estimates of Prop.\  \ref{prop:-aprio-est} it follows that 
$\sup_{\varepsilon}   \widetilde{S}_\eps \leq C $.
We set
\begin{equation}
\label{rescaled-curves}
\begin{aligned}
&
 \sft_\eps: =  (\widetilde{s}_\eps)^{-1}, \qquad 
 \sfq_\eps: =q_\eps\circ  \sft_\eps  = (\sfu_\eps, \sfz_\eps,\sfp_\eps) \,,
 \qquad   \sfe_\eps: = e_\eps\circ  \sft_\eps \,,  
 \qquad
 \serifsigma_\eps : =\sigma_\eps\circ  \sft_\eps \,,
 \end{aligned}
 \end{equation}
 \end{subequations}
that we may assume defined on a fixed interval $[0,S]$, with $S:=\lim_{\varepsilon \down 0} \widetilde{S}_\eps$ (the limit is taken along a suitable subsequence). 
\par
The rescaled solutions $( \sft_\eps,  \sfq_\eps): [0,S]\to [0,T]{\tim} \Qsp$
 satisfy     
 the normalization condition 
\begin{equation}\label{2906191307}
\begin{aligned}
&
{\sft'_\eps}(s) + \|\sfz'_\eps(s)\|_{L^1(\Omega)} +
\|\sfp'_\eps(s)\|_{L^1(\Omega;\MD)}
\\
&
\quad  +
\sqrt{\|\sfz_\eps'(s)\|_{L^2}^2{+}\|\sfp_\eps'(s)\|_{L^2}^2} \sqrt{\congdistviscq{z} {\sft_\eps(s)}{\sfq_\eps(s)}{+} \congdistviscq{p} {\sft_\eps(s)}{\sfq_\eps(s)}} 
 \equiv 1\quad \text{ for a.e.\ } s\in (0, S)\,,
 \end{aligned}
\end{equation}
 and the \emph{reparameterized} 
  energy-dissipation balance  for every $s\in [0,S]$
 \begin{equation}\label{2906191258}
 \begin{aligned}
 &
 \calEv(\sft_\eps(s) ,\sfq_\eps(s))+
 \int_0^s  \mathcal{M}_\eps (\sft_\eps(r),\sfq_\eps(r), \sft_\eps'(r), \sfq_\eps'(r)) \dd r
\\
&
 = \calEv(\sft_\eps(0),\sfq_\eps(0)) + \int_0^s \partial_t \calEv(\sft_\eps(r),\sfq_\eps(r)) \sft_\eps'(r) \dd r\,,
    \end{aligned}
\end{equation}
with  the functional
$ \mathcal{M}_\eps : [0,T]{\tim} \Qsp{\tim} (0,+\infty){\tim} \Qsp \to [0,+\infty] $ defined by 
\begin{equation}
\label{def-Meps}
\begin{aligned}
  \mathcal{M}_\eps(t,q,t',q'): =  &  \calR(z')+ \calH(z,p')     +I_{\{0\}} (  \|{-} \rmD_u \calE(t,q)\|_{(H_{\Dir}^1)^*})   
 \\
 & \quad   + \frac\eps{2t'}  \sqrt{\|z'\|_{L^2}^2 {+} \|p'\|_{L^2}^2}
 +\frac{t'}{2\eps}  \sqrt{ \congdistviscq zt{q} {+}   \congdistviscq pt{q} } \,,
\end{aligned}
\end{equation}
where $I_{\{0\}} : \R \to [0,+\infty] $ is the indicator function of the singleton $\{0\}$, defined by 
$I_{\{0\}}(r) =0$ if $r=0$, and $ I_{\{0\}}(r) =+\infty$ otherwise.

\begin{remark}[On the structure of $\mathcal{M}_\eps$]
\label{rmk:structureMeps}
\upshape
Recall
that $\rmD_u \calE(t,q) = - \mathrm{Div}\, \sigma(t)  - F(t)$  (cf.\ equation  \eqref{e:partial-tder}). Thus,
 the contribution  $I_{\{0\}} (  \|{-} \rmD_u \calE(t,q)\|_{(H_{\Dir}^1)^*})$ encodes the fact that, for curves $(\sft,\sfq)$ along which 
$ \int_0^S  \mathcal{M}_\eps (\sft_\eps,\sfq_\eps, \sft_\eps', \sfq_\eps') \dd r<+\infty$, the elastic equilibrium equation \eqref{1509172145-cont} holds
almost everywhere in $(0,S)$. In turn, the terms $ \frac\eps{2t'}  \sqrt{\|z'\|_{L^2}^2 {+} \|p'\|_{L^2}^2}$ and 
$\frac{t'}{2\eps}  \sqrt{ \congdistviscq zt{q} {+}   \congdistviscq pt{q} }$ convey the competition between the tendency of the system to be governed by \emph{viscous} dissipation, 
and that to relax towards the  state characterized by 
\begin{equation}
\label{local-stability+equi}
\congdistviscname z (t,q)= \congdistviscname p(t,q) =0,  \text{ and the elastic equilibrium 
$\|{-} \rmD_u \calE(t,q)\|_{(H_{\Dir}^1)^*}=0$.}
\end{equation} 
Now, in \cite[Lemma 7.4, Remark 7.5]{Crismale-Rossi19}  we have shown that conditions \eqref{local-stability+equi}
 occur in the \emph{rate-independent} regime,   when 
 the displacement variable $u$  is at elastic equilibrium  and one has  local stability for the damage parameter $z$ and the plastic strain $p$. 
\end{remark}
\par
Based on  \cite{MRS14, MiRo23, Crismale-Rossi19, CLR},  we expect that, up to a subsequence, 
the curves $(\sft_\eps,\sfq_\eps)_\eps$ converge to a curve
$(\sft,\sfq)$ satisfying an energy-dissipation balance akin to \eqref{2906191258} and 
involving a   \emph{vanishing-viscosity contact  potential} $\calM_0$, which 
will be properly introduced
(cf.\ \eqref{def:M0}  ahead)
 after some preliminary definitions.
\subsection{\bf Preliminary definitions and  energetics for the perfectly plastic damage system}
\label{ss:3.2-PP}
First of all, let us introduce the state space for the perfectly plastic damage system
  \begin{equation}
   \label{defQpp}
   \begin{aligned}
   \Qpp: =\big\{ & q=(u,z,p) \in \,\BD(\Omega){\tim} \Hs(\Omega) {\tim} \MbD\, : \\&   e: = \sig{u} -p \in \Lnn, \, 
    u \odot \mathrm{n} \,\hn + p =0 \text{ on }  \Gdir  \big\} ,
   \end{aligned}
   \end{equation}
where the condition $u \odot \mathrm{n} \,\hn + p =0 $ on $  \Gdir $ is a relaxation of the homogeneous Dirichlet condition $u=0$ on $\Gdir$.
We will consider $\Qpp $ endowed with the weak$^*$ topology of $\BD(\Omega)^*{\tim} \Hs(\Omega) {\tim} \MbD^*$. 
\subsubsection*{\bf {The driving energy functional for the perfectly plastic system}}
 The energy functional is the extension of  $\calEv$ from 
\eqref{viscous-energy} to the space
$[0,T]{\tim} \Qpp$. 
 To emphasize  the change in the topological setup, we will use a different notation for the extended energy. 
Thus, we define the functional $\Epp : [0,T]{\tim} \Qpp \to (-\infty,+\infty]$
\begin{equation}
\label{plastic-energy}
\begin{aligned}
&  \Epp(t,q): = \calQ( z,e(t))   
+ \Ez(z) 
- \langle F(t), u+w(t) \rangle_{\BD(\Omega)}\,
\end{aligned}
\end{equation}
with $\Ez$ from \eqref{viscous-energy}. 
Since 
 for every $q \in \Qpp$ we have that $ \sig{u} -p  \in L^2(\Omega;\Mnn)$, we have 
$e(t) :=  \sig{u{+}w(t)} -p \in L^2(\Omega;\Mnn) $ for every $t\in [0,T]$; also taking into account that $F \in H^1(0,T;\BD(\Omega)^*)$, we  find that $\Epp$ is well defined, with 
domain 
 $\mathrm{dom}(\Epp) = [0,T]{\tim} \Dpp$ where 
\[
\Dpp =  \{ (u,z,p) \in \Qpp  \, : \ \min_{x\in \overline\Omega} z(x)>0 \text{ in }
 \overline{\Omega}\}.
 \]
 Observe that for every $q\in \Dpp$ the function $t\mapsto \Epp(t,\cdot)$ is in $H^1(0,T)$ and for every $(t,q) \in [0,T]{\tim} \Qpp$ the partial time derivative $\partial_t \Epp(t,q)$ is given by \eqref{e:partial-tder}, with the duality pairings in $H^1(\Omega;\R^n)$ replaced by $\pairing{}{\BD(\Omega)}{\cdot}{\cdot}$.  
%
%
%
\subsubsection*{{\bf  The stress-strain duality.}}
  More in general, 
along the footsteps of  \cite{DMDSMo06QEPL} 
we introduce the class of 
\emph{admissible displacements and strains}
 associated with a function   $w \in H^1(\R^n; \R^n)$, 
  that is
\begin{equation*}
\begin{split}
A(w):=\{(u,e,p) \in  \, & \BD(\Omega) {\tim} L^2(\Omega;\Mnn) {\tim}  \MbD  \colon \\
& \rmE(v)   =e+p \text{ in }\Omega,\, p=(w{-}v){\,\odot\,}\rmn\,\hn \text{ on }  \Gdir \},
\end{split}
\end{equation*}
 (recall that   $ \mathrm{n} $ denotes the normal vector to $\partial \Omega$), where $\odot$ is  the symmetrized tensorial product. 
 The \emph{space of admissible plastic strains} is 
\begin{equation*}
\begin{aligned}
\Pi(\Omega) := \{p\in  \MbD  \colon  &  \exists \, (v, w,e)\in \BD(\Omega){\tim}   H^1(\R^n;\R^n)  {\tim} L^2(\Omega;\Mnn)\,
\\
& 
 \text{ s.t.}\, (v,e,p)\in A(w) \} .
\end{aligned}
\end{equation*}
 Given $\sigma \in \Sigma(\Omega)$ (cf.\ \eqref{space-sigma-omega}), $p \in \Pi(\Omega)$, and $v,\,e$ such that $(u,e,p)\in A(w)$, we define 
 the \emph{stress-strain duality}
\begin{equation}\label{sD}
\begin{aligned}
\langle [\sigma_\dev:p], \varphi\rangle:=-\int_\Omega \varphi\sigma\cdot (e{-} \rmE(w)  )\,\mathrm{d}x & -\int_\Omega\sigma\cdot[(v{-}w)\odot \nabla \varphi]\,\mathrm{d}x
\\
 &  -\int_\Omega \varphi \, ( \mathrm{div}\,\sigma)  \cdot (v{-}w)\,\mathrm{d}x
\end{aligned} 
\end{equation}
for every $\varphi \in \mathrm{C}^\infty_c(\R^n)$;
 in fact, this definition is independent of $v$ and $e$. 
 It can be checked that, for every  $\sigma \in \Sigma(\Omega)$ and $p \in \Pi(\Omega)$, the duality  $[\sigma_\dev{:}p]$ defines  a bounded Radon measure. 
Under these assumptions, 
$\sigma \in L^r(\Omega;\Mnn)$ for every $r < \infty$,
and   there holds
$\|[\sigma_\dev{:}p]\|_1\leq \|\sigma_\dev\|_{ L^\infty}\|p\|_{L^1}$.  
Restricting such  measure to $ \Omega \cup \Gdir $, we  set 
\begin{equation}
\label{duality-product}
 \langle \sigma_\dev\, |\,p\rangle:=[\sigma_\dev{:}p]( \Omega \cup \Gdir ).
 \end{equation}
  For later use, we record here the integration by parts formula  (see \cite{FraGia2012} for an equivalent version) 
 \begin{equation}
 \label{integr-by-parts-PP}
  \langle \sigma_\dev\, |\,p\rangle  = - \langle \sigma,  e- \sig w \rangle_{\Lnn} + 
\langle -\Diver\,\sigma, v-w \rangle_{\BD(\Omega)} 
 \end{equation}
 for all $ \sigma \in  \Sigma(\Omega)$, $(v,e,p) \in A(w)$.
%
\subsubsection*{\bf {The dissipation potential for perfect plasticity}.}
Let us emphasize that  $q\in \Qpp$ means that the plastic variable $p$ is now only a measure in $\MbD$. That is why,
the plastic dissipation mechanism for the rate-independent damage system will be encoded by a dissipation potential that extends
$\calH(p,\cdot) $ to $\MbD$
 via the theory of convex functions of measures \cite{Goffman-Serrin64}, see also \cite{Temam83}.  Namely, 
 we define $ \Hpp \colon \rmC^0(\overline{\Omega}; [0,1]) {\tim}  \MbD  \to \mathbb{R}$
  \begin{equation*}
\begin{aligned}
 &
\Hpp(z,\pi):= \int_{ \Omega \cup \Gdir} H\biggl(z(x),\frac{\mathrm{d}\pi}{\mathrm{d}\mu}(x) \biggr)\,\mathrm{d}\mu(x),
\end{aligned}
\end{equation*}  
 where  $\mu \in \MbD$ is a positive  measure such that $ \pi \ll \mu $ and $\frac{\mathrm{d} \pi}{\mathrm{d}\mu}$ is the Radon-Nikod\'ym derivative of 
  $\pi$    with respect to $\mu$; by the one-homogeneity of \ $H(z(x),\cdot)$, the definition of  $\Hpp$  does not depend  of $\mu$.
By \cite[Proposition~2.37]{AmFuPa05FBVF},    for every   $z \in \rmC^0(\overline{\Omega};[0,1])$  the functional 
 $p \mapsto \Hpp(z,p)$ 
is convex and positively one-homogeneous.
Moreover, for all $(z_k)_k\subset\rmC^0(\overline \Omega;[0,1])$ and $(\pi_k)_k\subset\MbD $   such that $z_k \rightarrow z$ in $\rmC^0(\overline\Omega)$ and $\pi_k \rightharpoonup \pi$ weakly$^*$ in $\MbD$ we have that 
 \begin{equation*}
 \Hpp(z,\pi) \leq \liminf_{k \rightarrow +\infty} \Hpp(z_k, \pi_k). \label{Hsci}
 \end{equation*}
Finally, 
 from \cite[Proposition~3.9]{FraGia2012} it follows that 
\begin{equation}\label{Prop3.9}
H\biggl(z, \frac{\mathrm{d}p}{\mathrm{d} |p|}\biggr)|p| \geq [\sigma_\dev:p],  \quad \text{as measures on }  \Omega \cup \Gdir, \quad \text{for all } \sigma \in\SiKappa z{\Omega},
\end{equation}
 where we use the notation
\begin{equation}
\label{SiKappa}
\SiKappa z{\Omega} := \{ \sigma \in \Sigma(\Omega)\, : \ \sigma_\dev \in  \calK_z(\Omega) \}\,. 
\end{equation} 
 In particular, we have 
\begin{equation}\label{eq:carH}
\Hpp(z, p)\geq  \sup_{\sigma  \in \SiKappa z{\Omega}}  \langle \sigma_\dev\, |\, p\rangle \,\qquad \text{for every $p \in \Pi(\Omega)$,} 
\end{equation} 
to be compared with  \eqref{support-calH}.  
 \subsubsection*{\bf Slopes.}
 We can no longer state that,
 for every fixed $t\in [0,T]$, the functional $\Epp(t,\cdot)$ is G\^ateaux-differentiable on $\Dpp$, 
 since the `natural candidate' for $\rmD_u \Epp$, namely the term $ -\mathrm{Div}\,\sigma(t) -F(t)$, need not be  an element in $\BD(\Omega)^*$.  In order to 
define $\calM_0$,  we will thus need a proxy for the   slope term 
$\| {-}\rmD_u \calEv\|_{(H^1)^*}$ 
which features in \eqref{def-Meps}. That is why, we define the \emph{slope} of $\Epp(t,\cdot)$ via 
 \begin{equation}
  \label{slope-pp}
  \slope utq: =  \|-\mathrm{Div}\,\sigma(t) -F(t)\|_{(H^1(\Omega;\R^n))^*} \qquad \text{for } (t,q) \in [0,T]{\tim} \Dpp \,.
  \end{equation}
  Observe that the above object is  well defined, since for every $(t,q) \in [0,T]{\tim} \Dpp$ we in fact have 
  $ -\mathrm{Div}\,\sigma(t) -F(t) \in   H^1(\Omega;\R^n)^*$.
  \par
   Analogously, we will need to introduce a
proxy for the distance term  
 $ \congdistvisc ptq$ from 
   \eqref{cong-distances}, since $\rmD_p \Epp(t,q)$ is no longer well defined as an element of $\MbD^*$. As a surrogate, it is natural to resort to the $L^2(\Omega;\MD)$-distance of $-\sigma_\dev$ from 
$\calK_z(\Omega)$.  
   Namely, we set 
\begin{equation}
\label{slope-funct-p}
\begin{aligned}
\congdist ptq:  = \mathrm{dist}_{L^2(\Omega)}({-}\sigma_\dev,  \calK_z(\Omega)) \,.
\end{aligned}
\end{equation}
Finally, for notational consistency, hereafter we will use the notation  
\begin{equation}
\label{slope-funct-z}
\congdist ztq: = \min_{\gamma \in \partial \calR(0)}    \mathfrak{f}_{L^2}( {-}\rmD_z\Epp (t,q) {-}\gamma) 
\end{equation}
in place of $\congdistvisc ztq$. 
 Note that,  whenever $\rmD_z\Epp (t,q) \notin L^2(\Omega)$, we have that $\congdist ztq = +\infty$. 

 Later on, we will resort to the following representations of $\congdistname p $ and  $\congdistname z $.
\begin{lemma}
\label{l:duality}
There holds for every $ (t,q) \in [0,T]{\tim} \Dpp $ 
\begin{subequations}
\label{duality-formulae}
\begin{align}
\label{duality-formulae-1}
& 
\congdist ptq = \sup_{ \substack{\eta_p \in L^2(\Omega;\MD) \\  \|\eta_p\|_{L^2}\le1} } \left(  \langle \sigma_\dev(t), \eta_p \rangle_{L^2(\Omega;\MD)} -\calH(z,\eta_p)\right), \\
&
\label{duality-formulae-2}
\congdist ztq = \sup_{ \substack{\eta_z \in \Hs(\Omega) \\  \|\eta_z\|_{L^2}\le1} }  \left(\langle {-} \As z {-} W'(z) - \tfrac12 \bbC'(z) e(t): e(t) , \eta_z  \rangle_{\Hs(\Omega)} - \calR(\eta_z)  \right).
\end{align}
\end{subequations}
\end{lemma}
\noindent 
Formulae \eqref{duality-formulae} follow from a duality argument already exploited in the proof of \cite[Lemma 3.6]{Crismale-Rossi19}; for the reader's convenience, we  will revisit and generalize it in Lemma \ref{l:gen-duality} ahead, which straightforwardly implies Lemma \ref{l:duality}, see also Remark \ref{rmk:proof-duality}.

\par
We will also rely on  the following lower semicontinuity result, cf.\ \cite[Prop.\ 7.7]{Crismale-Rossi19}. 
\begin{lemma}
\label{l:lsc-slopes}
Let $(t_k)_k, \, t  \subset [0,T]$ and  $ (q_k)_k \subset \Qsp $, $q \in  \Qpp$
fulfill as $k\to\infty$:  
\begin{align}
\nonumber
&
\begin{aligned}
&
t_k \to t, \ q_k \weaksto q \text{ in } \Qpp, 
\\
&  e(t_k)=\rmE(u_k+w(t_k))-p_k \to e(t)=\rmE(u+w(t))-p  \text{  in $\Lnn$.}
\end{aligned}
\intertext{Then,}
&
\label{eq:lsc-slopes}
\begin{cases}
\displaystyle
\slope utq \leq \liminf_{k\to\infty}  \|-\rmD_u \calEv (t_k,q_k)\|_{(H^1(\Omega;\R^n))^*}\,,
\smallskip
\\
\displaystyle
\congdist ztq \leq \liminf_{k\to\infty} \congdistvisc z{t_k}{q_k}\,, 
\smallskip
\\
\displaystyle
\congdist ptq \leq \liminf_{k\to\infty} \congdistvisc p{t_k}{q_k}\,. 
\end{cases}
\end{align}
\end{lemma}
\subsubsection*{\bf The vanishing-viscosity contact potential.}
Preliminarily, we introduce the place-holders
\begin{equation}
\label{calD-functionals}
 \begin{cases}
 \displaystyle
 \mathcal{D}(q') : =  \sqrt{\|z'\|_{L^2}^2 {+} \|p'\|_{L^2}^2}
     \\[1em]
  \displaystyle
 \mathcal{D}^*(t,q) := \sqrt{ \congdistq zt{q} {+}   \congdistq pt{q} }
 \end{cases}
 \qquad \text{for } (t,q) \in [0,T]{\tim} \mathscr{D}\,. 
\end{equation}
Now, whenever $\rmD_z \Epp(t,q) \notin L^2(\Omega)$, we have $ \congdist zt{q} = +\infty$ and hence
$\mathcal{D}^*(t,q) = +\infty$.  
We are in a position to introduce the functional 
$\Mfunzname 0: [0,T] {\tim} \Qpp {\tim} [0,T]{\tim} \Qpp \to [0,+\infty]$
 by setting
\begin{subequations}
\label{def:M0}
\begin{align}
& 
\Mfunz 0tq{t'}{q'}: =  \calR(z')+ \Hpp(z,p')   +I_{\{0\}} ( \slope utq) +     \REDfunz 0tq{t'}{q'} \,, &&
\intertext{where the reduced contact potential $  \REDfunzname 0$ is defined on $ [0,T] {\tim} \Qpp {\tim} [0,T]{\tim} \Qpp$ by } 
& 
 \REDfunz 0tq{t'}{q'}: = I_{\{0\}} (\congdist ztq ) +I_{\{0\}} (\congdist ptq)  && \text{if } t'>0 \,,
 \intertext{while for $t'=0$ special attention needs to be paid to the case in which $\mathcal{D}^*(t,q) = +\infty$. Indeed, along the footsteps of \cite[Def.\ 5.1]{MiRo23} we set }
 &
 \label{red-funz-at-0}
 \REDfunz 0tq0{q'}:  = \begin{cases}
  \mathcal{D}(q') \,  \mathcal{D}^*(t,q) & \text{if }  \mathcal{D}^*(t,q) <+\infty,
  \\
  0 & \text{if } q'=0 \text{ and } (t,q) \in \overline{\mathrm{dom}(\calD^*)}^{\mathrm{w}},
  \\
  +\infty & \text{otherwise} 
 \end{cases}
  && \text{if } t'=0\,,
  \end{align}
\end{subequations} 
where $\overline{\mathrm{dom}(\calD^*)}^{\mathrm{w}}$ is the weak closure of $\mathrm{dom}(\calD^*) = \{ (t,q) \in [0,T]{\tim} \mathscr{D}\, : \ \calD^*(t,q)<+\infty\}$
confined to energy sublevels, i.e.
\[
\overline{\mathrm{dom}(\calD^*)}^{\mathrm{w}} = \{ (t,q)\, : \ \exists\, (t_k,q_k)_k \subset  \mathrm{dom}(\calD^*) \text{ s.t. } t_k \to t, \ q_k \weaksto q, \ \sup_k \Epp(t_k,q_k)<+\infty\}\,.
\]
\begin{remark}[On the structure of $\Mfunzname 0$]
\label{rmk:structureM0}
\upshape
We emphasize that the definition of $\Mfunzname 0 $ encodes the fact that, when $t'>0$, in addition to the elastic equilibrium condition   $\slope utq=0$ 
the local stability conditions $\congdist ztq =0$ and $\congdist ptq=0$ hold,  
which is in accord with rate-independent evolution of    the system  (cf.\ Remark \ref{rmk:structureMeps}).  Instead, 
the contribution $  \mathcal{D}(q') \,  \mathcal{D}^*(t,q) $ to $\Mfunz 0tq{0}{q'}$ encodes the fact that when the system jumps (and the external time is frozen, i.e.\ $t'=0$), the system
may be governed by viscosity in $z$ and $p$. 
\end{remark}
 \par
We now introduce the concept of \emph{admissible parameterized curve},  tailored to describing the properties of parameterized curves   $(\sft,\sfq) \colon [0,S] \to [0,T] {\tim} \Qpp$ 
for which the term  $ \mathcal{D}(\sfq') \,  \mathcal{D}^*(\sft,\sfq)$ is well defined on a specific subset $A^\circ \subset [0,S]$. 
 Moreover, this notion also records the fact that, in the same way as for the rate-dependent system \eqref{viscous-intro}, also for the rate-independent system 
\ref{RIS-intro} the boundary condition for the displacements is \emph{weakly} formulated in terms of the variable $u{+}w$. 
\begin{definition}\label{def:admparcur}
A curve $(\sft,\sfq)=(\sft, \sfu, \sfz, \sfp)\colon [0,S] \to [0,T] {\tim} \Qpp$ is an \emph{admissibile parameterized curve for the perfectly plastic system} if 
\begin{subequations}\label{eqs:2906191823}
\begin{align}
&
\text{$\sft \in \AC ([0,S] ; [0,T])$ is nondecreasing},
\\
& \sfu \in  L^\infty \big(0,S; \BD(\Omega))\,,\label{2906191813}
\\
&  \sfz \in \AC([0,T];L^1(\Omega)),  
\\
 &  \sfp \in \AC ([0,T];  \MbD ), 
\\
&  \sfe=\mathrm{E}(\sfu + \sfw)-\sfp \in  L^\infty(0,S; \Lnn)  \quad \text{with } \sfw = w {\circ} \sft\,,  
    \label{2906191814}\\
&  (\sfz,\sfp) \in \AC_{\mathrm{loc}}(A^\circ; L^2(\Omega){{\tim}} L^2(\Omega; \MD))\,,   \text{ where } A^\circ \text{ is the open set} \label{2906191815}\\
\nonumber & \qquad A^\circ:=\{ s\in (0,S) \colon  \mathcal{D}^*(\sft(s), \sfq(s))>0\}\,,  \\
& 
 \label{2906191815+1}
\sft \text{ is constant in every connected component of }A^\circ\,.
\end{align}
\end{subequations}
We  will
  write $(\sft,\sfq) \in \mathscr{A}([0,S] ;[0,T] {{\tim}} \Qpp)$. 
\end{definition}
In view of Remarks \ref{rmk:structureMeps} \& \ref{rmk:structureM0}, hereafter  we will refer  to $A^\circ$ as the \emph{instability set}. 
%
  In what follows, we will also use the notation
 \begin{equation}
 \label{B-circ}
 B^\circ: = [0,S]\setminus A^\circ\,.
 \end{equation}
\subsection{Notion of $\BV$ solution and main result}
Eventually, we are in a position to specify the concept of $\BV$ solution to the perfectly plastic system that we will obtain in the vanishing-viscosity limit. 
  \begin{definition}\label{def:parBVsol-PP}
An admissible parameterized curve  $(\sft,\sfq)\colon [0,S] \to [0,T] {{\tim}} \Qpp$ is a \emph{(parameterized) Balanced Viscosity} ($\BV$, for short) solution to the  system for perfect plasticity  and damage 
\eqref{RIS-intro}
if 
 $(\sft,\sfq)$ fulfills the energy-dissipation balance 
\begin{equation}
\label{2906191838}
 \begin{aligned}
 &
   \Epp(\sft(s),\sfq(s)) + \int_{0}^{s}
   \calM_0(\sft(r),\sfq(r),\sft'(r),\sfq'(r)) \dd r  
   \\
   & =  \Epp(\sft(0),\sfq(0)) +\int_{0}^{s} \partial_t \Epp (\sft(r), \sfq(r)) \, \sft'(r) \dd r  
   \end{aligned}
\end{equation}
for all   $0\leq s \leq S$.
Finally,
we say that $(\sft,\sfq)$ is non-degenerate if it fulfills 
\[
\sft' + \|\sfz'\|_{L^2(\Omega)} + \|\sfp'\|_{L^1(\Omega;\MD)}  >0 \qquad \aein  (0,S)\,.
\]
\end{definition}

\begin{remark}
\label{rmk:sth-missing}
\upshape
Clearly, along an  admissible parameterized curve  the energy-dissipation balance in integral form \eqref{2906191838} is \emph{equivalent} to the \emph{pointwise} identity
\begin{equation}
\label{pointwise-EDB}
   \calM_0(\sft(s),\sfq(s),\sft'(s),\sfq'(s)) = - \frac{\dd}{\dd s}   \Epp(\sft(s),\sfq(s)) + \partial_t \Epp (\sft(s), \sfq(s)) \, \sft'(s)
\end{equation}
for almost all $s\in (0,S)$. 
In \cite{Crismale-Lazzaroni, CLR} (see also \cite{MRS14,MiRo23}), relation \eqref{pointwise-EDB}, \emph{combined} with the chain rule for $\calE$, has been the starting point for deriving an
additional characterization of $\BV$ solutions in terms of a system of subdifferential inclusions akin to the  viscously regularized system.
\par
 However, 
in the  present context  the validity of such a  chain rule along admissible curves is an open problem, due to their poor spatial regularity. That is why, we are not in a position to 
provide further characterizations of $\BV$ solutions other than \eqref{pointwise-EDB}.
\end{remark}
\par
The first main result of this paper states the convergence of the reparameterized viscous solutions to a $\BV$ solution.
 Recall that $\sfw = w {\circ} \sft$. 
\begin{theorem}
\label{mainth:1}
Under Hypotheses \ref{hyp:domain}--\ref{hyp:data}, let $(\eps_k)_k \subset (0,+\infty)$ be a null sequence  and 
$(q_{\eps_k})_k \subset H^1(0,T;\Qsp)$ be a sequence of solutions  to the Cauchy problem \eqref{visc-prob}. Let $\sft_{\eps_k}: [0,S] \to [0,T]$ be 
the time rescalings as in  \eqref{rescaled-curves} and  consider the reparameterized curves
$ \sfq_{\eps_k}: = q_{\eps_k}{\circ} \sft_{\eps_k}$.
\par
Then, there exist a (not relabeled) subsequence and an admissible parameterized curve $(\sft,\sfq) \in \mathscr{A}([0,S] ;[0,T] {{\tim}} \Qpp)$,  $\sfq=(\sfu,\sfz,\sfp)$,  such that
the following convergences hold  as $k\to \infty$
\begin{equation}
\label{pointwise-cvg}
\sft_{\eps_k}(s) \to \sft(s), \qquad \sfq_{\eps_k}(s) \weaksto \sfq(s) \text{ in } \Qpp \qquad \text{for all } s \in [0,S],
\end{equation}
and  $(\sft,\sfu{+}\sfw,\sfz,\sfp)$  is a $\BV$ solution to the perfectly plastic damage system \eqref{RIS-intro} with  additional 
temporal regularity, namely 
$\sft \in W^{1,\infty}(0,S;[0,T])$ and 
\begin{equation}
\label{additional-regularity}
\text{the mapping } \quad \begin{cases} 
\sfu \colon [0,S]\to \BD(\Omega)  \text{ is weakly$^*$ continuous,}
\\
\sfe \colon [0,S]\to \Lnn  \text{ is weakly continuous},
\\
\sfz \colon [0,S]\to \Hs(\Omega)  \text{ is weakly continuous},
\\
\sfp \colon [0,S]\to \MbD  \text{ is $1$-Lipschitz continuous}.
\end{cases}
\end{equation}
 In particular, the pair $(\sfe,\sfz)$ satisfies the elastic equilibrium equation \emph{everywhere} on $[0,S]$, i.e.
\begin{equation}
\label{EqEqEvery}
-\mathrm{Div}\big(\bbC(\sfz(s)) \sfe(s)\big) =  F(\sft(s))   \qquad \text{in } \BD(\Omega)^* \quad \text{for all } s \in [0,S]\,.
\end{equation} 
Furthermore,   on the instability set $A^\circ$ we have the additional regularity
\begin{equation}
\label{additional-regularity-instability}
\begin{cases}
\displaystyle
 \sfu \in  \rmC^0(A^\circ;  H_\Dir^1(\Omega;\R^n))\,,
\\
\displaystyle
 \sfe \in \rmC^0(A^\circ;  \Lnn)\,,
\\
\displaystyle
\sfz \in W_{\mathrm{loc}}^{1,\infty}(A^\circ; L^2(\Omega))\,,
\\
\displaystyle
\sfp \in W_{\mathrm{loc}}^{1,\infty}(A^\circ; L^2(\Omega;\MD))\,.
\end{cases}
\end{equation}  
\par
Finally,  we have the following \emph{enhanced} convergences for every $s\in [0,S]$
\begin{subequations}
\label{enh-cvg}
\begin{align}
\label{energy-convergence}
&
 \calEv(\sft_{\eps_k}(s) ,\sfq_{\eps_k}(s)) \longrightarrow \Epp(\sft(s),\sfq(s)),
 \\
 & 
 \label{by-potential-cvg}
  \int_0^s  \mathcal{M}_\eps (\sft_{\eps_k}(r),\sfq_{\eps_k}(r), \sft_{\eps_k}'(r), \sfq_{\eps_k}'(r)) \dd r  \longrightarrow \int_0^s \Mfunz 0{\sft(r)}{\sfq(r)}{\sft'(r)}{\sfq'(r)} \dd r \,.
\end{align}
\end{subequations}
\end{theorem}

\paragraph{\bf Outline of the proof.}  In Sec.\ \ref{s:4} we will settle all the  compactness properties  of the sequence 
$( \sfq_{\eps_k})_k$ and then, in Proposition \ref{prop:UEDI} ahead, we will  obtain the  
 \emph{upper} energy-dissipation inequality
\begin{equation}
\label{UEDI}
 \begin{aligned}
 &
   \Epp(\sft(s),\sfq(s)) + \int_{0}^{s}
   \calM_0(\sft(r),\sfq(r),\sft'(r),\sfq'(r)) \dd r  
   \\
   & 
    \leq  \Epp(\sft(0),\sfq(0)) +\int_{0}^{s} \partial_t \Epp (\sft(r), \sfq(r)) \, \sft'(r) \dd r  \quad \text{for all } s \in [0,S]
   \end{aligned}
\end{equation}
via lower semicontinuity arguments. (By convention, we shall refer to \eqref{UEDI} as an `upper' inequality because there the energy at the current time, $ \Epp(\sft(s),\sfq(s))$, is estimated from above by the energy at the initial time and by the work of the external forces.)
Sections \ref{s:5} and \ref{s:6}   will be devoted to the proof of the 
 \emph{lower} energy-dissipation inequality 
\begin{equation}
\label{LEDI}
 \begin{aligned}
 &
   \Epp(\sft(s),\sfq(s)) + \int_{0}^{s}
   \calM_0(\sft(r),\sfq(r),\sft'(r),\sfq'(r)) \dd r  
   \\
   &  \geq  \Epp(\sft(0),\sfq(0)) +\int_{0}^{s} \partial_t \Epp (\sft(r), \sfq(r)) \, \sft'(r) \dd r  \quad \text{for all } s \in [0,S]\,,
   \end{aligned}
\end{equation}
where  the energy $  \Epp(\sft(s),\sfq(s))$ is estimated from below, cf.\ the forthcoming Proposition  \ref{prop:LEDI}.  The enhanced convergences
\eqref{enh-cvg} are then a by-product of this limiting procedure, combined  with the validity of 
the energy-dissipation balance,
 via a standard argument that we choose not to detail. 

 Henceforth we will denote $\sfF(s):=F(\sft(s))$; recall also that $\sfw(s):=w(\sft(s))$, as introduced above.

\section{Proof of Theorem \ref{mainth:1}: the upper energy-dissipation inequality}
\label{s:4}
We start by collecting all a priori bounds on the sequence $(\sft_{\eps_k}, \sfq_{\eps_k})_k$.
\begin{lemma}
\label{l:aprio}
There exists $C>0$ such that for every $k \in \N$
\begin{subequations}
\label{aprio-bounds}
\begin{align}
& 
\label{aprio-bounds-1}
\| \sft_{\eps_k}\|_{W^{1,\infty}(0,S)} +
 \| \sfz_{\eps_k}\|_{W^{1,\infty}(0,S;L^1(\Omega))} + \| \sfp_{\eps_k}\|_{W^{1,\infty}(0,S;L^1(\Omega;\MD))} \leq C,
\\
& \label{aprio-bounds-1/2}
\sup_{s\in (0,S)}
\sqrt{\|\sfz_{\eps_k}'(s)\|_{L^2}^2+\|\sfp_{\eps_k}'(s)\|_{L^2}^2} \sqrt{\congdistviscq{z} {\sft_{\eps_k}(s)}{\sfq_{\eps_k}(s)}+ \congdistviscq{p} {\sft_{\eps_k}(s)}{\sfq_{\eps_k}(s)}}  \leq C,
\\
&
\label{aprio-bounds-2} 
\| \sfu_{\eps_k} \|_{L^\infty (0,S;\BD(\Omega))} + \| \sfe_{\eps_k} \|_{L^\infty (0,S;L^2(\Mnn))} + \|\sfz_{\eps_k} \|_{L^\infty(0,S;\Hs(\Omega))}  \leq C\,.
  \end{align}
\end{subequations}
\end{lemma}
\begin{proof}
Estimates \eqref{aprio-bounds-1} and  \eqref{aprio-bounds-1/2} obviously follow from the normalization condition \eqref{2906191307}. Next,  from \eqref{enrgy-bound}
we infer
 that 
\[
\exists\, C>0 \ \forall\, k \in \N,\, \forall\, s \in [0,S] \, : \quad  \int_\Omega \calQ(\sfz_{\eps_k}(s), \sfe_{\eps_k}(s)) \dd x  \leq C\,.
\]
Relying on the strict positivity condition \eqref{C2} we get
 $ \sup_{k\in \N} \| \sfe_{\eps_k} \|_{L^\infty (0,S;L^2(\Omega))} <+\infty$. Analogously, combining the estimate  for 
 $\sup_{[0,S]}\int_\Omega \ass (\sfz_{\eps_k}(s), \sfz_{\eps_k}(s)) \dd x$ with the information that $z_{\eps_k}(s,x) \in [m_0,1] $ for every 
 $(s,x) \in [0,S]{\tim} \overline\Omega$, cf.\ \eqref{strict-positivity}, we obtain the estimate for $ \|\sfz_{\eps_k} \|_{L^\infty(0,S;\Hs(\Omega)}$. 
 Finally, recalling that $\sig{\sfu_{\eps_k} {+} w{\circ} \sft_{\eps_k}} = \sfe_{\eps_k} +\sfp_{\eps_k}$ 
 and that $w\in H^1(0,T;H^1(\Omega;\R^n))$ we gather that 
 $\sig {\sfu_{\eps_k}} \in L^\infty (\Omega;\Mb(\Omega;\Mnn))$. We combine this with the information that 
 $\sfu_{\eps_k}(s)+w( \sft_{\eps_k}(s)) \in H_{\Dir}^1(\Omega;\R^n) $ for every $s\in [0,S]$  and use the Poincar\'e inequality
  in $\BD(\Omega)$ (cf.\ \eqref{PoincareBD}) to conclude the estimate for $\| \sfu_{\eps_k} \|_{L^\infty (0,S;\BD(\Omega)}$. 
\end{proof}

We can now settle  compactness properties of the sequence $(\sft_{\eps_k}, \sfq_{\eps_k})_k$.
Prior to that, 
 we recall that  the convergence in the space $\rmC^0([0,S]; \Spxw)$  is, by definition, defined by the convergence of  $\rmC^0([0,S]; (\Spx,d_\mathrm{weak}))$, where the metric $d_\mathrm{weak}$ induces the weak topology on a closed bounded set of the reflexive  space $\Spx$; the notation  $\rmC^0([0,S]; \Spx_{\mathrm{weak}^*})$ has the same meaning if $\Spx$ is the dual of a separable space.


\subsection*{Compactness}
%
By the Ascoli-Arzel\`a Theorem we obviously have that there exists $\sft \in W^{1,\infty}(0,S)$ such that, 
as $k\to
\infty$, 
\begin{subequations}
\begin{equation}
\label{e:4-cvg-t}
\sft_{\eps_k} \weaksto \sft \text{ in } W^{1,\infty}(0,S), \qquad \sft_{\eps_k} \to \sft  \text{ in } \rmC^0([0,S])\,.
\end{equation}
Analogously, 
by \eqref{aprio-bounds-1}
 there exists $\sfz\in W^{1,\infty} (0,S; \Mb(\Omega))$
  such that (possibly along a further subsequence)
\[
\sfz_{\eps_k} \weaksto \sfz  \qquad \text{ in } W^{1,\infty} (0,S;\Mb(\Omega))\,.
\]
By compactness results \`a la Ascoli-Arzel\`a in metric spaces \cite[Prop.\ 3.3.1]{AGS08}, we have
$\sfz_{\eps_k} \to \sfz $ in $ \mathrm{C}^0 ([0,S];\Hs(\Omega)_{\mathrm{weak}}) $ and thus
\begin{equation}
\label{e:4-cvg-z}
 \sfz_{\eps_k}(s) \weakto \sfz(s) \text{ in } \Hs(\Omega) \text{ for all } s \in [0,S].
\end{equation}
Exploiting \eqref{e:4-cvg-z}, it is immediate to conclude from the normalization condition \eqref{2906191307} that, additionally,  $\sfz \colon [0,S] \to L^1(\Omega)$ is 
$1$-Lipschitz continuous.
\par
Likewise, estimate \eqref{aprio-bounds-1} 
also yields that there exists $\sfp $ in $W^{1,\infty} (0,S;\MbD)$ such that 
\begin{equation}
\label{e:4-cvg-p}
\begin{aligned}
&
\begin{cases}
\sfp_{\eps_k} \weaksto \sfp   \text{ in } W^{1,\infty} (0,S;\MbD),
\\
\sfp_{\eps_k} \to \sfp   \text{ in }\mathrm{C}^0 ([0,S];\MbD_{\mathrm{weak}^*}),
\end{cases}
\\
&
 \text{ and thus } \sfp_{\eps_k}(s) \weaksto \sfp(s)  \text{ in $\MbD$ for every $s\in [0,S]$.}
 \end{aligned}
\end{equation}
Again, we  have
that $\sfp \colon [0,S]\to \MbD$ is $1$-Lipschitz continuous.
\par
By estimate \eqref{aprio-bounds-2}, 
there exist  
$\sfu \in L^\infty(0,S;\BD(\Omega))$ and $\sfe \in L^\infty(0,S;\Lnn)$ 
such that 
$\sfu_{\eps_k} \weaksto \sfu \in L^\infty(0,S;\BD(\Omega))$ and 
$ \sfe_{\eps_k} \weaksto \sfe$ in $ L^\infty(0,S;\Lnn)$.  It is immediate to check that $\sfe = \sig{\sfu {+} w {\circ} \sft} - \sfp$. 
Now, by repeating the very same arguments as in the proof of \cite[Thm.\ 4.6]{DalDesSol11} we may show that the above convergence improves to a pointwise one, and in fact we even have
\begin{equation}
\label{e:4-cvg-e}
\begin{cases}
\sfe_{\eps_k}(s_k) \weakto \sfe(s) \text{ in } \Lnn
\\
\sfu_{\eps_k}(s_k) \weaksto \sfu(s) \text{ in } \BD(\Omega)  
\end{cases}
  \text{ for all } s \in [0,S] \text{ and } (s_k)_k \text{ s.t. } s_k \to s.
\end{equation}
Without 
going into details, we may mention here that the proof of \eqref{e:4-cvg-e} 
is based on the fact that all limit points of the sequences $(\sfe_{\eps_k}(s_k))_k$ and $(\sfu_{\eps_k}(s_k))_k$ fulfill the elastic equilibrium equation, which  in turn
has a unique solution. Hence, the whole sequences  $(\sfe_{\eps_k}(s_k))_k$ and $(\sfu_{\eps_k}(s_k))_k$ converge. 
\par
Lastly, exploiting \eqref{e:4-cvg-e} and arguing as in the proof of \cite[Thm.\ 4.6]{DalDesSol11} we may obtain the 
continuity properties \eqref{additional-regularity} for $\sfe$ and $\sfu$ and the validity of the elastic equilibrium equation 
\eqref{EqEqEvery}
everywhere on $[0,S]$. 
We have thus proven \eqref{pointwise-cvg}.
\end{subequations}
\medskip

\par
It remains to show that the parameterized curve  $(\sft,\sfq) $ is  admissible in the sense of Definition \ref{def:admparcur}.  This follows
from our next result,   which partly proves the statement in
\eqref{additional-regularity-instability}. The other regularity properties therein, i.e.\
$ \sfu \in  \rmC^0(A^\circ;  H_\Dir^1(\Omega;\R^n)) $ and
$
\sfe \in \rmC^0(A^\circ;  \Lnn)$, 
will be proved in the forthcoming  Lemma \ref{le:0404231757}. 
\begin{lemma}
\label{l:Lip-est-zp}
In the open set $A^\circ$ we have 
\begin{equation}
\label{extra-lip}(\sfz,\sfp) \in W_{\mathrm{loc}}^{1,\infty}(A^\circ; L^2(\Omega){{\tim}}L^2(\Omega;\MD))
\end{equation}
 and, up to a further subsequence, 
\begin{equation}
\label{extra-cvg}
\forall\, [\alpha,\beta] \subset A^\circ \text{ we have } \quad \begin{cases}
\sfz_{\eps_k} \weaksto \sfz \text{ in } W^{1,\infty} ([\alpha,\beta];L^2(\Omega)),
\\
\sfp_{\eps_k} \weaksto \sfp \text{ in } W^{1,\infty} ([\alpha,\beta];L^2(\Omega;\MD)).
\end{cases}
\end{equation}
\end{lemma}
\begin{proof}
As we have just seen, the function
$[0,S] \ni s \mapsto \sfq(s)$ is weakly$^*$ continuous in $\Qpp$. Therefore,  thanks to Lemma \ref{l:lsc-slopes} we have that
 the function $[0,S] \ni s \mapsto \calD^*(\sft(s),\sfq(s))$ is lower semicontinuous. Thus, $A^\circ$ is open 
and 
 \[
 \forall\, [\alpha,\beta] \subset A^\circ   \ \exists\, c>0 \ \forall\, s \in  [\alpha,\beta] \, : \qquad   \calD^*(\sft(s),\sfq(s)) \geq c>0\,.
 \]
Furthermore, combining the uniform convergences from \eqref{e:4-cvg-t}--\eqref{e:4-cvg-e} with the lower semicontinuity properties from Lemma \ref{l:lsc-slopes} we conclude that 
 \[
 \exists\, \bar{k} \in \N  \ \forall\, k \geq \bar k \ \forall\, s \in  [\alpha,\beta] \, : \qquad 
 \sqrt{\congdistviscq{z} {\sft_{\eps_k}(s)}{\sfq_{\eps_k}(s)}{+} \congdistviscq{p} {\sft_{\eps_k}(s)}{\sfq_{\eps_k}(s)}} \geq c
 \]
 and, thus, from \eqref{aprio-bounds-1/2} we gather that 
\[
\exists\, C>0 \  \exists\, \bar{k} \in \N  \ \forall\, k \geq \bar k \ \foraa\, s \in  (\alpha,\beta)\, : \qquad \sqrt{\|\sfz_{\eps_k}'(s)\|_{L^2}^2{+}\|\sfp_{\eps_k}'(s)\|_{L^2}^2} 
\leq C\,.
 \]
This entails that $ \sfz \in  W^{1,\infty} ([\alpha,\beta];L^2(\Omega))$ and $\sfp \in  W^{1,\infty} ([\alpha,\beta];L^2(\Omega;\MD))$ for every $[\alpha,\beta] \subset A^\circ$ and,
 up to a further subsequence (possibly depending on $[\alpha,\beta]$), convergences in \eqref{extra-cvg} hold. Exhausting the open set $A^\circ$ by a countable family of intervals, with a diagonal procedure we can extract a subsequence such that convergences \eqref{extra-cvg}
hold on \emph{any} $[\alpha,\beta] \subset A^\circ$.
\end{proof}
 \subsection*{Proof of the upper energy-dissipation inequality}
 We will prove the following result.
 \begin{proposition}
 \label{prop:UEDI}
 The pair $(\sft, \sfq)$ satisfies the   upper energy-dissipation inequality \eqref{UEDI}. 
Moreover, 
   there exists a constant $\mathrm{M}>0$ such that 
   \begin{subequations}
   \label{straight-fwd-estimates}
   \begin{align}
   \label{straight-1}
   &
   \sup_{s\in [0,S]}  |\Epp (\sft(s),\sfq(s)) | \leq \mathrm{M},
   \\
   \label{straight-2}
   & 
   \int_0^S \Mfunz 0{\sft(s)}{\sfq(s)}{\sft'(s)}{\sfq'(s)} \dd s \leq \mathrm{M}.
   \end{align}
   \end{subequations}
 \end{proposition}
 \begin{proof}
 We send $k\to\infty$ in the upper energy-dissipation inequality 
\[
 \begin{aligned}
 &
 \calEv(\sft_{\eps_k}(s) ,\sfq_{\eps_k}(s))+
 \int_0^s  \mathcal{M}_{\eps_k} (\sft_{\eps_k}(r),\sfq_{\eps_k}(r), \sft_{\eps_k}'(r), \sfq_{\eps_k}'(r)) \dd r
\\
& 
 \leq \calEv(\sft_{\eps_k}(0),\sfq_{\eps_k}(0)) + \int_0^s \partial_t \calEv(\sft_{\eps_k}(r),\sfq_{\eps_k}(r)) \sft_{\eps_k}'(r) \dd r\,,
 \end{aligned}
\]
for any fixed $s\in [0,S]$. 
Relying on  \eqref{e:4-cvg-t}--\eqref{e:4-cvg-e} we immediately have 
\[
 \Epp (\sft(s),\sfq(s)) \leq \liminf_{n\to\infty} \calEv(\sft_{\eps_k}(s) ,\sfq_{\eps_k}(s))  \qquad \text{for all } s \in (0,S]
 \]
 while $\calEv(\sft_{\eps_k}(0) ,\sfq_{\eps_k}(0)) = \calEv(0,q(0)) = \Epp(0,q_0)$ for every $k\in \N$.  
 Taking into account the expressions of $ \partial_t \calEv(\sft_{\eps_k},\sfq_{\eps_k}) $ and $\partial_t \Epp(\sft,\sfq)$ from \eqref{partial-deriv-t}, we conclude 
 that 
 $
  \partial_t \calEv(\sft_{\eps_k},\sfq_{\eps_k}) \longrightarrow \partial_t \Epp(\sf,\sfq)   $  in  $ L^2(0,S). 
 $
 Hence, 
 \[
 \lim_{k\to\infty} \int_0^s \partial_t \calEv(\sft_{\eps_k}(r),\sfq_{\eps_k}(r)) \sft_{\eps_k}'(r) \dd r = \int_0^s \partial_t \Epp(\sft(r),\sfq(r)) \sft'(r) \dd r\,.
 \]
 Next, we observe that 
 \[
 \begin{aligned}
 \lim_{k\to\infty} \int_0^s  \calR(\sfz_{\eps_k}'(r)) \dd r  & = \kappa  \liminf_{k\to\infty}  \| \sfz_{\eps_k}(s) - \sfz_{\eps_k}(0)\|_{L^1(\Omega)} 
 \\
  & =  \kappa  \| \sfz(s) - \sfz(0)\|_{L^1(\Omega)} 
=  \int_0^s  \calR(\sfz'(r)) \dd r\,,
\end{aligned}
 \]
 while 
 \[
 \int_0^s  \Hpp(\sfz(r),\sfp'(r)) \dd r 
 \leq \liminf_{k\to \infty} \int_0^s \calH(\sfz_{\eps_k}(r), \sfp_{\eps_k}'(r)) \dd r  
  \]
by an argument based on the Reshetnyak Theorem, cf.\ the proof of \cite[Lemma 3.5]{Crismale-Lazzaroni}.  
It follows from Lemma  \ref{l:lsc-slopes}
that 
\[
I_{\{0\}}(\slope u{\sft(s)}{\sfq(s)}) \leq \liminf_{k\to\infty} I_{\{0\}} (  \|{-} \rmD_u \calE(\sft_{\eps_k}(s),\sfq_{\eps_k}(s))\|_{(H_{\Dir}^1)^*}) 
\qquad \text{for all } s \in [0,S]\,.
\]
Thus, it only remains to show that 
\begin{equation*}
 \int_0^s  \mathcal{M}_{0}^{\mathrm{red}} (\sft(r),\sfq(r), \sft'(r), \sfq'(r)) \dd r\leq \liminf_{k\to\infty}  \int_0^s  \mathcal{M}_{\eps_k}^{\mathrm{red}} (\sft_{\eps_k}(r),\sfq_{\eps_k}(r), \sft_{\eps_k}'(r), \sfq_{\eps_k}'(r)) \dd r\,.
\end{equation*}
For this,   we may argue in the same way as in the proof of  \cite[Prop.\ 7.1]{MiRo23} (cf.\ also  the lower semicontinuity arguments from \cite[Sec.\ 7]{Crismale-Rossi19} and \cite[Sec.\ 5]{CLR}). Namely, we rely on the fact that,  as $k\to\infty$,
 the functionals $  (\mathcal{M}_{\eps_k}^{\mathrm{red}})_{\eps_k} $
$\Gamma$-converge (in suitable topologies that we do not specify to avoid overburdening the exposition) to $  \mathcal{M}_{0}^{\mathrm{red}}$, cf.\  \cite[Prop.\ 5.2]{MiRo23} for an abstract result in this connection. Then,  the integral inequality follows from a version of the Ioffe Theorem  \cite{Ioff77LSIF} (cf.\ also \cite[Thm.\ 21]{Valadier90}) that has been obtained in \cite[Prop.\ 5.13]{MiRo23}. All in all, we conclude.  
   \end{proof}
    From  \eqref{straight-fwd-estimates} we derive an additional estimate, \eqref{0404231831} below,  which will play a crucial role in the proof of Lemma \ref{l:tech-approx-energies} ahead. In fact,  \eqref{0404231831}  will be derived with arguments similar to those for 
    \cite[Lemmas~7.1, 7.2, 7.3]{DalDesSol11}, albeit adapted to handle the present case in which the dependence of $\bbC$ on $z$ brings about additional difficulties. In particular, we mention that in the proof of Lemma \ref{le:0404231757} we will resort to a consequence of the  integration by parts formula \eqref{integr-by-parts-PP}, which will be also exploited in the proof of Proposition \ref{p:cornerstone} later on. As an immediate consequence of  \eqref{0404231831}  we will deduce the enhanced regularity for $\sfu$ and $\sfe$ in the instability set $A^\circ$ stated in \eqref{additional-regularity-instability}.    
\begin{lemma}\label{le:0404231757}
There exists a constant $C_L>0$,   only depending  on $\gamma_1$ in \eqref{C2}, on the Lipschitz constant
$\|\bbC\|_{\mathrm{Lip}}$
 of $\C$, and on $\mathrm{M}$ in \eqref{straight-fwd-estimates}, such that   for every
 connected component
 $(a,b) \subset A^\circ$  of $A^\circ$ and for all $[s_1, s_2]\subset (a,b)$ 
 there holds
 \begin{equation}\label{0404231831}
\|\sfe(s_2)-\sfe(s_1)\|_{L^2} \leq C_L (\|\sfp(s_2)-\sfp(s_1)\|_{L^2}+ \|\sfz(s_2)-\sfz(s_1)\|_{L^\infty})\,.
\end{equation}
 As a result, we have 
$\sfu \in  \rmC^0(A^\circ;  H_\Dir^1(\Omega;\R^n))$ and 
$
\sfe \in \rmC^0(A^\circ;  \Lnn)$. 
\end{lemma}



\begin{proof}
From the integration by parts formula  \eqref{integr-by-parts-PP} 
we deduce 
\begin{equation}
\label{L:4.5-calc1}
\begin{aligned}
&
 \langle \serifsigma_{\dev}(s_2){-} \serifsigma_{\dev}(s_1) , \sfp(s_2){-} \sfp(s_1) \rangle_{L^2(\Omega)}
 \\
 &  = - \langle \serifsigma(s_2){-}\serifsigma(s_1), 
 [ \EE(\sfu(s_2){+}\sfw(s_2))  {-} \sfp(s_2) ] {-}   [ \EE(\sfu(s_1){+}\sfw(s_1))  {-} \sfp(s_1)]  \rangle_{L^2(\Omega)}
 \\
 &  \quad + \langle \serifsigma(s_2){-}\serifsigma(s_1), 
\EE(\sfw(s_2)){-} \EE(\sfw(s_1)) \rangle_{L^2(\Omega)}
 \\
 & \quad + \langle {-} [\Diver\, \serifsigma(s_2){-}\Diver\,\serifsigma(s_1)],  [\sfu(s_2){+}\sfw(s_2){-} \sfw(s_2)] {-}  [\sfu(s_1){+}\sfw(s_1){-} \sfw(s_1)] \rangle_{\BD(\Omega)}
\\
& \stackrel{(1)}{=} - \langle \serifsigma(s_2){-}\serifsigma(s_1) ,\sfe(s_2){-} \sfe(s_1)   \rangle_{L^2(\Omega)} +  \langle \serifsigma(s_2){-}\serifsigma(s_1) ,\EE(\sfw(s_2)){-} \EE(\sfw(s_1))   \rangle_{L^2(\Omega)} 
\\
&
\quad 
  + \langle \sfF(s_2){-}  \sfF(s_1), [\sfu(s_2){-}\sfu(s_1)]  \rangle_{\BD(\Omega)}
  \\
&   \stackrel{(2)}{=} - \langle \serifsigma(s_2){-}\serifsigma(s_1) ,\sfe(s_2){-} \sfe(s_1)   \rangle_{L^2(\Omega)} +0+0\,.
\end{aligned}
\end{equation}
Here {\footnotesize (1)}
 follows from
using  that $\sfe = \EE(\sfu{+}\sfw)-\sfp$, and 
 testing  the elastic equilibrium equations \eqref{EqEqEvery}, evaluated at $s_1$ and $s_2$, by $\sfu(s_2){-}\sfu(s_1)$, while {\footnotesize (2)} is due to the fact that $\sft'\equiv 0$ on $[s_1,s_2]$ (recall \eqref{2906191815+1}), so  that $\EE(\sfw(s_2)) = \EE(\sfw(s_1))$ and $ \sfF(s_2) =  \sfF(s_1)$. 
In turn, we have 
 the identity
\begin{equation}
\label{L:4.5-calc2}
\begin{aligned}
 \langle \C(\sfz(s_2)) (\sfe(s_2){-}\sfe(s_1)), \sfe(s_2){-}\sfe(s_1) \rangle_{L^2(\Omega)}  &  +   \langle \big[\C(\sfz(s_1)){-}\C(\sfz(s_2))\big] \sfe(s_1), \sfe(s_2){-}\sfe(s_1) \rangle_{L^2(\Omega)}
 \\
 = 
\langle \serifsigma(s_2){-}\serifsigma(s_1), \sfe(s_2){-}\sfe(s_1)\rangle_{L^2(\Omega)} 
\end{aligned}
\end{equation}
Combining \eqref{L:4.5-calc1}\&\eqref{L:4.5-calc2} and using that $\bbC$ is (uniformly) positive definite (cf.\ \eqref{C2}) we obtain
\[
\begin{aligned}
& 
\gamma_1 \|  \sfe(s_2){-}\sfe(s_1) \|_{L^2(\Omega)}^2 \leq \Lambda_1 + \Lambda_2\,,
\\
& \text{with }  \qquad \begin{cases}
\displaystyle
\Lambda_1 = \left|   \langle \big[\C(\sfz(s_1)){-}\C(\sfz(s_2))\big] \sfe(s_1), \sfe(s_2){-}\sfe(s_1) \rangle_{L^2(\Omega)}\right|\,,
\smallskip
\\
\Lambda_2 = |  \langle \serifsigma_{\dev}(s_2){-} \serifsigma_{\dev}(s_1) , \sfp(s_2){-} \sfp(s_1) \rangle_{L^2(\Omega)}|\,.
\displaystyle 
\end{cases}
\end{aligned}
\]
Now, 
\[
\begin{aligned}
\Lambda_1 
& \leq \|   \C(\sfz(s_1)){-}\C(\sfz(s_2))  \|_{L^\infty(\Omega)} \| \sfe(s_1)\|_{L^2(\Omega)}   \|  \sfe(s_2){-}\sfe(s_1) \|_{L^2(\Omega)}
\\
&
\stackrel{(1)}\leq \|\bbC\|_{\mathrm{Lip}}\| \sfz(s_1)){-}\sfz(s_2)  \|_{L^\infty(\Omega)} \frac{\mathrm{M}}{\gamma_1}  \|  \sfe(s_2){-}\sfe(s_1) \|_{L^2(\Omega)} \,,
\end{aligned}
\]
where for {\footnotesize (1)}
we have used the Lipschitz continuity of $\bbC$ and estimate \eqref{straight-1}. In order to estimate $\Lambda_2$ we use that 
\[
\begin{aligned}
& \|  \serifsigma_{\dev}(s_2){-} \serifsigma_{\dev}(s_1) \|_{L^2(\Omega)} \\ & \leq \|  \bbC(\sfz(s_2)) \sfe(s_2){-}  \bbC(\sfz(s_1)) \sfe(s_1)\|_{L^2(\Omega)}
\\
& 
\leq  \| \bbC(\sfz(s_2))\|_{L^\infty(\Omega)}  \|  \sfe(s_2){-}\sfe(s_1) \|_{L^2(\Omega)}+ \| \sfe(s_2)\|_{L^2(\Omega)} \|   \C(\sfz(s_2)){-}\C(\sfz(s_1))  \|_{L^\infty(\Omega)}
\\
& 
\leq \big[ \|\bbC\|_{\mathrm{Lip}}\|\sfz(s_2))\|_{L^\infty(\Omega)}  +\|\bbC(0)\|_{L^\infty(\Omega)} \big]  \|  \sfe(s_2){-}\sfe(s_1) \|_{L^2(\Omega)} \\
& \quad +\frac{\mathrm{M}}{\gamma_1}
 \|\bbC\|_{\mathrm{Lip}} \|\sfz(s_2)){-}\sfz(s_1)\|_{L^\infty(\Omega)} \,.
\end{aligned}
 \]
All in all, we find that 
\[
\begin{aligned}
&
\gamma_1 \|  \sfe(s_2){-}\sfe(s_1) \|_{L^2(\Omega)}^2 
\\
& \leq \overline{C} \Big\{  \| \sfz(s_1)){-}\sfz(s_2)  \|_{L^\infty(\Omega)} \|  \sfe(s_2){-}\sfe(s_1) \|_{L^2(\Omega)}
{+} \|  \sfp(s_2){-}\sfp(s_1) \|_{L^2(\Omega)} \|  \sfe(s_2){-}\sfe(s_1) \|_{L^2(\Omega)} 
\\
&\qquad\quad  {+} \| \sfz(s_1)){-}\sfz(s_2)  \|_{L^\infty(\Omega)}\|  \sfp(s_2){-}\sfp(s_1) \|_{L^2(\Omega)}
\Big\} \,,
\end{aligned}
\]
for a constant $\overline{C}$ only depending on $\gamma_1$, $ \|\bbC\|_{\mathrm{Lip}}$, and $\mathrm{M}$. Then, \eqref{0404231831} ensues. 
 Combining \eqref{0404231831} with 
  the  regularity of $\sfp$ ensured by 
 \eqref{extra-lip}  and the strong continuity of $\sfz$ from $A^\circ$ into $L^\infty$ granted by
 \eqref{additional-regularity-instability},  we deduce  the claimed continuity of $\sfu$ and $\sfe$. 
\end{proof}

\section{Intermediate lower energy-dissipation inequalities}
\label{s:5}
 In this section we lay the ground for the proof of 
 Proposition 
\ref{prop:LEDI} ahead, stating  the validity of the  lower energy-dissipation inequality 
\eqref{LEDI} 
along the admissible parameterized curves $(\sft,\sfq) \in \mathscr{A}([0,S] ;[0,T] {{\tim}} \Qpp)$ 
arising from the limiting procedure of Sec.\ \ref{s:4}.
In order to prove it, 
we  will carefully revisit the techniques devised in 
\cite{DalDesSol11}, also borrowing some ideas from 
\cite{BabFraMor12}. 
\par
The key idea is subdividing the interval $[0,T]$ by means of  suitably chosen partitions and resorting to fine approximation results of the Bochner integral via Riemann sums.
We also mention that this idea has long been used in the context of the analysis of rate-independent systems, indeed  for proving the `lower' energy inequality and, ultimately, the energy-dissipation balance for energetic solutions/quasistatic evolutions, dating back to e.g.\ \cite{DMFT05, FraMie06ERCR},   cf.\ also \cite{MielkeRoubicek15}. 
\par
This is the overview of this section:
\begin{itemize}
\item
Section \ref{ss:5-crucial-estimate} revolves around a crucial estimate, proved in Proposition \ref{p:cornerstone}, that will be at the core of 
our proof of the  lower inequality 
\eqref{LEDI}.
\item In Section \ref{ss:5-outline} we are going to  outline the steps of the proof in detail. Indeed,  for any 
 fixed   partition  of $[0,S]$ 
 we will distinguish three families of induced sub-intervals. For each type of sub-interval
   we will obtain a  suitable discrete version of \eqref{LEDI}.
   \item
    These discrete inequalities will be proved
 in
 Sections \ref{ss:5-step1},  \ref{ss:5-step2}, and  \ref{ss:5-step3}, leading to 
  the \emph{overall}  discrete lower energy-dissipation inequality \eqref{overall-LEDI-discrete}. It is in  \eqref{overall-LEDI-discrete} that we will 
 pass to the limit, as the fineness of the partition tends to zero,   to conclude the lower energy-dissipation inequality  \eqref{LEDI}.
 \end{itemize}
  Before delving into the above tasks, let us specify that hereafter we shall suppose that, for the  parameterized curve
  $(\sft,\sfq) \in \mathscr{A}([0,S] ;[0,T] {{\tim}} \Qpp)$ here considered, 
  \begin{equation}
  \label{set-where-stuck}
  \text{the set } F: = \{ s\in (0,S)\, : \ \sft'(s) =0, \ \sfq'(s)=0 \} \text{ has empty interior.}
  \end{equation}
This can be assumed without loss of generality.  Indeed, suppose that $\sft' \equiv 0$ and $\sfq'\equiv 0$ on an interval $(s_*, s^*)$ with $0 \leq  s_*<s^*\leq S$. Then, the lower
energy-dissipation inequality
  \eqref{LEDI} trivially holds on $[s_*, s^*]$.   
  \par
   Let us also point out, for later use, that, since for the pair $(\sft, \sfq)$ the elastic equilibrium equation \eqref{EqEqEvery} holds everywhere, 
by \cite[Prop.\ 3.5]{DMDSMo06QEPL} we have 
\begin{equation}
\label{regularity-sigma}
\sigma(s) \in \Sigma(\Omega) \quad \text{for all } s\in [0,S].
\end{equation}



\subsection{\bf A crucial estimate}
\label{ss:5-crucial-estimate}
 The cornerstone of 
this section 
 is  estimate
\eqref{cornerstone} ahead. In order to state it, we  need to settle some notation. For a given interval  $[s_1,s_2] \subset [0,S]$
 we introduce the quantities that will enter in the discrete versions of the energy-dissipation inequality, i.e.\
\begin{itemize}
\item[$\vardiamond$]  the
\emph{energy/$\calH$-dissipation variation}  of  the curve 
$(\sft,\sfq): [0,S] \to [0,T]{\tim} \Qpp$ on  $[s_1,s_2]$
\[
\begin{aligned}
\HV((\sft,\sfq); [s_1,s_2]): &  =    \Hpp\left(\sfz(s_2), \frac{\sfp(s_2)-\sfp(s_1)}2\right)  +  \Hpp\left(\sfz(s_1), \frac{\sfp(s_2)-\sfp(s_1)}{2}\right) 
\\
& \qquad  + \|\sfp(s_2)-\sfp(s_1)\|_{L^2} \frac{ \congdist p{\sft(s_1)}{\sfq(s_1)}{+}  \congdist p{\sft(s_2)}{\sfq(s_2)}}{2},
\end{aligned}
\]
where the product $\|\sfp(s_2)-\sfp(s_1)\|_{L^2} \, \congdist p{\sft(s_i)}{\sfq(s_i)}$ is set by convention equal to 0 if $\congdist p{\sft(s_i)}{\sfq(s_i)}=0$ and $\sfp(s_2)-\sfp(s_1) \notin L^2(\Omega; \MD)$;
\item[$\vardiamond$]  the 
\emph{energy/$\calR$-dissipation variation} of $(\sft,\sfq)$  on $[s_1,s_2]$ 
\[
\RV((\sft,\sfq); [s_1,s_2]):=
 \mathcal{R} (\sfz(s_2){-} \sfz(s_1)) +
 \|\sfz(s_2){-}\sfz(s_1)\|_{L^2} \frac{ \congdist z{\sft(s_1)}{\sfq(s_1)}{+}  \congdist z{\sft(s_2)}{\sfq(s_2)}}{2} \,.
\]
\end{itemize}
The proposed names for $\HV$ and $\RV$ highlight that their definition features  contributions involving both the dissipation potential  and the energy. 
\par
We will also work with a version of $\RV$ augmented by the energy functional $\Phi$ from \eqref{viscous-energy},
 namely
\begin{itemize}
\item[$\vardiamond$]   the \emph{augmented energy/$\calR$-dissipation variation} of $(\sft,\sfq)$  on $[s_1,s_2]$
\[
\begin{aligned}
\ARV((\sft,\sfq); [s_1,s_2]):  & =  \Phi(\sfz(s_2)) -  \Phi(\sfz(s_1)) + \RV(\sfq; [s_1,s_2]) 
\\
& \qquad +K_W \|\sfz(s_2)-\sfz(s_1)\|_{L^1} \|\sfz(s_2)-\sfz(s_1)\|_{L^\infty}
\\ &   =  
 \Phi(\sfz(s_2)) -  \Phi(\sfz(s_1)) + \mathcal{R} (\sfz(s_2){-} \sfz(s_1))  
\\
&
\qquad  + \|\sfz(s_2)-\sfz(s_1)\|_{L^2} \frac{ \congdist z{\sft(s_1)}{\sfq(s_1)}{+}  \congdist z{\sft(s_2)}{\sfq(s_2)}}{2} 
\\
&
\qquad  + K_W \|\sfz(s_2)-\sfz(s_1)\|_{L^1} \|\sfz(s_2)-\sfz(s_1)\|_{L^\infty}
\,,
\end{aligned}
\]
\end{itemize}
where we have set $K_W := \|W''\|_{L^\infty(0,1)}$. Finally, estimate
\eqref{cornerstone}
will also feature
\begin{itemize}
\item[$\vardiamond$]     a term  approximating  the work of the external forces on $[s_1,s_2]$, i.e.
%
%
%
\[
\begin{aligned}
\WE((\sft,\sfq); [s_1,s_2]): = &
 \frac{1}{2} \langle \serifsigma(s_1) + \serifsigma(s_2), \EE(\sfw(s_2)-\sfw(s_1)) \rangle 
 \\
& \quad  -  \langle \tfrac12 \left( \sfF(s_1) +  \sfF(s_2) \right), \sfw(s_2)-\sfw(s_1) \rangle_{\BD(\Omega)}  
 \\
 & \quad- \langle ( \sfF(s_2) -  \sfF(s_1) ,\tfrac12  (\sfu(s_1)+\sfu(s_2))  \rangle_{\BD(\Omega)} 
 \\
 & \quad 
- \langle ( \sfF(s_2) -  \sfF(s_1) ,\tfrac12  (\sfw(s_1)+\sfw(s_2))  \rangle_{\BD(\Omega)} \,.
\end{aligned}
\]
 \end{itemize}
In fact,  in view of Lemmas \ref{l:approx-sigma}--\ref{l:approx-F'w} in appendix, 
$\WE((\sft,\sfq); [s_1,s_2])$ turns out to be a discrete form of  the external work 
$\int_{s_1}^{s_2} \partial_t \Epp(\sft,\sfq) \, \sft' \dd s$. 
 
 We are now in a position to state the main result of this section, where we show a discrete form of the lower energy-dissipation inequality.
Hereafter, we will use the notation 
\begin{equation}\label{0504232022}
K_{\bbC}: = \frac{\|\bbC'\|_{\mathrm{Lip}}}{\min_{z\in [m_0,1]}\bbC'(z) }\,, \qquad K_W:=\|W''\|_{L^\infty(0,1)} \,,
\end{equation}
where we have denoted by   $\|\bbC'\|_{\mathrm{Lip}} $ the Lipschitz constant of $\bbC'$
 and we have used for shorter notation the place-holder
\[
\min_{z\in [m_0,1]}\bbC'(z): = \min_{z\in [m_0,1]} \inf_{\xi \in \R^{n\tim n}}\bbC'(z)\xi {:}\xi \,; 
\]
recall that  $m_0$ is the constant from Prop.\ \ref{prop:-aprio-est}.  
\begin{proposition}
\label{p:cornerstone}
For all $0\leq s_1 < s_2 \leq S$  we have $ \ARV(\sfq; [s_1,s_2]) \geq 0$. Moreover, 
for $0\leq s_1 < s_2 \leq S$ such that 
\begin{equation}
\label{additional-cond-indices}
 \left( s_1, \, s_2 \in B^\circ \right) \quad \text{or} \quad \left( \text{the interval } [s_1,s_2] \subset A^\circ \right) \,,
\end{equation}
then 
\begin{equation}
\label{cornerstone}
\begin{aligned}
\WE((\sft,\sfq); [s_1,s_2])
+  \Epp(\sft(s_1), \sfq(s_1)) 
&  \leq   \Epp(\sft(s_2), \sfq(s_2))   + \HV(\sfq; [s_1,s_2])+ \RV(\sfq; [s_1,s_2]) 
 \\
 & \qquad 
 +    K_{\bbC}\|\sfz(s_2)-\sfz(s_1)\|_{L^\infty}  \ARV(\sfq; [s_1,s_2])
  \\
 & \qquad 
 +  K_W \|\sfz(s_2)-\sfz(s_1)\|_{L^1} \|\sfz(s_2)-\sfz(s_1)\|_{L^\infty}\,.
\end{aligned}
\end{equation}
\end{proposition}
%
\begin{proof}
Let $s_1 \leq s_2 \in [0,S]$ be fixed.
 Since $\sigma(s) \in \Sigma(\Omega)$ for all $s\in [0,S]$,   we may apply the integration by parts formula \eqref{integr-by-parts-PP}  with the following choices for the triple 
$(v,e,p)$:
\[
\begin{cases}
v=  (\sfu(s_2){+}\sfw(s_2)) {-} (\sfu(s_1){+}\sfw(s_1))\,,
\\
e = 
 \big(\EE(\sfu(s_2){+}\sfw(s_2)) {-} \sfp(s_2)\big)  {-} \big(\EE(\sfu(s_1){+}\sfw(s_1)) {-} \sfp(s_1)\big)\,,
\\
p = \sfp(s_2) - \sfp(s_1) \,.
\end{cases}
\]
This leads to 
\begin{equation}
\label{even-elsewhere-quoted}
\begin{aligned}
&
 \langle \serifsigma_{\dev}(s_i) | \sfp(s_2){-} \sfp(s_1) \rangle 
 \\
 &  = - \langle \serifsigma(s_i), 
 [ \EE(\sfu(s_2){+}\sfw(s_2) ) {-} \sfp(s_2) {-} \EE(\sfw(s_2))] {-}   [ \EE(\sfu(s_1){+}\sfw(s_1) ) {-} \sfp(s_1) {-} \EE(\sfw(s_1))]  \rangle_{L^2(\Omega)}
 \\
 & \quad + \langle {-}\Diver\, \serifsigma(s_i),  [\sfu(s_2){+}\sfw(s_2){-} \sfw(s_2)] {-}  [\sfu(s_1){+}\sfw(s_1){-} \sfw(s_1)] \rangle_{\BD(\Omega)}
\\
& \stackrel{(1)}{=} - \langle \serifsigma(s_i),\sfe(s_2){-} \sfe(s_1)   \rangle_{L^2(\Omega)} +  \langle \serifsigma(s_i),\EE(\sfw(s_2)){-} \EE(\sfw(s_1))   \rangle_{L^2(\Omega)} 
\\
&
\quad 
  + \langle \sfF(s_i), [\sfu(s_2){-}\sfu(s_1)]  \rangle_{\BD(\Omega)}, 
\end{aligned}
\end{equation}
where {\footnotesize (1)} follows from
recalling that $\sfe = \EE(\sfu{+}\sfw)-\sfp$, and from 
 testing  the elastic equilibrium equation \eqref{EqEqEvery}, evaluated at $s_i$, by $\sfu(s_2){-}\sfu(s_1)$. 
Therefore, adding the relation at $s_1$ with that at $s_2$,
 we obtain
\begin{subequations}
\begin{equation}\label{2712201745-old} 
\begin{aligned}
&
 \langle\tfrac12( \serifsigma(s_1) {+} \serifsigma(s_2)), \EE(\sfw(s_2)-\sfw(s_1)) \rangle  +  \langle \tfrac12 \left( \sfF(s_1) {+}  \sfF(s_2) \right), \sfu(s_2)-\sfu(s_1) \rangle_{\BD(\Omega)}  
\\
 & = \langle \tfrac12 \left( \serifsigma_{\dev}(s_1) {+} \serifsigma_{\dev}(s_2) \right) | \sfp(s_2)-\sfp(s_1) \rangle+ \langle \tfrac12 \left( \serifsigma(s_1){+} \serifsigma(s_2) \right), \sfe(s_2)-\sfe(s_1) \rangle_{L^2(\Omega)}\,.
\end{aligned}
\end{equation}
 Now, we use that for $\sfvar \in \{ \sfw, \, \sfu\}$
 \[
 \begin{aligned}
  \langle \tfrac12 \left( \sfF(s_1) {+}  \sfF(s_2) \right), \sfvar(s_2)-\sfvar(s_1) \rangle_{\BD(\Omega)} 
       = &  -   \langle   \sfF(s_2) -  \sfF(s_1) ,\tfrac12 ( \sfvar(s_1){+}\sfvar(s_2) )\rangle_{\BD(\Omega)} 
  \\
   &  + \langle \sfF(s_2), \sfvar(s_2) \rangle_{\BD(\Omega)} - \langle \sfF(s_1), \sfvar(s_1) \rangle_{\BD(\Omega)}     \,,
   \end{aligned}
 \]
 to replace the second term on the right-hand side of \eqref{2712201745-old} by 
\[
   -   \langle   \sfF(s_2) {-}  \sfF(s_1) ,\tfrac12 ( \sfu(s_1){+}\sfu(s_2) )\rangle_{\BD(\Omega)} +  \langle \sfF(s_2), \sfu(s_2) \rangle_{\BD(\Omega)} - \langle \sfF(s_1), \sfu(s_1) \rangle_{\BD(\Omega)} 
   \]
   Moreover, we 
   add to   \eqref{2712201745-old}  the identity
 \[
 \begin{aligned}
& -  \langle \tfrac12 \left( \sfF(s_1) {+}  \sfF(s_2) \right), \sfw(s_2){-}\sfw(s_1) \rangle_{\BD(\Omega)}      -   \langle   \sfF(s_2) {-}  \sfF(s_1) ,\tfrac12 ( \sfw(s_1){+}\sfw(s_2) )\rangle_{\BD(\Omega)} 
  \\
   &  \quad + \langle \sfF(s_2), \sfw(s_2) \rangle_{\BD(\Omega)} - \langle \sfF(s_1), \sfw(s_1) \rangle_{\BD(\Omega)} =0    \,.
   \end{aligned}
 \]
Ultimately, we obtain
\begin{equation}\label{2712201745} 
\begin{aligned}
&
 \WE((\sft,\sfq);[s_1,s_2]) + \langle \sfF(s_2), \sfu(s_2){+} \sfw(s_2) \rangle_{\BD(\Omega)}  -  \langle \sfF(s_1), \sfu(s_1){+} \sfw(s_1) \rangle_{\BD(\Omega)} 
 \\
 &
= \langle \tfrac12 \left( \serifsigma_{\dev}(s_1) {+} \serifsigma_{\dev}(s_2) \right) | \sfp(s_2)-\sfp(s_1) \rangle+ \langle \tfrac12 \left( \serifsigma(s_1){+} \serifsigma(s_2) \right), \sfe(s_2)-\sfe(s_1) \rangle_{L^2(\Omega)}\,.
\end{aligned}
\end{equation}
\end{subequations}
We now estimate from above  two   terms on the right-hand side of \eqref{2712201745}.
\medskip
\par
\noindent 
\emph{{\bf Claim $1$:} For every $s_1 \leq s_2 \in [0,S]$ there holds}
\begin{equation}
\label{claim-2}
\begin{aligned}
&
 \langle \tfrac12(\serifsigma(s_1) {+} \serifsigma(s_2)), \sfe(s_2)-\sfe(s_1) \rangle_{L^2} \\ & 
\leq \Q(\sfz(s_2), \sfe(s_2)) - \Q(\sfz(s_1), \sfe(s_1)) +  \left(1{+} K_{\bbC}\|\sfz(s_2)-\sfz(s_1)\|_{L^\infty}\right)  \ARV(\sfq; [s_1,s_2])\,.
\end{aligned}
\end{equation}
To show this, we start by observing that 
\begin{equation}\label{2712201749}
\begin{split}
 \langle\tfrac12( \serifsigma(s_1) {+} \serifsigma(s_2)), \sfe(s_2)-\sfe(s_1) \rangle_{L^2} &= \Q(\sfz(s_2), \sfe(s_2)) - \Q(\sfz(s_1), \sfe(s_1)) \\&\hspace{1em}+ \frac12 \langle \big[ \bbC(\sfz(s_1)) {-} \bbC(\sfz(s_2))  \big] \sfe(s_1), \sfe(s_2) \rangle_{L^2}\,.
\end{split}
\end{equation}
Using that $\bbC(\sfz(s_1)) - \bbC(\sfz(s_2))$ is a positive definite fourth-order tensor,  cf.\ \eqref{C3}, 
we find that
\begin{equation}\label{2712201754}
\begin{split}
\frac12\langle & \big[ \bbC(\sfz(s_1)) {-} \bbC(\sfz(s_2))  \big] \sfe(s_1), \sfe(s_2) \rangle_{L^2}  
\\& \leq\frac14 \langle \big[ \bbC(\sfz(s_1)) {-} \bbC(\sfz(s_2))  \big] \sfe(s_1), \sfe(s_1) \rangle_{L^2}   +\frac14 \langle \big[ \bbC(\sfz(s_1)) {-} \bbC(\sfz(s_2))  \big] \sfe(s_2), \sfe(s_2) \rangle_{L^2}    \,.
\end{split}
\end{equation}
We now estimate the terms $ \big[ \bbC(\sfz(s_1)) - \bbC(\sfz(s_2))  \big] \sfe(s_i) \colon \sfe(s_i)$, for $i=1,\,2$. By the Lagrange Theorem,
for $i \in \{1,2\}$
 there exist functions $\zeta_i\colon \Omega\to [0,1]$ (that we may assume measurable) with $ \sfz(s_2) \leq \zeta_i  \leq \sfz(s_1)$
 a.e.\ in $\Omega$, such that  for almost all $x\in \Omega$
\[
\big[ \bbC(\sfz(s_1,x)) - \bbC(\sfz(s_2,x))  \big] \sfe(s_i,x) \colon \sfe(s_i,x)= \bbC'(\zeta_i(x)) \big(\sfz(s_1,x)-\sfz(s_2,x)\big)  \sfe(s_i,x) \colon \sfe(s_i,x) \,. 
\]
Now, we have 
\[
\begin{aligned}
&
\bbC'(\zeta_i) \big(\sfz(s_1)-\sfz(s_2)\big)  \sfe(s_i) \colon \sfe(s_i)  
\\
 & \leq  \left( \bbC'(\sfz(s_i)){+} \| \bbC'\|_{\mathrm{Lip}} \| \zeta_i {-}\sfz(s_i)\|_{L^\infty} \right)\big(\sfz(s_1)-\sfz(s_2)\big)  \sfe(s_i) \colon \sfe(s_i)  
 \\
 & \leq (1{+} K_{\bbC} \| \sfz(s_2)-\sfz(s_1)\|_{L^\infty}) \bbC'(\sfz(s_i))\big(\sfz(s_1)-\sfz(s_2)\big)  \sfe(s_i) \colon \sfe(s_i)  \qquad \aein\, \Omega\,,
\end{aligned}
\]
where we have  exploited the Lipschitz continuity of $\bbC'$  and    relied on the positivity of $  \big(\sfz(s_1)-\sfz(s_2)\big) $ a.e.\ in $\Omega$ 
by the unidirectionality constraint.  
%
%
All in all, we obtain
\begin{equation}
\label{further-ingredient}
\begin{aligned}
\text{r.h.s.\ of \eqref{2712201754}}  & = 
 \sum_{i=1}^2 \tfrac14 \langle \big[ \bbC(\sfz(s_1)) - \bbC(\sfz(s_2))  \big] \sfe(s_i), \sfe(s_i) \rangle_{L^2}
 \\
 &
\leq  (1{+} K_{\bbC} \| \sfz(s_2)-\sfz(s_1)\|_{L^\infty})  \sum_{i=1}^2 I_i  \,,
\\
&
\quad \text{with }
I_i=  \int_\Omega \tfrac14 \bbC'(\sfz(s_i)) \big(\sfz(s_1)-\sfz(s_2)\big)  \sfe(s_i): \sfe(s_i) \dd x\,.
 \end{aligned}
\end{equation}
\par
In order to estimate from above the terms  $I_i$  we resort to the representation formula \eqref{duality-formulae-2} for 
$\congdist z {\sft(s_i)}{\sfq(s_i)}$, $i=1,2$, which gives  (if $\sfz(s_1) \neq \sfz(s_2)$)
\[
\begin{aligned}
\congdist z {\sft(s_i)}{\sfq(s_i)} &  \geq \frac{\langle - \As \sfz(s_i) - W'(\sfz(s_i)) - \tfrac12 \bbC'(\sfz(s_i)) \sfe(s_i){:} \sfe(s_i) ,   
\sfz(s_2){-}\sfz(s_1)\rangle_{\Hs}}{\|\sfz(s_2){-}\sfz(s_1)\|_{L^2}} \\ & \qquad  - \frac{\calR(\sfz(s_2){-}\sfz(s_1))}{\|\sfz(s_2){-}\sfz(s_1)\|_{L^2}}  \,.
 \end{aligned} 
\]
 Thus,
%
 we obtain 
\[
\begin{aligned}
&
\frac14 \int_\Omega   \bbC'(\sfz(s_1)) \sfe(s_1){:} \sfe(s_1) ( \sfz(s_1){-}\sfz(s_2) ) \dd x  
 + \frac14 \int_\Omega  \bbC'(\sfz(s_2)) \sfe(s_2){:} \sfe(s_2) (  \sfz(s_1){-}\sfz(s_2) ) \dd x 
 \\
 & 
 \leq 
\frac12 \left(  \congdist z{\sft(s_1)}{\sfq(s_1)}{+}  \congdist z{\sft(s_2)}{\sfq(s_2)} \right)  \|\sfz(s_2){-}\sfz(s_1)\|_{L^2} 
+ \frac12 \ass (\sfz(s_1){+}\sfz(s_2), \sfz(s_2){-}\sfz(s_1)) \\
& \qquad + \frac12  \int_\Omega \left\{ W'(\sfz(s_1)) {+} W'(\sfz(s_2))\right\} (\sfz(s_2){-}\sfz(s_1)) \dd x 
 +  \calR( \sfz(s_2){-}\sfz(s_1))\,.
\end{aligned}
\]
Applying the Lagrange Theorem we find  a function  $\tilde\zeta_{1,2}\colon \Omega\to [0,1]$ (again assumed measurable), with $ \sfz(s_2) \leq \tilde\zeta_{1,2}  \leq \sfz(s_1)$
a.e.\ in $\Omega$, such that   $ W(\sfz(s_2))  -  W(\sfz(s_1)) = W'(\tilde\zeta_{1,2}) (\sfz(s_1)  {-}\sfz(s_2))$ a.e.\ in $\Omega$. Thus, 
\[
\begin{aligned}
& 
\frac12 \int_\Omega [ W'(\sfz(s_1)){+} W'(\sfz(s_2)) ](\sfz(s_2)  {-}\sfz(s_1))
\\
 &  =
\int_\Omega [W(\sfz(s_2))  -  W(\sfz(s_1))] \dd x +\frac12  \int_\Omega \{  W'(\sfz(s_1)){+} W'(\sfz(s_2) {-} 2  W'(\tilde\zeta_{1,2})  \}(\sfz(s_2)  {-}\sfz(s_1)) \dd x
\\
& \leq \int_\Omega [W(\sfz(s_2))  -  W(\sfz(s_1))] \dd x + \|W''\|_{L^\infty(0,1)} \int_\Omega( \sfz(s_1)  {-}\sfz(s_2))^2 \dd x 
\\
& \leq  \int_\Omega [W(\sfz(s_2))  -  W(\sfz(s_1))] \dd x + K_W \|    \sfz(s_1)  {-}\sfz(s_2)\|_{L^1}  \|    \sfz(s_1)  {-}\sfz(s_2)\|_{L^\infty}\,.
\end{aligned}
\]
Ultimately, 
we conclude that 
\begin{equation*}
\begin{aligned}
 I_1+I_2
&    = \frac14 \int_\Omega   \bbC'(\sfz(s_1)) \sfe(s_1): \sfe(s_1) ( \sfz(s_1){-}\sfz(s_2) ) \dd x 
\\ &  \quad  + \frac14 \int_\Omega  \bbC'(\sfz(s_2)) \sfe(s_2): \sfe(s_2) (  \sfz(s_1){-}\sfz(s_2) ) \dd x   
\\
& 
\leq \frac12 \ass (\sfz(s_2), \sfz(s_2))  + \int_\Omega W(\sfz(s_2)) \dd x - \frac12 \ass (\sfz(s_1), \sfz(s_1))  - \int_\Omega W(\sfz(s_1)) \dd x
\\
& \quad   + \calR( \sfz(s_2){-}\sfz(s_1))
+\frac12 \left(  \congdist z{\sft(s_1)}{\sfq(s_1)}{+}  \congdist z{\sft(s_2)}{\sfq(s_2)} \right)  \|\sfz(s_2){-}\sfz(s_1)\|_{L^2} 
\\
& \quad + K_W \|    \sfz(s_1)  {-}\sfz(s_2)\|_{L^1}  \|    \sfz(s_1)  {-}\sfz(s_2)\|_{L^\infty}
\\ & 
= \ARV((\sft,\sfq); [s_1,s_2]) \,.
\end{aligned}
\end{equation*}
Since   $I_1+I_2\geq 0$,  we deduce in particular that $\ARV((\sft,\sfq); [s_1,s_2]) \geq 0$. 
\par
Combining the previous inequality with \eqref{2712201749}, \eqref{2712201754}, and \eqref{further-ingredient}, we conclude  
\eqref{claim-2}. 
\medskip
\par
\noindent 
\emph{{\bf Claim $2$:} For every $s_1 \leq s_2 \in [0,S]$ satisfying \eqref{additional-cond-indices}
 there holds}
\begin{equation}\label{2712201926}
\langle \frac12(\serifsigma_{\dev}(s_1) {+} \serifsigma_{\dev}(s_2)) | \sfp(s_2)-\sfp(s_1) \rangle \leq
\HV((\sft,\sfq); [s_1,s_2])\,.
\end{equation} 
Indeed, if 
$s_1$ and $s_2$ are in $B^\circ$, then $\congdist p{\sft(s_i)}{\sfq(s_i)} =0$ for $i=1,2$, 
hence 
$\serifsigma_{\dev}(s_i) \in \mathcal{K}_{\sfz(s_i)}(\Omega)$
 and, a fortiori, by \eqref{regularity-sigma} we have $\serifsigma(s_i) \in \SiKappa {\sfz(s_i)}{\Omega}$ (recall notation
\eqref{SiKappa}). 
 Therefore,  by  \eqref{eq:carH} 
 we conclude that 
\[
\frac12\langle \serifsigma_{\dev}(s_i)  | \sfp(s_2)-\sfp(s_1) \rangle  \leq \frac12 \Hpp (\sfz(s_i), \sfp(s_2)-\sfp(s_1) ) \qquad \text{for } i =1,2,
\]
hence 
\[
\begin{aligned}
&
\langle \tfrac12 (\serifsigma_{\dev}(s_1)  {+} \serifsigma_{\dev}(s_2)) | \sfp(s_2)-\sfp(s_1) \rangle 
\\
& \leq  \Hpp\left(\sfz(s_2), \frac{\sfp(s_2)-\sfp(s_1)}{2}\right)  + \Hpp\left(\sfz(s_1), \frac{\sfp(s_2)-\sfp(s_1)}{2}\right)\,,
\end{aligned}
\]
 which is indeed \eqref{2712201926}, since  $\congdist p{\sft(s_i)}{\sfq(s_i)} =0$ for $i=1,2$. 
\par
Otherwise, if $[s_1,s_2] \subset A^\circ$, then $\sfp' |_{[s_1,s_2]} $ takes values  in $\Lnn$ and thus $\sfp(s_2)-\sfp(s_1) \in \Lnn$,
so that 
\[
\langle \tfrac12 (\serifsigma_{\dev}(s_1)  {+} \serifsigma_{\dev}(s_2) | \sfp(s_2)-\sfp(s_1) \rangle = \langle \tfrac12 (\serifsigma_{\dev}(s_1)  {+} \serifsigma_{\dev}(s_2) , \sfp(s_2)-\sfp(s_1) \rangle_{L^2(\Omega)}\,.
\] 
Now,
 it follows by the representation formula
\eqref{duality-formulae-1}
  for $\congdist p{\sft(s_i)}{\sfq(s_i)}$ that 
\[
\congdist p{\sft(s_i)}{\sfq(s_i)} \geq \frac1{\| \sfp(s_2)-\sfp(s_1) \|_{L^2}} \Big[ \langle \serifsigma_{\dev}(s_i), \sfp(s_2)-\sfp(s_1) \rangle -  \Hpp(\sfz(s_i),  \sfp(s_2)-\sfp(s_1))  \Big]
\]
for $i=1,2 $. Therefore, 
\[
\begin{aligned}
 & \langle \tfrac12 (\serifsigma_{\dev}(s_1)  {+} \serifsigma_{\dev}(s_2) | \sfp(s_2)-\sfp(s_1) \rangle
 \\
&  \leq
  \Hpp\left(\sfz(s_2), \frac{\sfp(s_2)-\sfp(s_1)}{2}\right) + \Hpp\left(\sfz(s_1), \frac{\sfp(s_2)-\sfp(s_1)}{2}\right) 
  \\
  & \quad 
  + \|\sfp(s_2)-\sfp(s_1)\|_{L^2} \frac{\congdist p{\sft(s_1)}{\sfq(s_1)}{+}  \congdist p{\sft(s_2)}{\sfq(s_2)}}2 \,,
\end{aligned}
\]
which is, again, \eqref{2712201926}.
\medskip
\par
\noindent
\textbf{Conclusion of the proof:} We have
\begin{equation}
\label{2112200955}
\begin{aligned}
&
 \WE((\sft,\sfq);[s_1,s_2]) + \langle \sfF(s_2), \sfu(s_2){+} \sfw(s_2) \rangle_{\BD(\Omega)}  -  \langle \sfF(s_1), \sfu(s_1){+} \sfw(s_1) \rangle_{\BD(\Omega)}  
\\
&  \leq \Q(\sfz(s_2),\sfe(s_2)) - \Q(\sfz(s_1),\sfe(s_1)) 
+ \HV(\sfq; [s_1,s_2])
\\
& \quad +   \left(1{+} K_{\bbC}\|\sfz(s_2)-\sfz(s_1)\|_{L^\infty}\right)  \ARV(\sfq; [s_1,s_2]) \,.
\end{aligned}
\end{equation}
It suffices to combine \eqref{2712201745}   \eqref{2712201926}
 and \eqref{claim-2}.  Then,  \eqref{cornerstone} follows by suitably  rearranging some terms in  \eqref{2112200955}.
 This finishes the proof. 
 \end{proof}

%
%
%
%

\subsection{\bf Outline of the proof of the lower energy-dissipation inequality}
\label{ss:5-outline}
For given $k\in \mathbb{N}$ let us consider a partition $\mathscr{P}_k = (s_k^{i})_{i=0}^{N_k}$ of the interval $[0,S]$ 
\begin{equation}\label{0809231252}
0=s_k^0 < s_k^1 < \dots < s_k^{N_{k-1}} < s_k^{N_k}=S\,,\quad\text{with}\quad  \max_{1\leq i \leq N_k} (s_k^i - s_k^{i-1}) \to 0\,.
\end{equation}
 Now, we can choose $\mathscr{P}_k $ in such a way that 
\begin{equation}
\label{Dstar-finite}
\calD^*(\sft(s_k^i), \sfq(s_k^i))<+\infty \qquad \text{for all } i \in \{0,\ldots, N_k\}.
\end{equation}
 To check this claim, we start by observing that  from \eqref{straight-fwd-estimates} we get $  \REDfunz 0{\sft(s)}{\sfq(s)}{\sft'(s)}{\sfq'(s)}  <+\infty$ for a.a.\ $s\in (0,S)$. 
Hence, taking into account the definition of  $\REDfunzname 0$, we have that 
\[
\begin{aligned}
\calD^*(\sft(s), \sfq(s)) \equiv 0 \qquad \text{ for a.a.\ } s \in (0,S) \cap \{ s\in (0,S)\, : \ \sft'(s)>0\}.
\end{aligned}
\]
Moreover, if $U$ is an open subset of $\tilde{F}:=\{s\in [0,S]\colon \calD^*(\sft(s), \sfq(s))=+\infty\}$, we claim that $\sft'=0$ and $\sfq'=0$ for a.a.\ $s\in U$, so $U=\emptyset$ by our previous assumption \eqref{set-where-stuck}.
Indeed, since $  \REDfunz 0{\sft(s)}{\sfq(s)}{\sft'(s)}{\sfq'(s)}  <+\infty$ for a.a.\ $s\in (0,S)$, it holds that $\sft'=\sfz'=\sfp'=0$ for a.a.\ $s\in U$. By standard properties of perfect plasticity (see e.g.\ \cite[Theorem~3.8]{DMDSMo06QEPL}) it  then follows that also  $\sfu'=0$, hence  $\sfq'=0$ for a.a.\ $s\in U$. Therefore $[0,S]\setminus \tilde{F}$ has empty interior, and \eqref{Dstar-finite} ensues. 
\par
For later convenience, it is important to distinguish between the indices 
corresponding to consecutive partition times in the \emph{stable set} 
$B^\circ$ 
from \eqref{B-circ}, namely
\begin{subequations}
\label{2112200937}
\begin{align}
&
I_k:=\{i \in \{1, \dots, N_k\}\colon s_k^{i-1},\,s_k^i \in B^\circ\}\,,
\intertext{and indices such that at least one of the corresponding consecutive partition times belongs to the instability set $A^\circ = [0,S]\setminus B^\circ$, i.e.}
&
 J_k:=   \{1, \dots, N_k\} \sm I_k = \{i \in \{1, \dots,  N_k \}\colon s_k^{i-1} \in A^\circ \text{ or } \ s_k^i \in A^\circ\}\,.
 \end{align}
\end{subequations}
\paragraph{\bf Distinguished sub-intervals of the partition:}
Arguing as in  \cite[Lemmas 7.7, 7.8, 8.5]{DalDesSol11} and \cite[Sec.\ 4.6]{BabFraMor12},
 we observe that for any $i\in J_k$ we either have $(s_k^{i-1},s_k^i) \subset A^\circ $, or 
$(s_k^{i-1},s_k^i)  \cap B^\circ \neq \emptyset$. We then distinguish two sub-families of indices in $J_k$:
\begin{equation}
\label{sub-families-indices}
\hat{J}_k:= \{ i \in J_k\, : \ (s_k^{i-1},s_k^i) \subset A^\circ \}, \qquad \check{J}_k: =  \{ i \in J_k\, : \ (s_k^{i-1},s_k^i)  \cap B^\circ \neq \emptyset \}\,.
\end{equation}
Next, for $i\in \check{J}_k$, let us set 
\begin{equation}\label{decomp-additional}
s_k^{i-\tfrac23}: = \inf \{ s \in B^\circ \cap  (s_k^{i-1},s_k^i)\}\,, \qquad s_k^{i-\tfrac13}: = \sup \{ s \in B^\circ \cap  (s_k^{i-1},s_k^i)\}\,.
\end{equation}
Since $B^\circ$ is closed, we have $s_k^{i-\tfrac23},\, s_k^{i-\tfrac13} \in B^\circ$. 
Thus, we have obtained the decomposition 
\begin{equation*}
\begin{aligned}
\text{for all $i \in \check{J}_k$:}  \qquad &  (s_k^{i-1},s_k^i)  =  (s_k^{i-1},s_k^{i-\tfrac23}) \cup [s_k^{i-\tfrac23}, s_k^{i-\tfrac13}]  \cup 
(s_k^{i-\tfrac13}, s_k^i) 
\\
& 
\text{with } 
\left\{
\begin{array}{ll}
    (s_k^{i-1},s_k^{i-\tfrac23}) \subset A^\circ, & s_k^{i-\tfrac23} \in B^\circ,
  \\
   (s_k^{i-\tfrac13}, s_k^i) \subset A^\circ  & s_k^{i-\tfrac13}\in B^\circ .
\end{array}
\right.
\end{aligned}
\end{equation*}
(Clearly, we have $s_k^{i-\tfrac23}= s_k^{i-1}$ if $s_k^{i-1} \in B^\circ$; if, otherwise,   $s_k^i \in B^\circ$, we have $s_k^{i-\tfrac13}=s_k^i$.)
\par
We now consider a refined partition, consisting of $\mathscr{P}_k$ and of the nodes from \eqref{decomp-additional}. With  slight abuse of notation, let us call again $\mathscr{P}_k = (s_k^{i})_{i=0}^{N_k}$ the resulting partition. 
All in all,
it is meaningful to distinguish three sets of indices:
\begin{subequations}
\label{DIN}
\begin{enumerate}
\item the set of  indices 
corresponding to consecutive partition times in the stable set 
(including the nodes originally associated with indices in $I_k$, as well as the nodes from \eqref{decomp-additional}), 
\begin{equation}
\label{DIN-1}
\DIN k1: = \{i \in \{1, \dots, N_k\}\colon s_k^{i-1},\,s_k^i \in B^\circ\};
\end{equation}
\item the set of indices 
corresponding to consecutive partition times in the instability set 
\begin{equation}
\label{DIN-2}
\DIN k2: = \{i \in \{1, \dots, N_k\}\colon s_k^{i-1},\,s_k^i \in A^\circ\};
\end{equation}
\item the set of all the other indices
\begin{equation}
\label{DIN-3}
\DIN k3: =  \{1, \dots, N_k\} \setminus \{ \DIN k1{\cup} \DIN k2 \}\,. 
\end{equation}
\end{enumerate}
Observe that, by the previous discussion we may suppose that for $i\in \DIN k2 $ the enclosed  interval is `fully unstable', i.e. $[ s_k^{i-1},s_k^i] \subset A^\circ$;
similarly, for $i\in \DIN k3$ we have at least $(s_k^{i-1},s_k^i) \subset A^\circ$. 
Let us mention in advance that the distinction between indices in $\DIN k2$ and indices in $\DIN k3$ is motivated by the fact that 
the cornerstone estimate \eqref{cornerstone}  only  holds under conditions \eqref{additional-cond-indices}. Therefore, it will be \emph{directly} applicable
either on intervals $[s_k^{i-1},s_k^i]$ with $i\in \DIN k1$, or on intervals 
 $[s_k^{i-1},s_k^i]$ with $i\in \DIN k2$. For the  intervals $[s_k^{i-1},s_k^i]$ with $i\in \DIN k3$ we will need to argue by approximation.
\end{subequations}
\medskip

\paragraph{\bf Outline of the of the proof of the lower inequality  \eqref{LEDI}}
We are now in a position to specify the steps in our proof of   \eqref{LEDI}. Namely, 
\begin{description}
\item[\textbf{[Step $1$]}]  First of all, in Lemma \ref{l:DK1} ahead we shall deduce from estimate \eqref{cornerstone} a 
discrete version of the  lower  energy-dissipation inequality on  intervals whose endpoints are in $\DIN k1$, cf.\ \eqref{cornerstone-DIN1}  below.
\item[\textbf{[Step $2$]}] Secondly, in Sec.\ \ref{ss:5-step2} we will proceed to handle the fully unstable intervals having end-points $s_k^{i-1},\, s_k^i $ with $ i \in \DIN k2$, so that 
$[s_k^{i-1},s_k^i ] \subset A^\circ$. In that case we will derive from estimate \eqref{cornerstone}   a discrete version of the lower energy-dissipation inequality,
cf.\ Lemma  \ref{l:DK2}, by resorting to  suitable sub-partitions.
\item[\textbf{[Step $3$]}] In Sec.\  \ref{ss:5-step3}  we  will address intervals  $[s_k^{i-1},s_k^i]$ with $i\in \DIN k3$,  for which  in principle
we only have the inclusion   $(s_k^{i-1},s_k^i) \subset A^\circ$. In this case, a discrete version of the lower energy-dissipation inequality 
will be proved in  Lemma  \ref{l:DK3} by combining sub-partitions with a suitable approximation argument.
\item[\textbf{[Step $4$]}]  Eventually,  in Section  \ref{s:6} we will combine the inequalities from 
Lemmas \ref{l:DK1}, \ref{l:DK2}, and \ref{l:DK3}, to conclude an overall discrete energy-dissipation inequality for the partition   
$\mathscr{P}_k = (s_k^{j})_{j=0}^{N_k}$. Therein, we will pass to the limit as $k\to\infty$ and finally prove the lower energy-dissipation inequality   \eqref{LEDI}. 
\end{description}
\begin{remark}
\label{rmk:differencetoBFM}
\upshape
The  present approach for the lower energy-dissipation inequality has been inspired by the corresponding analysis in \cite{BabFraMor12}
 which, in turn, borrowed ideas from \cite{DalDesSol11}. 
  However,  there are some remarkable differences. 
The first sequence of partitions in \cite{BabFraMor12}, corresponding to that in \eqref{0809231252}, besides having vanishing fineness has to satisfy further conditions. Moreover, for intervals corresponding to indices in $\DIN k2$ and $\DIN k3$, 
the argument in 
\cite{BabFraMor12} exploits the fact that a chain rule is available in $A^\circ$, so the desired estimates are directly  obtained  by an integration in time. This is essentially due to the
fact that in  the system from \cite{BabFraMor12}  plasticity is not coupled with damage. Thus,  the further regularity for 
 the plastic strain  in $A^\circ$ (absolutely continuous with values in $L^2(\Omega;\Mnn)$) implies that all the variables are absolutely continuous, in $A^\circ$, in their target spaces.

Now the presence of damage, which is absolutely continuous with values in $L^1(\Omega)$ (while its target space is $\Hs(\Omega)$), prevents us  from proving absolute continuity  for the whole evolution. Notice also that the enhanced estimate for the damage variable obtained for the viscous approximations as in \cite[Proposition~4.4]{Crismale-Rossi19} cannot be repeated along jump intervals.

 That is why,   we need to   resort to a discrete approximation even in the `unstable set'. This  refinement of the  analysis, in turn,  allows us to consider a generic choice of the initial sequence of partitions, as in \eqref{0809231252}.
\end{remark}

\subsection{\bf Step $1$}
\label{ss:5-step1}
As an immediate corollary of Proposition \ref{p:cornerstone} we have the following result.
\begin{lemma}
\label{l:DK1}
For every $i \in \DIN k1$ we have 
\begin{equation}
\label{cornerstone-DIN1}
\begin{aligned}
&
\WE((\sft,\sfq); [s_k^{i-1},s_k^i])
+  \Epp(\sft(s_k^{i-1}), \sfq(s_k^{i-1})) 
\\
 & \leq   \Epp(\sft(s_k^i), \sfq(s_k^i))  +   \Hpp\left(\sfz(s_k^{i-1}), \frac{\sfp(s_k^i){-}\sfp(s_k^{i-1})}2\right)  + \Hpp\left(\sfz(s_k^i), \frac{\sfp(s_k^{i}){-}\sfp(s_k^{i-1})}{2}\right) 
 \\
 & \quad 
 +
 \calR(\sfz(s_k^i){-}\sfz(s_k^{i-1}))  +
 \mathrm{Rem}_1([s_k^{i-1},s_k^i])
  \end{aligned}
 \end{equation}
 where 
 the remainder term is given by 
 \begin{equation}
 \label{Rem1}
 \mathrm{Rem}_1( [s_k^{i-1},s_k^i] ) = \Delta_1(s_k^{i-1},s_k^i)\, \|\sfz(s_k^i){-}\sfz(s_k^{i-1})\|_{L^\infty}
 \end{equation}
and  we have used the place-holder
 \begin{equation}
 \label{Delta1}
 \begin{aligned}
 \Delta_1(s_k^{i-1},s_k^i) & := 
    K_{\bbC} \left(  \Phi(\sfz(s_k^i)) -  \Phi(\sfz(s_k^{i-1})) +  \calR(\sfz(s_k^i){-}\sfz(s_k^{i-1})) \right)
  \\ & \quad +K_W \|\sfz(s_k^i){-}\sfz(s_k^{i-1})\|_{L^1} \left( 1{+}
 K_{\bbC}  \|\sfz(s_k^i){-}\sfz(s_k^{i-1})\|_{L^\infty}  \right)  \,.
 \end{aligned}
 \end{equation}
\end{lemma}
\begin{proof}
It suffices to apply Proposition \ref{p:cornerstone} with the choices $s_1=s_k^{i-1}$, $s_2=s_k^{i}$, observing that 
\[
\HV((\sft,\sfq); [s_k^{i-1},s_k^i]) =  \Hpp\left(\sfz(s_k^{i-1}), \frac{\sfp(s_k^i){-}\sfp(s_k^{i-1})}2\right)  + \Hpp\left(\sfz(s_k^i), \frac{\sfp(s_k^i){-}\sfp(s_k^{i-1})}{2}\right) 
\]
since $\congdist p{\sft(s_k^{i-1})}{\sfq(s_k^{i-1})} = \congdist p{\sft(s_k^{i})}{\sfq(s_k^{i})}=0$, 
and that, analogously,
\[
\RV((\sft,\sfq); [s_k^{i-1},s_k^i]) = \calR(\sfz(s_k^i){-}\sfz(s_k^{i-1}))
\]
and 
\[
\begin{aligned}
\ARV((\sft,\sfq); [s_k^{i-1},s_k^i])  &  = \Phi(\sfz(s_k^i)) -  \Phi(\sfz(s_k^{-1})) +  \calR(\sfz(s_k^i){-}\sfz(s_k^{i-1})) 
\\
& \qquad 
 + K_W \|\sfz(s_k^i){-}\sfz(s_k^{i-1})\|_{L^1} \|\sfz(s_k^i){-}\sfz(s_k^{i-1})\|_{L^\infty}\,.
 \end{aligned}
\]
\end{proof}
%

\subsection{\bf Step $2$}
\label{ss:5-step2}
 With the main result of this section, Lemma \ref{l:DK2}, we obtain an energy-dissipation inequality for \emph{fully unstable} intervals
by resorting to suitable sub-partitions. 
The usage of sub-partitions relies on the following result, whose proof  is postponed to Appendix \ref{ss:proofLemVito}. 
\begin{lemma}\label{le:mediesci}
Let $\psi \colon [a,b] \to (0,+\infty]$ be a lower semicontinuous function such that 
$\psi(a)$, $\psi(b) \in \R$. Then for every $\eta>0$ there exists a partition $(\subd_\eta^i)_{i=0}^{N_\eta}$ of $[a,b]$
 $a=\subd_\eta^0<\subd_\eta^1< \dots < \subd_\eta^{N_\eta-1}<\subd_\eta^{N_\eta} = b$ such that
\begin{equation}\label{2012201629}
\psi(s) \geq \frac{1}{2}\Big(\psi(\subd_\eta^{j-1}) + \psi(\subd_\eta^j) \Big) - \eta \quad \text{for all } s\in (\subd_\eta^{j-1}, \subd_\eta^j) \quad \text{and all } j \in \{1, \dots, N_\eta\}\,.
\end{equation} 
\end{lemma}
Based on this result, we obtain the following estimate on  intervals contained in the instability set $A^\circ$. 
\begin{lemma}
\label{l:estimate-fully-unstable}
Let $[s_{\#},s^{\#}]\subset A^\circ$
such that $\calD^*(s_{\#})\in \R$ and $\calD^*(s^{\#})\in \R$.
 Then, for every $\eta>0$ there exists a partition
$(\subd_\eta^i)_{i=0}^{N_\eta}$ of $[s_{\#},s^{\#}]$ such that 
\begin{equation}
\label{est-on-subpart}
\begin{aligned}
&
\sum_{j=1}^{N_\eta} \Big\{
\|\sfz(\subd_\eta^{j}){-}\sfz(\subd_\eta^{j-1})\|_{L^2} \frac{ \congdist z{\sft(\subd_\eta^{j-1})}{\sfq(\subd_\eta^{j-1})}{+}  \congdist z{\sft(\subd_\eta^{j})}{\sfq(\subd_\eta^{j})}}{2}
\\
& 
\qquad \qquad \qquad
{+}
 \|\sfp(\subd_\eta^{j}){-}\sfp(\subd_\eta^{j-1})\|_{L^2} \frac{ \congdist p{\sft(\subd_\eta^{j-1})}{\sfq(\subd_\eta^{j-1})}{+}  \congdist p{\sft(\subd_\eta^{j})}{\sfq(\subd_\eta^{j})}}{2}
 \Big\}
 \\
 & \leq \int_{s_{\#}}^{s^{\#}}   \mathcal{D}(\sfq'(s)) \,  \mathcal{D}^*(\sft(r),\sfq(r)) \dd r 
 \\
 & \quad + \eta\, 
   \frac1{\min_{r\in [s_{\#},s^{\#}]}  \mathcal{D}^*(\sft(r),\sfq(r))}
   \int_{s_{\#}}^{s^{\#}}  \calD(\sfq'(r)) \, \mathcal{D}^*(\sft(r),\sfq(r))  \dd r \,. 
 \end{aligned}
\end{equation}
\end{lemma}
\begin{proof}
We apply Lemma \ref{le:mediesci} 
 with the choices $[a,b]=[s_{\#}, s^{\#}] $ and
 $\psi=\calD^*(\sft(\cdot), \sfq(\cdot))|_{[s_{\#}, s^{\#}]}$, and thus for every $\eta>0$ we  obtain a partition
$(\subd_\eta^j)_{j=0}^{N_\eta}$ of $[s_{\#},s^{\#}]$ such that for all $ j \in \{1, \dots, N_\eta\}$
\begin{equation}
\label{key-ineq-for-subpart}
\calD^*(\sft(r),\sfq(r)) \geq \frac{1}{2}\Big(\calD^*(\sft(\subd_\eta^{j-1}),\sfq(\subd_\eta^{j-1}))  {+} \calD^*(\sft(\subd_\eta^{j}),\sfq(\subd_\eta^{j}))  \Big) - \eta \quad \text{for all } r\in (\subd_\eta^{j-1}, \subd_\eta^j) \,.
\end{equation}
Now, let us use the place-holders
\[
\begin{aligned}
&
\mathsf{Z}: = \|\sfz(\subd_\eta^{j}){-}\sfz(\subd_\eta^{j-1})\|_{L^2}\,,
\\
& 
\mathsf{P}:=  \|\sfp(\subd_\eta^{j}){-}\sfp(\subd_\eta^{j-1})\|_{L^2}\,,
\\
& 
\mathsf{M}_{\mathsf{Z}}: = \frac{ \congdist z{\sft(\subd_\eta^{j-1})}{\sfq(\subd_\eta^{j-1})}{+}  \congdist z{\sft(\subd_\eta^{j})}{\sfq(\subd_\eta^{j})}}{2}\,,
\\
&
\mathsf{M}_{\mathsf{P}}: =  \frac{ \congdist p{\sft(\subd_\eta^{j-1})}{\sfq(\subd_\eta^{j-1})}{+}  \congdist p{\sft(\subd_\eta^{j})}{\sfq(\subd_\eta^{j})}}{2} \,.
\end{aligned}
\]
We estimate for each $j \in \{1, \ldots, N_\eta\}$
\[
\begin{aligned}
\mathsf{Z}\cdot  \mathsf{M_Z} + \mathsf{P}\cdot \mathsf{M_P}  \leq \| (\mathsf{Z} , \mathsf{P}) \| \cdot \| (\mathsf{M_Z} , \mathsf{M_P}) \| 
\end{aligned}
\]
by the Cauchy-Schwarz inequality.  Now, 
\[
\begin{aligned}
\| (\mathsf{Z} , \mathsf{P}) \| &  = \sqrt{ \|\sfz(\subd_\eta^{j}){-}\sfz(\subd_\eta^{j-1})\|_{L^2}^2{+} \|\sfp(\subd_\eta^{j}){-}\sfp(\subd_\eta^{j-1})\|_{L^2}^2}
\\
&
= \left\| \int_{\subd_\eta^{j-1}}^{\subd_\eta^{j}} (\sfz'(r), \sfp'(r)) \dd r \right\|_{L^2}
\\
& \leq \int_{\subd_\eta^{j-1}}^{\subd_\eta^{j}} \| (\sfz'(r), \sfp'(r))\|_{L^2} \dd r  = \int_{\subd_\eta^{j-1}}^{\subd_\eta^{j}}  \calD(\sfq'(r)) \dd r \,.
\end{aligned}
\]
On the other hand,
\[
\begin{aligned}
& 
\| (\mathsf{M_Z} , \mathsf{M_P}) \|  
\\
 &  =  \frac12 \big\| ( \congdist z{\sft(\subd_\eta^{j-1})}{\sfq(\subd_\eta^{j-1})},  \congdist p{\sft(\subd_\eta^{j-1})}{\sfq(\subd_\eta^{j-1})}) {+} ( \congdist z{\sft(\subd_\eta^{j})}{\sfq(\subd_\eta^{j})},  \congdist p{\sft(\subd_\eta^{j})}{\sfq(\subd_\eta^{j})}) \big  \| 
\\
& \leq 
\frac12 \calD^*(\sft(\subd_\eta^{j-1}),\sfq(\subd_\eta^{j-1}))  +\frac12  \calD^*(\sft(\subd_\eta^{j}),\sfq(\subd_\eta^{j})) \,.
\end{aligned}
\]
Therefore, 
\[
\begin{aligned}
\mathsf{Z} \cdot \mathsf{M_Z} + \mathsf{P} \cdot \mathsf{M_P}   & \leq  \int_{\subd_\eta^{j-1}}^{\subd_\eta^{j}}  \left\{ 
\frac12 \calD^*(\sft(\subd_\eta^{j-1}),\sfq(\subd_\eta^{j-1}))  {+}\frac12  \calD^*(\sft(\subd_\eta^{j}),\sfq(\subd_\eta^{j})) \right\}
 \calD(\sfq'(r)) \dd r 
 \\
 & 
 \leq   \int_{\subd_\eta^{j-1}}^{\subd_\eta^{j}}  \calD(\sfq'(r)) \, \calD^*(\sft(r),\sfq(r))   \dd r  +\eta   \int_{\subd_\eta^{j-1}}^{\subd_\eta^{j}}  \calD(\sfq'(r)) \,   \dd r\,,
 \end{aligned}
\]
 where the last inequality follows from \eqref{key-ineq-for-subpart}.  
Finally, we observe that 
\[
\begin{aligned}
\eta   \int_{\subd_\eta^{j-1}}^{\subd_\eta^{j}}  \calD(\sfq'(r)) \,   \dd r 
& \leq \eta \frac1{\min_{r\in [s_{\#},s^{\#}]}  \mathcal{D}^*(\sft(r),\sfq(r))} \int_{\subd_\eta^{j-1}}^{\subd_\eta^{j}}  \calD(\sfq'(r)) \, \mathcal{D}^*(\sft(r),\sfq(r))   \dd r\,,
 \end{aligned}
\]
where the latter estimate ensues from the upper inequality \eqref{UEDI}. 
We thus conclude estimate \eqref{est-on-subpart} upon adding over the index $j=1, \ldots, N_\eta$. 
\end{proof}
\par
Eventually, combining  the key estimate \eqref{cornerstone} with Lemma \ref{l:estimate-fully-unstable}  we obtain the counterpart to  Lemma \ref{l:DK1}, for indices in $\DIN k2$. 
 We emphasize that, now, in the right-hand side of  \eqref{cornerstone-DIN2}  the `energy-dissipation' term $\int \mathcal{D}(\sfq')\,  \mathcal{D}^*(\sft,\sfq) \dd s$ appears,
  whose presence records the fact that $[s_k^{i-1}, s_k^i]$ is an `unstable interval'. 
\begin{lemma}
\label{l:DK2}
For every $i \in \DIN k2$ and for every $\eta>0$
there exists a partition $(\subdiv{\eta}ij)_{j=0}^{\cardi \eta i}$ of the interval $[s_k^{i-1}, s_k^i]$
such that 
\begin{equation}
\label{cornerstone-DIN2}
\begin{aligned}
&
\sum_{j=1}^{\cardi \eta i}
\WE((\sft,\sfq); [\subdiv \eta i{j-1},\subdiv \eta ij])
+  \Epp(\sft(s_k^{i-1}), \sfq(s_k^{i-1})) 
\\
 & \leq   \Epp(\sft(s_k^i), \sfq(s_k^i))   
 \\
 & \qquad + \sum_{j=1}^{\cardi \eta i} \Hpp\left(\sfz(\subdiv \eta i{j-1}), \frac{\sfp(\subdiv \eta i{j}){-}\sfp(\subdiv \eta i{j-1})}2\right)  + \Hpp\left(\sfz(\subdiv \eta i{j}), \frac{\sfp(\subdiv \eta i{j}){-}\sfp(\subdiv \eta i{j-1})}{2}\right) 
 \\
 & \qquad 
 +
 \calR(\sfz(s_k^i){-}\sfz(s_k^{i-1}))  + \int_{s_k^{i-1}}^{s_k^i} \mathcal{D}(\sfq'(s)) \,  \mathcal{D}^*(\sft(s),\sfq(s)) \dd s 
 \\
 & \qquad 
 + \eta\, \frac{\mathrm{M}}{\min_{s\in [s_k^{i-1}, s_k^i]}  \mathcal{D}^*(\sft(s),\sfq(s))} +
 \mathrm{Rem}_2([s_k^{i-1},s_k^i])
  \end{aligned}
 \end{equation}
 with the constant $\mathrm{M}>0$ from   estimates  \eqref{straight-fwd-estimates}, 
 where 
 the remainder term is now given by 
 \begin{equation}
 \label{Rem2}
 \mathrm{Rem}_2( [s_k^{i-1},s_k^i] ) =   \sum_{j=1}^{\cardi \eta i} \Delta(\subdiv \eta i {j-1},\subdiv \eta i {j})   \|\sfz(\subdiv \eta i j){-}\sfz(\subdiv \eta i {j-1})\|_{L^\infty}
 \end{equation}
 with $\Delta$ from \eqref{Delta1}. 
\end{lemma}
\begin{proof}
We apply Lemma \ref{l:estimate-fully-unstable} with $s_\# = s_k^{i-1}$ and $s^\#= s_k^i$ and obtain a partition 
 $(\subdiv{\eta}ij)_{j=0}^{\cardi \eta i}$ of the interval $[s_k^{i-1}, s_k^i]$ for which \eqref{est-on-subpart} holds.  Since $[s_k^{i-1}, s_k^i] \subset A^\circ$,
 we have $[\subdiv{\eta}i{j-1}, \subdiv{\eta}i{j}] \subset A^\circ$ for all $j \in \{ 1, \ldots, \cardi \eta i\}$. Therefore, 
 we may apply estimate \eqref{cornerstone} with $s_1= \subdiv{\eta}i{j-1}$
 and $s_2 = \subdiv{\eta}i{j}$. Summing   \eqref{cornerstone} over the index $j$, combining it with \eqref{est-on-subpart},
 and observing that 
 \[
 \begin{aligned}
 \sum_{j=1}^{\cardi \eta i}  \int_{\subd_\eta^{j-1}}^{\subd_\eta^{j}}  \calD(\sfq'(r)) \, \mathcal{D}^*(\sft(r),\sfq(r))   \dd r
 &  = \int_{s_k^{i-1}}^{s_k^i} \calD(\sfq'(r)) \, \mathcal{D}^*(\sft(r),\sfq(r))   \dd r
 \\
 & 
 \leq \int_{(0,S){\cap}A^\circ} \calD(\sfq'(r)) \, \mathcal{D}^*(\sft(r),\sfq(r))   \dd r \leq \mathrm{M}
 \end{aligned}
 \]
 by   estimates  \eqref{straight-fwd-estimates}, 
  we arrive  at \eqref{cornerstone-DIN2}. 
\end{proof}
\subsection{Step $3$}
\label{ss:5-step3}
Again, we aim to prove a discrete version of the energy-dissipation inequality  on the interval $[s_k^{i-1}, s_k^i]$. Since in this case
we only have, in principle, that  $(s_k^{i-1}, s_k^i)$ is a connected  component of $ A^\circ$
but, possibly, either $s_k^{i-1} \notin A^\circ$ or $s_k^{i} \notin A^\circ$, 
 we need to devise an approximation argument in order to 
reproduce the situation of   intervals contained  in $A^\circ$ \emph{with their closure}. 
We will rely on the following technical result: 
essentially, it  ensures that, if one of the end-points of a connected component of  $ A^\circ$
does not belong to $A^\circ$, it is in any case possible to approximate it 
`in energy'
with points from within 
 $ A^\circ$.
\begin{lemma}
\label{l:tech-approx-energies}
Let $(a,b) \subset A^\circ$ be a connected component of $A^\circ$.
 Then, there exist sequences $(a_n), \, (b_n)_n \subset A^\circ$ with 
$a\leq a_n < b_n \leq b$, with $a_n \downarrow a$  decreasingly  and $b_n\uparrow b$  increasingly as $n\to\infty$, such that 
\begin{subequations}
\label{approx-energies}
\begin{align}
&
\label{approx-energies-a}
\begin{cases}
\displaystyle 
\serifsigma(a_n)\to \serifsigma(a) \text{  strongly in $\Lnn$}, 
 \\
\displaystyle  \Phi(\sfz(a_n)) \to \Phi(\sfz(a)), 
\\
\displaystyle 
 \Epp(\sft(a_n), \sfq(a_n)) \to  \Epp(\sft(a), \sfq(a));
 \end{cases}
\\
&
\label{approx-energies-b}
\begin{cases}
\displaystyle 
\serifsigma(b_n)\to \serifsigma(b) \text{  strongly in $\Lnn$},  
\\
\displaystyle  \Phi(\sfz(b_n)) \to \Phi(\sfz(b)), 
 \\
\displaystyle   \Epp(\sft(b_n), \sfq(b_n)) \to  \Epp(\sft(b), \sfq(b))\,.
\end{cases}
\end{align}
\end{subequations}
\end{lemma} 
\begin{proof} 
To fix ideas, let us detail the construction of the sequence $(b_n)_n$ (analogous arguments give the construction of $(a_n)_n$). 
Clearly, if $b \in A^\circ$, then it is sufficient to take $b_n \equiv b$. Suppose thus that $b \notin A^\circ$, so that 
$\calD^*(\sft(b), \sfq(b)) =0$.

 We split the proof of \eqref{approx-energies-b}
in four claims;  the first two settle the convergence for the stresses because, indeed, we will show that for \emph{any} $(b_n)_n$ with 
$b_n \uparrow b$ there holds $ \serifsigma(b_n)\to \serifsigma(b) $ \emph{strongly} in $\Lnn$. 

\medskip
\noindent
\textbf{Claim $1$:} {\it  the properties} 
\begin{equation}
\label{strong-sigma-n}
\lim_{s\uparrow b} \serifsigma(s)= \serifsigma(b) \text{  strongly in $\Lnn$}
\end{equation}
{\it and}
 \begin{equation}\label{0404232301}
 \lim_{s\uparrow b} \congdist p{\sft(s)}{\sfq(s)}  =0
  \end{equation}
  {\it are equivalent.}\medskip\\
   In order to show that $\eqref{strong-sigma-n}$ implies \eqref{0404232301}, we first of all observe that, since $\sfz :[0,S]\to \Hs(\Omega)$ is continuous w.r.t.\ the weak topology of $\Hs(\Omega)$ by \eqref{additional-regularity}, there holds $\sfz(s) \weakto \sfz(b)$ in $\Hs(\Omega)$, and thus $\sfz(s) \to \sfz(b)$ in $L^\infty(\Omega)$. By Remark \ref{rmk:equivalent} we have that 
\begin{equation}
\label{Hausdorff-cvg}
\sup_{x\in \overline\Omega} d_{\mathscr{H}} (K(\sfz(s,x)), K(\sfz(b,x)))\, \longrightarrow 0 \text{ as } s \uparrow b\,.
\end{equation}
Therefore, 
\[
\begin{aligned}
 \congdist p{\sft(s)}{\sfq(s)}  & = \mathrm{dist}_{L^2(\Omega)}({-}\serifsigma_\dev(s),  \calK_{\sfz(s)}(\Omega))
 \\ &  \,  \longrightarrow \, \mathrm{dist}_{L^2(\Omega)}({-}\serifsigma_\dev(b),  \calK_{\sfz(b)}(\Omega)) = \congdist p{\sft(b)}{\sfq(b)} =0
 \end{aligned}
\]
as $s \uparrow b$ (recall that $\calD^*(\sft(b), \sfq(b)) =0$). 
 \par
Conversely, suppose \eqref{0404232301} and 
 let us fix an arbitrary sequence $(s_k)_k$ with $s_k \uparrow b$.  Since $\sfe : [0,S]\to \Lnn $ is weakly continuous by  \eqref{additional-regularity}, we have that 
 $\sfe(s_k) \weakto \sfe(b)$ in $\Lnn$;  we combine this  with the fact that $\sfz(s_k) \weakto \sfz(b)$ in $\Hs(\Omega)$   and conclude that 
 $  \serifsigma(s_k) = \C(\sfz(s_k)) \sfe(s_k) \weakto    \C(\sfz(b)) \sfe(b)  = \serifsigma(b) $ in $\Lnn$. 
 We
 split the proof of the convergence
 \begin{equation}
 \label{strong-sigma-k}
  \serifsigma(s_k) \to \serifsigma(b)  \text{  strongly in $\Lnn$} 
 \end{equation}
 in three steps:
 \begin{description}
 \item[\textbf{Step $1$}]
 We apply Proposition \ref{prop:regolaritaellittica} from Appendix \ref{ss:appD} with the choices
 $u_k:= \sfu(s_k) \subset \BD(\Omega)$,
 $u= \sfu(b)$, 
  $z_k := \sfz(s_k) \subset \Hs(\Omega)$
  $z= \sfz(b)$, 
  $p_k := \sfp(s_k) \subset L^2(\Omega;\MD)$, and $p = \sfp(b)$:   indeed, 
  $\sfu: [0,S]\to \BD(\Omega)$, $\sfz: [0,S]\to \Hs(\Omega)$   are weakly$^*$ continuous (cf.\  again \eqref{additional-regularity}), whereas 
 $p: [0,S]\to \MbD  $  is $1$-Lipschitz continuous. 
 Therefore, we conclude that the sequence $(\nabla \sfu(s_k))_k$ is  Cauchy w.r.t.\ convergence in measure. 
 Then, $(\sig{ \sfu(s_k)})_k$ is Cauchy w.r.t.\ convergence in measure. Now,  $(\sfp(s_k))_k \subset L^1(\Omega; \MD)$  is itself a Cauchy sequence, as 
 $\|\sfp(s_k){-} \sfp(s_h)\|_{L^1(\Omega)} = \|\sfp(s_k){-} \sfp(s_h)\|_{\mathcal{M}_{\mathrm{b}}(\Omega)} $ for all $k,h\in \N$. Thus,  the sequence  $(\sfe(s_k))_k $
is 
 Cauchy w.r.t.\ convergence in measure, and so is $\serifsigma(s_k) = \C(\sfz(s_k)) \sfe(s_k)$. 
 \item[\textbf{Step $2$}]    
 We consider the decomposition 
 \begin{equation}
 \label{decomp-dev}
\serifsigma_\dev(s_k)  = \pi_{\calK_{\sfz(s_k)}(\Omega)}(\serifsigma_\dev(s_k) ) + \serifsigma_\dev(s_k)  - \pi_{\calK_{\sfz(s_k)}(\Omega)}(\serifsigma_\dev(s_k) ) 
 \end{equation}
where $ \pi_{\calK_{\sfz(s_k)}(\Omega)}(\serifsigma_\dev(s_k) ) $ denotes the projection of $\serifsigma_\dev(s_k) $ onto $\calK_{\sfz(s_k)}(\Omega)$. It follows from 
\eqref{0404232301} that $ \serifsigma_\dev(s_k)  -  \pi_{\calK_{\sfz(s_k)}(\Omega)}(\serifsigma_\dev(s_k) ) \longrightarrow 0$ in $L^2(\Omega;\MD)$.
 Taking into account that $( \serifsigma_\dev(s_k) )_k$ is a Cauchy sequence w.r.t.\  convergence in measure  by Step $1$, 
 from 
\eqref{decomp-dev} we conclude that $( \pi_{\calK_{\sfz(s_k)}(\Omega)}(\serifsigma_\dev(s_k) ))_k$ is also  a Cauchy sequence w.r.t.\ convergence in measure. 
Up to a not relabeled subsequence,  $(\pi_{\calK_{\sfz(s_k)}(\Omega)}(\serifsigma_\dev(s_k) ))_k$  converges a.e. in $\Omega$ to some limit function $\bar\pi$. Using \eqref{Hausdorff-cvg}, it is not difficult to identify $\overline \pi$ as $ \pi_{\calK_{\sfz(b)}(\Omega)}(\serifsigma_\dev(b) )$. 
 In turn, 
from \eqref{propsH} (cf.\ also Remark \ref{rmk:equivalent}) we gather that $ \pi_{\calK_{\sfz(s_k)}(\Omega)}(\serifsigma_\dev(s_k) ) \in B_{\overline{R}}(0)$ for all $k \in \N$, hence 
the sequence $ (\pi_{\calK_{\sfz(s_k)}(\Omega)}(\serifsigma_\dev(s_k) ))_{k}$ is bounded in $L^\infty(\Omega;\MD)$. 
Therefore, 
\begin{equation}
\label{strong-projections}
\pi_{\calK_{\sfz(s_k)}(\Omega)}(\serifsigma_\dev(s_k) ) \longrightarrow \pi_{\calK_{\sfz(b)}(\Omega)}(\serifsigma_\dev(b) ) \stackrel{(*)}= \serifsigma_\dev(b)  \quad \text{in } 
L^p(\Omega;\MD) \ \text{for all } 1 \leq p <\infty\,,
\end{equation}
where  {\footnotesize (*)} ensues from the fact that $\calD^*(\sft(b), \sfq(b)) =0$. 
Combining \eqref{strong-projections} with \eqref{decomp-dev}, we ultimately conclude that 
\begin{equation}
\label{as-a-result-of-Step2}
\serifsigma_\dev(s_k)  \to \serifsigma_\dev(b) \text{  strongly in $L^2(\Omega;\MD) $.}
\end{equation}
\item[\textbf{Step $3$}]    We now observe that 
\[
 -\mathrm{Div}(\serifsigma(s))= F(\sft(s)) \equiv  F(\sft(b)) =  
  -\mathrm{Div}(\serifsigma(b))  \qquad \text{for all } s \in (a,b)
\]
since $\sft'\equiv 0$ on $(a,b)$. Therefore, $\mathrm{Div}(\serifsigma(s_k)) \to  \mathrm{Div}(\serifsigma(b))$ in $L^2(\Omega)$.  Combining this with 
the weak convergence $  \serifsigma(s_k) \weakto \serifsigma(b) $ in $\Lnn$ and with \eqref{as-a-result-of-Step2}, 
we  may thus apply Lemma 
\ref{0404232220} ahead to deduce the strong convergence \eqref{strong-sigma-k}.
\end{description} 


%

\medskip

\noindent
\textbf{Claim $2$:} {\it \eqref{strong-sigma-n} and \eqref{0404232301} hold true.}\medskip\\
In fact, assuming by contradiction \eqref{0404232301} false (and then also \eqref{strong-sigma-n}, by Claim~1), we have in particular
\[
\liminf_{s\uparrow b}  \calD^*(\sft(s), \sfq(s))>0,
\] so there exist $c\in (a,b)$ and  $\eta>0$ such that $\calD^*(\sft(s), \sfq(s))\geq \eta$ for every $s \in [c,b)$.
 From estimate
\eqref{straight-fwd-estimates} we thus conclude
\[
\int_{c}^b \| \sfp'(s)\|_{L^2} \dd s \leq  \int_{c}^b \sqrt{\| \sfz'(s)\|_{L^2}^2{+}\| \sfp'(s)\|_{L^2}^2} \dd s \leq \frac{\mathrm{M}}{\eta}\,.
\]
 Then Lemma~\ref{le:0404231757} gives 
\begin{equation*}
\|\sfe(s)-\sfe(b)\|_{L^2} \leq C_L \Big(\int_{s}^b \|\sfp'(s)\|_{L^2} \dd s+ \|\sfz(b)-\sfz(s)\|_{L^\infty}\Big) \qquad \text{for all } s \geq c
\end{equation*}
 with $C_L$ from \eqref{0404231831}.  
Therefore we have $\sfe(s) \to \sfe(b)$ in $\Lnn$ and ultimately $\serifsigma(s) \to \serifsigma(b)$ in  $\Lnn$. But 
 \eqref{strong-sigma-n} was  false by assumption. This concludes the proof of Claim $2$. 
\medskip

\noindent
\textbf{Claim $3$:} {\it There exists a sequence $(b_n)_n$ with $b_n \uparrow b$ such that}
\begin{equation}
\label{strong-Phi}
 \Phi(\sfz(b_n)) \to \Phi(\sfz(b)). 
\end{equation}
We distinguish two cases:
\begin{itemize}
\item[\textbf{Case 1}] $\liminf_{s\uparrow b}  \congdist z{\sft(s)}{\sfq(s)} <+\infty$;
\item[\textbf{Case 2}] $\liminf_{s\uparrow b}  \congdist z{\sft(s)}{\sfq(s)} =+\infty$.
\end{itemize}\par
In \textbf{Case 1}  we choose the sequence $(b_n)_n$, with
 $b_n \uparrow b$, such that $\lim_{n\to\infty} \congdist z{\sft(b_n)}{\sfq(b_n)} = \liminf_{s\uparrow b}  \congdist z{\sft(s)}{\sfq(s)}$.\par
In \textbf{Case 2}
  we choose  a  sequence $(b_n)_n$ 
 such that $\congdist z{\sft(b_n)}{\sfq(b_n)} \leq \congdist z{\sft(s)}{\sfq(s)}$ for all $s \in [b_n, b)$ and every $n\in \N$.
  Indeed,
  it is enough to define $b_n$ as the minimizer of $s\mapsto \congdist z{\sft(s)}{\sfq(s)}$ on the interval $[\max\{b_{n-1},b{-}\frac{1}{n}\},b)$, that is the minimizer over $[\max\{b_{n-1},b{-}\frac{1}{n}\},b]$ of the (lower semicontinuous) function $\congdistqtildeu z {\sft(s)}{\sfq(s)}$ defined as $\congdist z{\sft(s)}{\sfq(s)}$ for $s<b$ and $\congdistqtildeu z {\sft(b)}{\sfq(b)}:=+\infty$.

 To show \eqref{strong-Phi},  recalling Proposition \ref{p:cornerstone} we observe    that for every $n \in \N$ there holds
\begin{equation}
\label{ARV-est}
\begin{aligned}
0 \leq \ARV((\sft,\sfq); [b_n,b]):  & =  \Phi(\sfz(b)) -  \Phi(\sfz(b_n)) + \mathcal{R} (\sfz(b){-} \sfz(b_n))  
\\
&
\qquad  + \|\sfz(b)-\sfz(b_n)\|_{L^2} \frac{ \congdist z{\sft(b_n)}{\sfq(b_n)}{+}  \congdist z{\sft(b)}{\sfq(b)}}{2} 
\\
&
\qquad  + K_W \|\sfz(b)-\sfz(b_n)\|_{L^1} \|\sfz(b)-\sfz(b_n)\|_{L^\infty}
\,. 
\end{aligned}
\end{equation}
Now, since $\sfz : [0,S]\to \Hs(\Omega)$ is weakly continuous, taking into account that 
$\Hs(\Omega) \Subset \rmC^0(\overline\Omega)$ we clearly have that $ \mathcal{R} (\sfz(b){-} \sfz(b_n))   \to 0$ as $n \to \infty$. 
Analogously, $ K_W \|\sfz(b)-\sfz(b_n)\|_{L^1} \|\sfz(b)-\sfz(b_n)\|_{L^\infty} \to 0$.  
 Let us show that 
\begin{equation}\label{0404232324}
\|\sfz(b)-\sfz(b_n)\|_{L^2} \congdist z{\sft(b_n)}{\sfq(b_n)} \to 0.
\end{equation}
In \textbf{Case 1}, since $\lim_{n\to\infty} \congdist z{\sft(b_n)}{\sfq(b_n)} = \liminf_{s\uparrow b}  \congdist z{\sft(s)}{\sfq(s)} <+\infty$, we have that \[
\sup_n \congdist z{\sft(b_n)}{\sfq(b_n)}   <+\infty\]
 and \eqref{0404232324} follows from $\|\sfz(b)-\sfz(b_n)\|_{L^2} \to 0$;

In \textbf{Case 2},
from the choice of $b_n$ we have that for every $n$
\[
\congdist z{\sft(b_n)}{\sfq(b_n)} \leq \congdist z{\sft(s)}{\sfq(s)} \text{ for }s \in [b_n,b)
\] 
and thus
\begin{equation*}
\|\sfz(b)-\sfz(b_n)\|_{L^2} \congdist z{\sft(b_n)}{\sfq(b_n)} \leq \int_{b_n}^b \| \dot{\sfz}(s)\|_2 \congdist z{\sft(s)}{\sfq(s)} \dd s \longrightarrow 0,
\end{equation*}
since $ \| \dot{\sfz}(s)\|_2 \congdist z{\sft(s)}{\sfq(s)}$ is integrable in $(a,b)$.
We notice that $ \sfz(b)-\sfz(b_n)= \int_{b_n}^b \dot{\sfz}(s) \dd s$, where the right term is a Bochner integral in $L^2(\Omega)$ and it is well defined and finite by \eqref{straight-2} and since $\liminf_{s\uparrow b}  \calD^*(\sft(s), \sfq(s)) >0$.
Therefore \eqref{0404232324} is proven. Clearly, we also have 
\[
\|\sfz(b)-\sfz(b_n)\|_{L^2} , \congdist z{\sft(b)}{\sfq(b)} \longrightarrow 0.
\]


\par
All in all, from \eqref{ARV-est}
 we conclude that
 $\limsup_{n\to \infty}   \Phi(\sfz(b_n))  \leq \Phi(\sfz(b)) $. Since, by lower semicontinuity of $\Phi$, we also have 
 $\liminf_{n\to \infty}   \Phi(\sfz(b_n))  \geq \Phi(\sfz(b)) $, the desired convergence \eqref{strong-Phi} follows.
 \medskip

\noindent
\textbf{Claim $4$:} {\it It holds}
\begin{equation}
\label{strong-E}
\Epp(\sft(b_n), \sfq(b_n)) \to  \Epp(\sft(b), \sfq(b))\,.
\end{equation}
Indeed,
\[
\Epp(\sft(b_n), \sfq(b_n)) =   \calQ( \sfz(b_n),\sfe(b_n))   
+ \Ez(\sfz(b_n)) 
- \langle \sfF(b_n), \sfu(b_n){+}\sfw(b_n) \rangle_{\BD(\Omega)}\,.
\]
Now, since $\sfu: [0,S]\to \BD(\Omega)$ is weakly$^*$-continuous, also taking into account the continuity properties of $\sfF$ and $\sfw$ we immediately deduce that 
\[
\langle \sfF(b_n), \sfu(b_n){+}\sfw(b_n) \rangle_{\BD(\Omega)} \to \langle \sfF(b), \sfu(b){+}\sfw(b) \rangle_{\BD(\Omega)}\,.
\]
Therefore, it suffices to observe  that 
\[
\lim_{n\to \infty}  \calQ( \sfz(b_n),\sfe(b_n))   = \lim_{n\to \infty} \frac12 \int_{\Omega} \serifsigma(b_n): \sfe(b_n) \dd x 
=  \frac12 \int_{\Omega} \serifsigma(b): \sfe(b) \dd x  = 
   \calQ( \sfz(b),\sfe(b))\,,
\]
since $\serifsigma(b_n)\to \serifsigma(b) $ strongly in $\Lnn$
 and $\sfe(b_n) \weakto \sfe(b)$ weakly in $\Lnn$ by the weak continuity of $\sfe$. 
Then,  \eqref{strong-E}
ensues. 
\end{proof}
\par
We thus arrive at the following result.
\begin{lemma}
\label{l:DK3}
For every $i \in \DIN k3$ and for every $0<\eta, \beta \ll 1$, 
there exist points 
$s_k^{i-1} \leq a_{\beta}^i < b_{\beta}^i \leq s_k^i$  and a partition $(\subdivhat{\eta}ij)_{j=0}^{\cardihat \eta i}$ of the interval $[a_{\beta}^i, b_{\beta}^i]$ 
such that
\begin{subequations}\label{eqs:0809231011}
\begin{equation}
\label{features-a-b}
a_{\beta}^i,\, b_{\beta}^i \in A^\circ, \qquad 
 a_{\beta}^i - s_k^{i-1} \leq \frac{\beta}2\,, \qquad s_k^i - b_{\beta}^i \leq \frac{\beta}2\,,
 \end{equation}
  \begin{equation}\label{0809231026}
 \begin{split}
  & \| \sfw(a_{\beta}^i)-\sfw(s_k^{i-1})\|_{H^1(\Omega)} + \| \sfw(b_{\beta}^i)-\sfw(s_k^i)\|_{H^1(\Omega)} + \| \sfF(a_{\beta}^i)-\sfF(s_k^{i-1})\|_{\mathrm{BD}(\Omega)^*} 
  \\& \hspace{1em} + \| \sfF(b_{\beta}^i)-\sfF(s_k^i)\|_{\mathrm{BD}(\Omega)^*} + \| \sfp(a_{\beta}^i)-\sfp(s_k^{i-1})\|_{\Mb(\Omega)} + \| \sfp(b_{\beta}^i)-\sfp(s_k^i)\|_{\Mb(\Omega)}< \beta,
  \end{split}
 \end{equation} 
 \end{subequations}
 and
\begin{equation}
\label{cornerstone-DIN3}
\begin{aligned}
&
\sum_{j=1}^{\cardihat \eta i}
\WE((\sft,\sfq); [\subdivhat \eta i{j-1},\subdivhat \eta ij])
+  \Epp(\sft(s_k^{i-1}), \sfq(s_k^{i-1})) 
\\
 & \leq   \Epp(\sft(s_k^i), \sfq(s_k^i))   
 \\
 & \qquad + \sum_{j=1}^{\cardihat \eta i} \Hpp\left(\sfz(\subdivhat \eta i{j-1}), \frac{\sfp(\subdivhat \eta i{j}){-}\sfp(\subdivhat \eta i{j-1})}2\right)  + \Hpp\left(\sfz(\subdivhat \eta i{j}), \frac{\sfp(\subdivhat \eta i{j}){-}\sfp(\subdivhat \eta i{j-1})}{2}\right) 
 \\
 & \qquad 
 +
 \calR(\sfz(s_k^i){-}\sfz(s_k^{i-1}))  + \int_{s_k^{i-1}}^{s_k^i} \mathcal{D}(\sfq'(s)) \,  \mathcal{D}^*(\sft(s),\sfq(s)) \dd s 
 \\
 & \qquad 
 + \eta\, \frac{\mathrm{M}}{ \min_{s\in [a_{\beta}^i,b_{\beta}^i]}  \mathcal{D}^*(\sft(s),\sfq(s)) } +
 \mathrm{Rem}_2([a_{\beta}^i, b_{\beta}^i]) + \beta \,,
  \end{aligned}
 \end{equation}
 where
 $\mathrm{M}>0$ is from \eqref{straight-fwd-estimates} and  $\mathrm{Rem}_2([a_{\beta}^i, b_{\beta}^i]) $ is defined by the right-hand side of \eqref{Rem2}.
\end{lemma}
\begin{proof}
In view of Lemma \ref{l:tech-approx-energies},  \eqref{hyp-data}, and \eqref{additional-regularity},  for every fixed $\beta>0$ we may pick $a_{\beta}^i,\,b_{\beta}^i \in A^\circ$ such that  \eqref{eqs:0809231011} hold,  together with 
\begin{equation}
\label{control-of-energies}
\begin{split}\left|  \Epp(\sft(a_{\beta}^i),\sfq(a_{\beta}^i)) {-}   \Epp(\sft(s_k^{i-1}), \sfq(s_k^{i-1}))  \right| &\leq \frac{\beta}{4}\,, 
\\
 \left|  \Epp(\sft(b_{\beta}^i),\sfq(b_{\beta}^i)) {-}   \Epp(\sft(s_k^{i}), \sfq(s_k^{i}))  \right| &\leq \frac{\beta}{4}\,.
\end{split}\end{equation}
 Furthermore,  recall that  $\sfz: [0,S]\to \Hs(\Omega) $ is weakly continuous, so that  by the compact embedding
 $\Hs(\Omega) \Subset L^\infty(\Omega)$ we have that $\sfz \in \rmC^0([0,S];L^\infty(\Omega))$. Therefore, we may suppose that  the quantities
 $\| \sfz( a_{\beta}^i){-}\sfz (s_k^{i-1})\|_{L^\infty}$ and  $\| \sfz( b_{\beta}^i){-}\sfz (s_k^{i})\|_{L^\infty}$ 
 are so small  as to ensure that 
 \begin{equation}
 \label{control-of-Rs}
\calR(\sfz(b_{\beta}^i){-}\sfz(a_{\beta}^i))  \leq \calR(\sfz(s_k^i){-} \sfz (s_k^{i-1})) + \frac\beta 4\,.
 \end{equation}
 \par
 After these preparations, we proceed as in Lemma \ref{l:DK2} and 
with the points $a_{\beta}^i$ and $b_{\beta}^i$ we associate a partition 
$(\subdivhat \eta ij)_{j=1}^{\cardihat \eta i}$ of the interval $[a_{\beta}^i,b_{\beta}^i]$ such that 
\[
\begin{aligned}
&
\sum_{j=1}^{\cardihat \eta i}
\WE((\sft,\sfq); [\subdivhat \eta i{j-1},\subdivhat \eta ij])
+  \Epp(\sft(a_{\beta}^i), \sfq(a_{\beta}^i)) 
\\
 & \leq   \Epp(\sft(b_{\beta}^i), \sfq(b_{\beta}^i))   
 \\
 & \qquad + \sum_{j=1}^{\cardihat \eta i} \Hpp\left(\sfz(\subdivhat \eta i{j-1}), \frac{\sfp(\subdivhat \eta i{j}){-}\sfp(\subdivhat \eta i{j-1})}2\right)  + \Hpp\left(\sfz(\subdivhat \eta i{j}), \frac{\sfp(\subdivhat \eta i{j}){-}\sfp(\subdivhat \eta i{j-1})}{2}\right) 
 \\
 & \qquad 
 +
 \calR(\sfz(b_{\beta}^i){-}\sfz(a_{\beta}^i))  + \int_{a_{\beta}^i}^{b_{\beta}^i} \mathcal{D}(\sfq'(s)) \,  \mathcal{D}^*(\sft(s),\sfq(s)) \dd s 
 \\
 & \qquad 
 + \eta\, \frac{ 
\mathrm{M}}{\min_{s\in [a_{\beta}^i,b_{\beta}^i]}  \mathcal{D}^*(\sft(s),\sfq(s))} +
 \mathrm{Rem}_2([a_{\beta}^i,b_{\beta}^i])\,.
  \end{aligned}
  \]
  Then,
estimate \eqref{cornerstone-DIN3}
 follows by  \eqref{control-of-energies},  \eqref{control-of-Rs}, and by observing that 
 \[
 \begin{aligned}
 &
  \int_{a_{\beta}^i}^{b_{\beta}^i} \mathcal{D}(\sfq'(s)) \,  \mathcal{D}^*(\sft(s),\sfq(s)) \dd s  \leq  \int_{s_k^{i-1}}^{s_k^i}
   \mathcal{D}(\sfq'(s)) \,  \mathcal{D}^*(\sft(s),\sfq(s)) \dd s\,.
\end{aligned}
 \]
 \end{proof}

\section{Proof of Theorem \ref{mainth:1}: the lower energy-dissipation inequality}
\label{s:6}
In this section we eventually prove
the following result.
\begin{proposition}
 \label{prop:LEDI}
 The pair $(\sft, \sfq)$ satisfies the   lower energy-dissipation inequality \eqref{LEDI}. 
\end{proposition}
\begin{proof}
We start by adding up 
\begin{enumerate}
\item estimate \eqref{cornerstone-DIN1} over all indices $i \in \DIN k1 $;
\item estimate \eqref{cornerstone-DIN2}, with fixed $\eta>0$,  over all indices $i \in \DIN k2 $,
\item
estimate \eqref{cornerstone-DIN3}, with fixed $\eta,\, \beta>0$,  over all indices $i \in \DIN k3 $.
\end{enumerate}
Eventually, we will  add the resulting inequalities and obtain an overall discrete version of  \eqref{LEDI},
cf.\ \eqref{overall-LEDI-discrete} below. 
 In order to write it in a compact form,
let us introduce some place-holders: we set
\[
\begin{aligned}
\FWE {}{\mathsf{w}}k  & = \sum_{i \in \DIN k1}  \langle \tfrac12 \left( \serifsigma(s_k^{i-1}) + \serifsigma(s_k^i)\right) , \EE(\sfw(s_k^{i})-\sfw(s_k^{i-1})) \rangle_{L^2(\Omega)}
\\ & \quad 
 + \sum_{i \in \DIN k2} \sum_{j=1}^{\cardi \eta i}   \langle \tfrac12 \left( \serifsigma(\subdiv \eta i {j-1}) + \serifsigma(\subdiv \eta i j)\right) , \EE(\sfw(\subdiv \eta i j)-\sfw(\subdiv \eta i {j-1})) \rangle_{L^2(\Omega)}
 \\
 & \quad 
 + \sum_{i \in \DIN k3} \sum_{j=1}^{\cardihat \eta i}   \langle \tfrac12 \left( \serifsigma(\subdivhat \eta i {j-1}) + \serifsigma(\subdivhat \eta i j)\right) , \EE(\sfw(\subdivhat \eta i j)-\sfw(\subdivhat \eta i {j-1})) \rangle_{L^2(\Omega)}
 \end{aligned}
 \]
 and 
\[
\begin{aligned}
\FWE 1{\mathsf{F}, \sfw}k  & = \sum_{i \in \DIN k1} \langle \tfrac12 \left( \sfF(s_k^{i-1}) +  \sfF(s_k^i) \right), \sfw(s_k^{i})-\sfw(s_k^{i-1}) \rangle_{\BD(\Omega)}  
\\ & \quad 
 + \sum_{i \in \DIN k2} \sum_{j=1}^{\cardi \eta i} \langle   \tfrac12\left( \sfF(\subdiv \eta i {j-1}) +  \sfF(\subdiv \eta i {j}) \right), \sfw(\subdiv \eta i {j})-\sfw(\subdiv \eta i {j-1}) \rangle_{\BD(\Omega)}  
 \\
 & \quad 
  + \sum_{i \in \DIN k3} \sum_{j=1}^{\cardihat \eta i}  \langle  \tfrac12\left( \sfF(\subdivhat \eta i {j-1}) +  \sfF(\subdivhat \eta i {j}) \right), \sfw(\subdivhat \eta i {j})-\sfw(\subdivhat \eta i {j-1}) \rangle_{\BD(\Omega)}  \,,
 \end{aligned}
 \]
  as well as 
 \[
\begin{aligned}
\FWE 2{\mathsf{F}, \sfw}k  & = \sum_{i \in \DIN k1} \langle   \sfF(s_k^{i}) {-}  \sfF(s_k^{i-1}) , \tfrac12 ( \sfx(s_k^{i}){+}\sfx(s_k^{i-1})) \rangle_{\BD(\Omega)}  
\\ & \quad 
 + \sum_{i \in \DIN k2} \sum_{j=1}^{\cardi \eta i} \langle    \sfF(\subdiv \eta i {j}) {-}\sfF(\subdiv \eta i {j-1}) ,  \tfrac12( \sfx(\subdiv \eta i {j}){+}\sfx(\subdiv \eta i {j-1}) )\rangle_{\BD(\Omega)}  
 \\
 & \quad 
  + \sum_{i \in \DIN k3} \sum_{j=1}^{\cardihat \eta i}  \langle   \sfF(\subdivhat \eta i {j}){-} \sfF(\subdivhat \eta i {j-1}) , \tfrac12 ( \sfx(\subdivhat \eta i {j}){+}\sfx(\subdivhat \eta i {j-1}) )\rangle_{\BD(\Omega)}  \,,
 \end{aligned}
 \]  
 where for shorter notation here  we have used the place-holder $\sfx := \sfu+\sfw$. 
 Clearly, we have  
\[
\begin{aligned}
\sum_{i \in \DIN k1} \WE((\sft,\sfq); [s_k^{i-1},s_k^i])  &  + \sum_{i \in \DIN k2} \sum_{j=1}^{\cardi \eta i}  \WE((\sft,\sfq); [\subdiv \eta i {j-1},\subdiv \eta i j])
\\
&
\quad + \sum_{i \in \DIN k3} \sum_{j=1}^{\cardihat \eta i}  \WE((\sft,\sfq); [\subdivhat \eta i {j-1},\subdivhat \eta i j])
\\
&
= \FWE{} {\mathsf{w}}k  - \FWE 1{\mathsf{F}, \mathsf{w}}k   -   \FWE 2{\mathsf{F}, \mathsf{w}}k \,. 
\end{aligned}
\]
Analogously, 
we introduce a  place-holder for the terms that approximate the integral $\int_0^S \Hpp(\sfz, \sfp') \dd s$, namely
\[
\begin{aligned}
& \FHV k  (\sft, \sfq; [0,S]): =
\\
 &  \sum_{i\in \DIN k1}  
 \Hpp\left(\sfz(s_k^{i-1}), \frac{\sfp(s_k^i){-}\sfp(s_k^{i-1})}2\right)  + \Hpp\left(\sfz(s_k^i), \frac{\sfp(s_k^{i}){-}\sfp(s_k^{i-1})}{2}\right) 
 \\
 & \quad  + \sum_{i\in \DIN k2}  \sum_{j=1}^{\cardi \eta i} \Hpp\left(\sfz(\subdiv \eta i{j-1}), \frac{\sfp(\subdiv \eta i{j}){-}\sfp(\subdiv \eta i{j-1})}2\right)  + \Hpp\left(\sfz(\subdiv \eta i{j}), \frac{\sfp(\subdiv \eta i{j}){-}\sfp(\subdiv \eta i{j-1})}{2}\right) 
 \\
  & \quad  + \sum_{i\in \DIN k3}  \sum_{j=1}^{\cardihat \eta i} \Hpp\left(\sfz(\subdivhat \eta i{j-1}), \frac{\sfp(\subdivhat \eta i{j}){-}\sfp(\subdivhat \eta i{j-1})}2\right)  + \Hpp\left(\sfz(\subdivhat \eta i{j}), \frac{\sfp(\subdivhat \eta i{j}){-}\sfp(\subdivhat \eta i{j-1})}{2}\right) \,.
  \end{aligned}
\]
%
We also consider the sum of the remainder terms
\[
\FREM k ([0,S]): = \sum_{i\in \DIN k1}  \mathrm{Rem}_1([s_k^{i-1},s_k^i])+\sum_{i\in \DIN k2}  \mathrm{Rem}_2([s_k^{i-1},s_k^i])
+\sum_{i\in \DIN k3}  \mathrm{Rem}_2([a_{\beta}^i, b_{\beta}^i]) 
\]
with $  \mathrm{Rem}_1$ and $ \mathrm{Rem}_2$ from \eqref{Rem1} and  \eqref{Rem2}, respectively,
 and  where $(a_{\beta}^i)_{i\in \DIN k3}$ and $(b_{\beta}^i)_{i\in \DIN k3}$  are the points 
associated with  $(s_k^{i-1})_{i\in \DIN k3}$ and $(s_k^{i})_{i\in \DIN k3}$, respectively,
as in Lemma \ref{l:DK3}. 
Furthermore, we  observe that 
\[
\begin{aligned}
\sum_{i \in \DIN k1} \calR(\sfz(s_k^i){-}\sfz(s_k^{i-1})) +  \sum_{i \in \DIN k2} \calR(\sfz(s_k^i){-}\sfz(s_k^{i-1}))+ 
\sum_{i \in \DIN k3} \calR(\sfz(s_k^i){-}\sfz(s_k^{i-1}))  
& = \calR(\sfz(S){-}\sfz(0)) 
\\
&
= \int_0^S \calR(\sfz'(s)) \dd s 
\end{aligned}
\]
and, analogously,  
\[
\begin{aligned}
& \sum_{i \in \DIN k1} \{  \Epp(\sft(s_k^i),\sfq(s_k^i)) {-}  \Epp(\sft(s_k^{i-1}),\sfq(s_k^{i-1})) \}
\\ & + \sum_{i \in \DIN k2} \{  \Epp(\sft(s_k^i),\sfq(s_k^i)) {-}  \Epp(\sft(s_k^{i-1}),\sfq(s_k^{i-1})) \}
\\ & + \sum_{i \in \DIN k3} \{  \Epp(\sft(s_k^i),\sfq(s_k^i)) {-}  \Epp(\sft(s_k^{i-1}),\sfq(s_k^{i-1})) \}
=  \Epp(\sft(S), \sfq(S)) -  \Epp(\sft(0), \sfq(0)),
\end{aligned}
\]
while
\[
\begin{aligned}
&
 \sum_{i \in \DIN k2} \int_{s_k^{i-1}}^{s_k^i} \calD(\sfq'(s))\, \calD^*(\sft(s), \sfq(s)) \dd s 
+ \sum_{i \in \DIN k3} \int_{s_k^{i-1}}^{s_k^i} \calD(\sfq'(s))\, \calD^*(\sft(s), \sfq(s)) \dd s
\\
&   \leq  \int_{(0,S) \cap A^\circ} \calD(\sfq'(s))\, \calD^*(\sft(s), \sfq(s)) \dd s\,.
\end{aligned}
\]
\par
Adding up \eqref{cornerstone-DIN1},  \eqref{cornerstone-DIN2}, and  \eqref{cornerstone-DIN3}, we eventually obtain
\begin{equation}
\label{overall-LEDI-discrete} 
\begin{aligned}
&
\FWE {}{\mathsf{w}}k  - \FWE {1}{\mathsf{F},\mathsf{w}}k   - \FWE {2}{\mathsf{F},\mathsf{w}}k   +  \Epp(\sft(0), \sfq(0))
\\
&
 \leq   \Epp(\sft(S), \sfq(S))
+ \FHV k  (\sft, \sfq; [0,S]) 
 + \int_0^S \calR(\sfz'(s)) \dd s  
 + \int_{(0,S) \cap A^\circ} \!\!\!\!\!\!\!\!\!\!\calD(\sfq'(s))\, \calD^*(\sft(s), \sfq(s)) \dd s
\\
& \quad  +\FREM k ([0,S])
+ \eta \mathrm{M}  \sum_{i\in \DIN k2}
 \frac1{\min_{s\in [s_k^{i-1}, s_k^i]}  \mathcal{D}^*(\sft(s),\sfq(s))}  
 \\
 & \quad 
  +  \eta \mathrm{M}  \sum_{i\in \DIN k3}
 \frac1{\min_{s\in [a_{\beta}^i,b_{\beta}^i]}  \mathcal{D}^*(\sft(s),\sfq(s))}
 + \beta\, \#(\DIN k3)
 \end{aligned}
\end{equation}
where the very last term on the right-hand side derives from adding up,
for each index $ i \in \DIN k3$,
the term 
$\beta$ on the r.h.s.\ of \eqref{cornerstone-DIN3}. 
Recalling \eqref{eqs:2906191823}, set
\begin{equation}\label{0809231031}
\widetilde{M}:=\sup_{s\in [0,S]} \left( \|\sfs(s)\|_{L^2(\Omega)} + \|\sfF(s)\|_{\mathrm{BD}(\Omega)^*} + \|
\sfu(s) + \sfw(s)\|_{\mathrm{BD}(\Omega)} + \|\sfz(s)\|_{L^\infty(\Omega)} \right).
\end{equation} 

Let us take the limit in \eqref{overall-LEDI-discrete} as  $\eta \to 0$, $\beta \to 0$ and  $k\to +\infty$  in this order. 
 In fact, for any fixed $k\in \N$ it is possible to choose $\beta=\beta(k)$ and then $\eta=\eta(\beta,k)$ in such a way  to make the last three terms in the r.h.s.\ of  \eqref{overall-LEDI-discrete}  arbitrarily small.
 More precisely, for fixed $k\in \N $ we choose $0<\beta\ll1$ in such a way as to make the last term on the right-hand side of \eqref{overall-LEDI-discrete} 
arbitrarily small; then, we choose $0<\eta\ll 1$ , depending on the intervals $\DIN k2$, $\DIN k3$, and on the previously found $\beta$ so that the third-to-last term and the second-to-last term on the r.h.s.\ of \eqref{overall-LEDI-discrete} are arbitrarily small.

Moreover, in the following lines we will show   that also the (discrete) terms depending on $k$ in \eqref{overall-LEDI-discrete} may be made arbitrarily close to their continuous counterparts in the lower energy-dissipation inequality \eqref{LEDI}, as $k\to +\infty$ for suitable $\beta=\beta(k)$.

 Let us introduce the partition of $[0,S]$
\[
(\mathfrak{s}_k^j)_{j=0}^{L_k}:=\bigcup_{i \in \DIN k1} \{s_k^i\} \cup \bigcup_{i \in \DIN k2} \bigcup_{j=1}^{M_{\eta,i}} \{r^j_{\eta,i}\} \cup \bigcup_{i \in \DIN k3} \left( \bigcup_{j=1}^{L_{\eta,i}} \{\varrho^j_{\eta,i}\} \cup \{s_k^i\} \right).
\]
In fact, the number of nodes  of this partition  also depends on $\eta$, but for shorter notation we do not highlight this dependence. Alternatively, one could formally consider a suitable vanishing sequence $(\eta_k)_k$, to avoid explicit dependence on $\eta$.

 Thus, we may express 
 \begin{equation*}
 \begin{split}
 \FWE {}{\mathsf{w}}k &= \sum_{j=1}^{L_k} \langle \tfrac12 \left( \serifsigma(\mathfrak{s}_k^{j-1}) + \serifsigma(\mathfrak{s}_k^j)\right) , \EE(\sfw(\mathfrak{s}_k^{j})-\sfw(\mathfrak{s}_k^{j-1})) \rangle_{L^2(\Omega)} 
 \\  &- \tfrac12\sum_{i \in \DIN k3}  \Big( \langle   \serifsigma(a_{\beta}^i)  + \serifsigma(s_k^{i-1}), \EE(\sfw(s_k^{i-1})-\sfw(a_{\beta}^i)) \rangle_{L^2(\Omega)} \\
 & \hspace{3em}+  \langle \serifsigma(b_{\beta}^i)  + \serifsigma(s_k^{i}) , \EE(\sfw(s_k^{i})-\sfw(b_{\beta}^i)) \rangle_{L^2(\Omega)} \Big),
 \end{split}
 \end{equation*}
  so that
 \begin{equation*}
 \Big| \FWE {}{\mathsf{w}}k - \sum_{j=1}^{L_k} \langle \tfrac12 \left( \serifsigma(\mathfrak{s}_k^{j-1}) + \serifsigma(\mathfrak{s}_k^j)\right) , \EE(\sfw(\mathfrak{s}_k^{j})-\sfw(\mathfrak{s}_k^{j-1})) \rangle_{L^2(\Omega)}  \Big| \leq   \beta   \, \widetilde{M}  \,  \# (\DIN k3),
 \end{equation*}
 by \eqref{0809231026}, for $\widetilde{M}$ from \eqref{0809231031}. 

Now, it follows from 
Lemma \ref{l:approx-sigma} ahead  that 
\[
\lim_{k\to\infty}  \sum_{j=1}^{L_k} \langle \tfrac12 ( \serifsigma(\mathfrak{s}_k^{j-1}) {+}\serifsigma(\mathfrak{s}_k^j)) , \EE(\sfw(\mathfrak{s}_k^{j}){-}\sfw(\mathfrak{s}_k^{j-1})) \rangle_{L^2(\Omega)} 
=  \int_0^S \langle \serifsigma_k(s), \EE \sfw'(s) \rangle \dd s \,,
\]
and therefore we ultimately have 
\[
 \lim_{k\to\infty}   \lim_{\beta\to 0}  \FWE {}{\mathsf{w}}k=  \int_0^S \langle \serifsigma_k(s), \EE \sfw'(s) \rangle \dd s \,.
\]
Arguing similarly, we may  infer 
\[
\Big| \FWE 1{\mathsf{F},\mathsf{w}}k  -\  \sum_{j=1}^{L_k}  \langle \tfrac12 \left( \sfF(\mathfrak{s}_k^{i-1}) +  \sfF(\mathfrak{s}_k^i) \right), \sfw(\mathfrak{s}_k^{i})-\sfw(\mathfrak{s}_k^{i-1}) \rangle_{\BD(\Omega)}  \Big| \leq  \beta   \, \widetilde{M}  \# (\DIN k3),
\] 
 and  thanks to Lemma  \ref{l:approx-Fw'} conclude that 
\[
\lim_{k\to \infty}  \lim_{\beta\to 0}  \FWE 1{\mathsf{F},\mathsf{w}}k  = 
%
 \int_0^S \langle \sfF(s), \EE\sfw'(s)\rangle_{\BD(\Omega)} \dd s.
\]
Likewise,  and employing now  Lemma \ref{l:approx-F'w}, we have 
\[
\lim_{k\to \infty}  \lim_{\beta\to 0} \FWE 2{\mathsf{F},\mathsf{w}}k  = 
%
 \int_0^S \langle \sfF'(s), \EE(\sfu(s){+}\sfw(s)) \rangle_{\BD(\Omega)} \dd s.
\]
%

 The convergence
\begin{equation*}
 \lim_{k\to \infty} \lim_{\beta\to 0}  \FHV k  (\sft, \sfq; [0,S]) = \int_0^S \Hpp(\sfz(s), \sfp'(s)) \dd s
\end{equation*}
follows from \cite[Lemma~8.2]{DalDesSol11}, noticing that (similarly as before)
 we may  obtain 
\[
\begin{aligned}
\Bigg|\FHV k  (\sft, \sfq; [0,S]) & -\tfrac12 \sum_{j=1}^{L_k} \Big[
 \Hpp\left(\sfz(\mathfrak{s}_k^{j-1}), \sfp(\mathfrak{s}_k^j){-}\sfp(\mathfrak{s}_k^{j-1})\right)  
 \\
  & + \Hpp\left(\sfz(\mathfrak{s}_k^j), \sfp(\mathfrak{s}_k^{j}){-}\sfp(\mathfrak{s}_k^{j-1})\right) \Big ] \Bigg| \leq \beta \, \widetilde{M}  \, \# \DIN k3 \,.
  \end{aligned}
\]

Eventually, $\FREM k ([0,S])\to 0$ by the following Lemma~\ref{le:0504232015}.
All in all, we conclude the lower energy-dissipation inequality \eqref{LEDI}. 
\end{proof}
%
%

\noindent
 We  prove the following  technical lemma employed in the proof of Proposition~\ref{prop:LEDI}. 
\begin{lemma}\label{le:0504232015}
There exists a constant $K_{\mathrm{M}}>0$, only depending on the
toughness constant $\kappa$, on the 
 constants $K_W $ and  $K_\bbC$ from \eqref{0504232022}, and on $\mathrm{M}>0$ from \eqref{straight-fwd-estimates},  such that 
\begin{equation}
\label{est-REM}
\forall\, \delta\in (0,1] \  \exists\, \bar{k} \in \N \  \forall\, k \geq \bar k \, : \quad \FREM k ([0,S])\leq \delta K_{\mathrm{M}}\,.
\end{equation}
\end{lemma}

\begin{proof}
 Recall that  
 $\sfz \in \rmC^0([0,S];L^\infty(\Omega))$.
Therefore, 
 for every 
 $\delta \in (0,1]$
 there exists $\bar k \in \N$ such that, for $k\geq \bar  k$ the fineness $\max_{i=1,\ldots, N_k}(s_k^i{-}s_k^{i-1})$ of the partition is so small that 
\[
\sup_{i =1,\ldots, N_k} \| \sfz(s_k^i){-}\sfz(s_k^{i-1})\|_{L^\infty(\Omega)} \leq \delta\,.
\]
Then,
\begin{equation}
\label{estimate-rem1}
\begin{aligned}
&
\sum_{i\in \DIN k1}  \mathrm{Rem}_1([s_k^{i-1},s_k^i]) 
\\
 & \leq \delta \sum_{i\in \DIN k1}
\Big(  K_{\bbC} \left\{  \Phi(\sfz(s_k^i)) {-}  \Phi(\sfz(s_k^{i-1})) {+}  \calR(\sfz(s_k^i){-}\sfz(s_k^{i-1})) \right\}
\\
& \qquad \qquad \qquad 
   +K_W \|\sfz(s_k^i){-}\sfz(s_k^{i-1})\|_{L^1} \left\{ 1{+}
 K_{\bbC}  \|\sfz(s_k^i){-}\sfz(s_k^{i-1})\|_{L^\infty}  \right\}\Big)
\\
& 
\stackrel{(1)}{\leq} \delta \left(  (2C_{\mathrm{M}} {+}\mathrm{M}) K_{\bbC} +\frac1{\kappa}\mathrm{M} K_W (1+ K_{\bbC} \delta) \right) \doteq \delta \widehat{K}_{\mathrm{M}}
\end{aligned}
\end{equation}
with $\kappa>0$ the toughness constant from \eqref{damage-pot-intro}. 
Now, for {\footnotesize (1)}  we have used that 
\[
 \sum_{i\in \DIN k1}  \left( \Phi(\sfz(s_k^i)) {-}  \Phi(\sfz(s_k^{i-1}))\right) \leq 2 \sup_{s \in [0,S]} \Phi(\sfz(s)) \leq 2C_{\mathrm{M}},
 \]
 for some positive constant $C_{\mathrm{M}}$ only depending on the constant $\mathrm{M}>0$ from \eqref{straight-fwd-estimates} and on  the problem data,
 since $\sup_{s\in [0,S]} |\Epp(\sft(s),\sfq(s))|\leq \mathrm{M}$. We have also used that  
  \[
 \left.
 \begin{array}{rr}
 &
  \sum_{i\in \DIN k1} \calR(\sfz(s_k^i){-}\sfz(s_k^{i-1})) 
  \\
  &
  \kappa  \sum_{i\in \DIN k1}  \|\sfz(s_k^i){-}\sfz(s_k^{i-1})\|_{L^1} 
  \end{array}
  \right\}
  \leq \int_0^S \calR(\sfz'(s)) \dd s \leq \mathrm{M}\,.
 \]
 By adapting the above calculations we estimate for each $i\in \DIN k2$
 \[
 \begin{aligned}
 & \mathrm{Rem}_2( [s_k^{i-1},s_k^i] ) 
 \\
  & 
  =   \sum_{j=1}^{\cardi \eta i} \Delta(\subdiv \eta i {j-1},\subdiv \eta i {j})   \|\sfz(\subdiv \eta i j){-}\sfz(\subdiv \eta i {j-1})\|_{L^\infty}
  \\ & \leq \delta  \sum_{j=1}^{\cardi \eta i} \Big(  K_{\bbC} \left\{  \Phi(\sfz(\subdiv \eta ij)) {-}  \Phi(\sfz(\subdiv \eta i {j-1})) {+}  \calR(\subdiv \eta ij){-}\sfz(\subdiv \eta i {j-1})) \right\}
  \\ & \qquad \qquad \quad  +K_W \|\sfz(\subdiv \eta ij){-}\sfz(\subdiv \eta i{j-1})\|_{L^1} \left\{ 1{+}
 K_{\bbC}  \delta \right\}\Big)
 \\
 & 
\stackrel{(2)}{\leq} \delta \left( \Phi(\sfz(s_k^i){-}  \Phi(\sfz(s_k^{i-1}) {+}   \calR(\sfz(s_k^i){-}\sfz(s_k^{i-1})) {+} K_W (1{+}K_{\bbC} \delta) \int_{s_k^{i-1}}^{s_k^i} \| \sfz'(s)\|_{L^1(\Omega)} \dd s \right)
 \end{aligned}
 \]
 where for {\footnotesize (2)} we have used that the points  $(\subdiv \eta ij)_{j=1}^{\cardi \eta i}$ provide a partition of the interval $[s_k^{i-1},s_k^i]$. 
 Therefore, 
 \begin{equation}
 \label{estimate-rem2} 
 \begin{aligned}
 & \sum_{i\in \DIN k2}   \mathrm{Rem}_2( [s_k^{i-1},s_k^i] ) 
 \\
 & \leq \delta \sum_{i\in \DIN k2}   \left( \Phi(\sfz(s_k^i){-}  \Phi(\sfz(s_k^{i-1}) {+}   \calR(\sfz(s_k^i){-}\sfz(s_k^{i-1})) {+} K_W (1{+}K_{\bbC} \delta) \int_{s_k^{i-1}}^{s_k^i} \| \sfz'(s)\|_{L^1(\Omega)} \dd s \right)
 \\
 &
\leq  \delta \widehat{K}_{\mathrm{M}}
\end{aligned}
 \end{equation}
 by the very same arguments as in the proof of \eqref{estimate-rem1}.  
 In a completely analogous way we find that 
  \begin{equation}
 \label{estimate-rem3} 
 \sum_{i\in \DIN k3}   \mathrm{Rem}_2( [a_\eta^{i},b_\eta^i] ) \leq \delta  \widehat{K}_{\mathrm{M}}\,.
 \end{equation}
Eventually, adding \eqref{estimate-rem1}, \eqref{estimate-rem2}, and 
\eqref{estimate-rem3},   we conclude \eqref{est-REM} with $K_{\mathrm{M}}: = 3\widehat{K}_{\mathrm{M}}$.
\end{proof}

\bigskip

\appendix 
\section{} We collect here some technical results employed in the paper.

\subsection{Proof of Lemma \ref{le:mediesci}}
\label{ss:proofLemVito}
  We provide here the proof of Lemma \ref{le:mediesci}.  
\medskip

\noindent
\textbf{Step 1.}
Let 
\begin{equation*}
m_1:=\min_{x \in [a,b]} \psi(x)\,.
\end{equation*}
Since $\psi$ 
attains its minimum on $[a,b]$ and is finite in $a$ and $b$, we have 
that $m_1 \in (0,+\infty)$. 
Then, the set 
\[
A_1:=\{ x \in [a,b] \colon \psi(x)\leq m_1+\eta\}
\]
is not empty.
Since $\psi$ is lower semicontinuous, $A_1$ is a compact subset  of $\R$; let  
\begin{equation}\label{2012201630}
a_1:=\min A_1\,; \qquad b_1:= \max A_1\,;  \qquad J_1:=\{a_1,\, b_1\}\,.
\end{equation}
Let us distinguish different cases:
\begin{itemize}
\item[\textbf{Case 1.I:}] \textbf{$a_1=a$, $b_1=b$.} 
If $a_1=a$, $b_1=b$, we set $\subd_\eta^0:=a_1=a$ and $\subd_\eta^1:=b_1=b$. 
We have that 
\begin{equation*}
\psi(s) \geq \frac{1}{2}\Big(\psi(\subd_\eta^0) + \psi(\subd_\eta^1) \Big)-\eta\quad \text{for every } s\in (a,b)\,,
\end{equation*} 
so that the partition given by  $J_1=(\subd_\eta^j)_{j=0}^1$ satisfies \eqref{2012201629} and the proof is finished. 
\item[\textbf{Case 1.II:}] Otherwise,  three subcases may occur:
\begin{itemize}
\item[\textbf{Case 1.IIa:}]
 \textbf{$a<a_1 \leq b_1 =b$} (notice that possibly $a_1=b_1$).
  We define the function  $\widehat{\psi}_1^-\colon [a,a_1]\to (0,+\infty]$ as
 \begin{equation}\label{2012201708}
 \widehat{\psi}_1^-(x):=
 \begin{dcases}
 \psi(x)& \quad x \in [a,a_1)\,,
 \\
 \liminf_{r\to a_1^-}\psi(r) & \quad x=a_1\,.
 \end{dcases}
 \end{equation}
 \item[\textbf{Case 1.IIb:}]
\textbf{$a=a_1  \leq b_1 < b$.} 
We 
define $\widehat{\psi}_1^+\colon [b_1,b]\to (0,+\infty]$ as
 \begin{equation}\label{2012201707}
 \widehat{\psi}_1^+(x) :=
 \begin{dcases}
 \psi(x),& \quad x \in (b_1,b]\,,\\
 \liminf_{r\to b_1^+} \psi(r) & \quad x=b_1\,. 
 \end{dcases}
 \end{equation}
 \item[\textbf{Case 1.IIc:}] \textbf{$a <a_1 \leq b_1 <b$.} We  
  define $\widehat{\psi}_1^-\colon [a,a_1]\to (0,+\infty]$, $\widehat{\psi}_1^+\colon [b_1,b]\to (0,+\infty]$ as in \eqref{2012201708} and \eqref{2012201707}, respectively.
\end{itemize}
\end{itemize}
In any case, if $a_1<b_1$, we get
\begin{equation}\label{2012201856}
\psi(s) \geq \frac{1}{2} \Big(\psi(a_1) + \psi(b_1) \Big) - \eta \quad\text{for }s \in (a_1,b_1)\,.
\end{equation}
This concludes the discussion in  Step 1.
\medskip

\noindent
\textbf{Step 2.} Let
\begin{equation*}
m_2^\pm:= \min \widehat{\psi}_1^\pm\,.
\end{equation*}
We notice that $m_2^+$, $m_2^- \geq m_1 +\eta$. Then we define  
\[
A_2^-:=\{x \in [a, a_1] \colon \widehat{\psi}_1^-(x) \leq m_2^-+\eta\}\,,\quad A_2^+:=\{x \in [b_1,b] \colon \widehat{\psi}_1^+(x) \leq m_2^++\eta\}\,.
\]
Similarly as in Step~1, we consider 
\[
a_2:=\min A_2^-\,,\qquad b_2:=\max A_2^+\,,
\]
and 
%
observe that by construction
\begin{subequations}\label{eqs:2012201859}
\begin{equation}\label{2012201859}
\psi(s)\geq\frac{1}{2}\Big(\psi(a_2) + \psi(a_1) \Big) - \eta \quad\text{for }s \in (a_2, a_1)\,,
\end{equation}
\begin{equation}\label{2012201900}
\psi(s)\geq\frac{1}{2}\Big(\psi(b_1) + \psi(b_2) \Big) - \eta \quad\text{for }s \in (b_1, b_2)\,,
\end{equation}
where possibly $a_1=a_2$, $b_1=b_2$.
We set 
\[
J_2: = J_1 \cup \{a_2,\, b_2\}\,.
\]
\end{subequations}
As in Step~1, we distinguish different cases:
\begin{itemize}
\item[\textbf{Case 2.I:}]\textbf{$a_2=a$, $b_2=b$.} We have that the points of the set $J_2$
  provide a partition of $[a,b]$ satisfying 
 in view of \eqref{2012201856} and \eqref{eqs:2012201859}, and this finishes the proof.
\item[\textbf{Case 2.II:}]
We again distinguish three sub-cases:
\begin{itemize}
\item[\textbf{Case 2.IIa:}] \textbf{$a<a_2\leq b_2=b$.}
We define 
$\widehat{\psi}_2^-\colon [a,a_2]\to (0,+\infty]$ as
 \begin{equation}\label{2012201915}
 \widehat{\psi}_2^-(x):=
 \begin{dcases}
 \psi(x)& \quad x \in [a,a_2)\,,\\
 \liminf_{r\to a_2^-}\psi(r)& \quad x =a_2\,,
 \end{dcases}
 \end{equation}
\item[\textbf{Case 2.IIb:}] \textbf{$a=a_2\leq b_2<b$.}
We define
$\widehat{\psi}_2^+\colon [b_2,b]\to (0,+\infty]$ as
 \begin{equation}\label{2012201915'}
 \widehat{\psi}_2^+(x):=
 \begin{dcases}
 \psi(x)& \quad x \in (b_2,b]\,,
 \\
 \liminf_{r\to b_2^+}\psi(r) & \quad r=b_2\,.
 \end{dcases}
 \end{equation}
 \item[\textbf{Case 2.IIc:}] \textbf{$a<a_2\leq b_2<b$.} We consider both $\widehat{\psi}_2^-$ and $\widehat{\psi}_2^+$. 
 \end{itemize}
 \end{itemize}
  This concludes the discussion of Step $2$. 
 \medskip

\noindent
 \textbf{Step 3.} We consider  
$
m_3^\pm:= \min \widehat{\psi}_3^\pm\,.
$
and accordingly set 
\[
\begin{aligned}
&
A_3^-:=\{x \in [a, a_2] \colon \widehat{\psi}_2^-(x) \leq m_2^-+\eta\}\,,\quad A_3^+:=\{x \in [b_2,b] \colon \widehat{\psi}_2^+(x) \leq m_2^++\eta\}\,.
\\
&
a_3:=\min A_3^-\,,\qquad b_3:=\max A_3^+\,, \qquad J_3 = J_2 \cup \{a_3, b_3\}\,.
\end{aligned}
\]
If $a_3=a$ and $b_3=b$, then the proof is finished. Otherwise, 
%
%
similarly as in the previous steps,
 If $a_3=a$ and $b_3<b$,   in the subsequent step we have still to refine the  partition only ``from the right'' by resorting to a function $\widehat{\psi}_3^+ $defined in analogy with $\widehat{\psi}_1^+$ , $\widehat{\psi}_2^+$,  while ``from the left'' the partition is fine for our purposes. Similarly,  if $a_3>a$ and $b_3=b$,  in the next step we have still to refine the  partition only ``from the left''  by using $\widehat{\psi}_3^-$, 
   while ``from the right'' the partition is fine. In this way, we arrive at  
 \medskip
 
 \noindent
\textbf{Step $n$.} We are given the points in 
\[
J_{n-1}=\{a_{n-1},\dots, a_1, b_1, \dots, b_{n-1}\}
\]
with $a_{n-1}\leq a_{n-2}\leq \dots \leq a_1 \leq b_1 \leq \dots \leq b_{n-1}$.
Expressing $J_{n-1}$ as $\{\tilde{\subd}^1, \dots, \tilde{\subd}^{m}\}$ with $m\leq 2(n-1)$ and $\tilde{\subd}^j < \tilde{\subd}^{j+1}$ for $j \in \{1, \dots, m-1\}$, it holds that
\begin{equation}\label{2012202018}
\psi(s) \geq \frac{1}{2}\Big(\psi(\tilde{\subd}^j) + \psi(\tilde{\subd}^{j+1}) \Big) - \eta \quad \text{for every } s\in (\tilde{\subd}^j, \tilde{\subd}^{j+1}) \quad \text{and every } j \in \{1, \dots, m-1\}\,.
\end{equation} 
If $a<a_{n-1}$ we consider the function
\begin{equation*}\label{2012202023}
 \widehat{\psi}_{n-1}^-(x):=
 \begin{dcases}
 \psi(x)& \quad x \in [a,a_{n-1})\,,\\
 \liminf_{r\to a_{n-1}^-}\psi(r)& \quad x=a_{n-1}\,,
 \end{dcases}
 \end{equation*}
 and if $b_{n-1}<b$ we  also introduce  the function
 \begin{equation*}\label{2012202024}
 \widehat{\psi}_{n-1}^+(x):=
 \begin{dcases}
 \psi(x)& \quad x \in (b_{n-1},b]\\
\liminf_{r\to b_{n-1}^+}\psi(r) & \quad x=b_{n-1}. 
 \end{dcases}
 \end{equation*}
Let $m_n^\pm:=\min \widehat{\psi}_{n-1}^\pm$. We have 
 exist and 
 \begin{equation}\label{2012202055}
 m_n^\pm\geq m_{n-1}^\pm +\eta\geq \dots \geq m_1 + (n-1)\eta\,.
 \end{equation}
 We repeat the procedure performed in the previous step:
 we set  
 \begin{subequations}\label{eqs:2012202056}
\begin{equation}\label{2012202056}
A_n^-:=\{x \in [a, a_{n-1}] \colon \widehat{\psi}_{n-1}^-(x) \leq m_n^-+\eta\}\,,\quad A_{n}^+:=\{x \in [b_{n-1},b] \colon \widehat{\psi}_{n-1}^+(x) \leq m_n^++\eta\}\,,
\end{equation}
\begin{equation}\label{2012202057}
a_n:=\min A_n^-\,,\qquad b_n:=\max A_n^+\,.
\end{equation}
\end{subequations}
By construction
\begin{subequations}\label{eqs:2012201859'}
\begin{equation}\label{2012201859'}
\psi(s)\geq\frac{1}{2}\Big(\psi(a_n) + \psi(a_{n-1}) \Big) - \eta \quad\text{for }s \in (a_n, a_{n-1})\,,
\end{equation}
\begin{equation}\label{2012201900'}
\psi(s)\geq\frac{1}{2}\Big(\psi(b_{n-1}) + \psi(b_n) \Big) - \eta \quad\text{for }s \in (b_{n-1}, b_n)\,,
\end{equation}
where possibly $a_{n-1}=a_n$, $b_{n-1}=b_n$.
\end{subequations}\vspace{0.5em}\\
\textbf{Conclusion.}
The procedure can be then iterated, obtaining at each application a larger set of nodes satisfying the analogue of \eqref{2012201629}. In order to obtain exactly \eqref{2012201629} it is enough to find $n_\eta$ such that both $a_{n_\eta} =a$ and $b_{n_\eta}=b$.
In fact, this can be done in a finite number of steps: since $\psi(a)$, $\psi(b) \in \R$, then 
at most $\frac{\max\{\psi(a),\psi(b)\}}{\eta}+1$ steps are needed, in view of \eqref{2012202055} and the definition of $a_n$, $b_n$ (cf.\ \eqref{eqs:2012202056}).
\qed

\subsection{An elliptic regularity estimate}
\label{ss:appD}
\noindent
The following  proposition 
 uses elliptic regularity arguments as a way to derive an estimate 
 that will ultimately enhance
 the compactness properties of a suitable sequence. 


\begin{proposition}\label{prop:regolaritaellittica}
Let $(u_k)_k \subset H_{\Dir}^1(\Omega;\R^n)$, $(z_k)_k\subset \Hs(\Omega)$, $(p_k)_k \subset L^2(\Omega;\MD)$ fulfill
  $u_k \weaksto u$ in $\BD(\Omega)$,  $z_k \weakto z$ in $\Hs(\Omega)$, and $p_k \to  p$ in $\Mb(\Omega; \Mnn)$, 
  and suppose that $(u_k,z_k,p_k)$
 satisfy the equations
\begin{equation}\label{0206231033}
-\mathrm{div}(\C(z_k) \EE u_k)=-\mathrm{div}(\C(z_k) p_k) \quad\text{in }\Omega \quad \text{for every }k\in \N \,,
\end{equation}
where $\C(z)=\CV(z)\isoC$, with  $\CV \in \rmC^{1,1}(\R)$
 as in \eqref{C_dMdSS} 
 and $\isoC$ isotropic.
\par
Then,  $(\nabla u_k)_k$ is a Cauchy sequence w.r.t.\ the convergence in measure. 
  In particular,  there exists a not relabeled subsequence of $(u_k)_k$ such that  $\nabla u_k \to \nabla u$  a.e.\ in $\Omega$. 
\end{proposition}

Proposition  \ref{prop:regolaritaellittica} is based on a corresponding elliptic regularity result from \cite[Thm.\ 9.1]{DalDesSol11}
that we recall below for  the sake of completeness. 

\begin{remark}
\label{rmk:purtroppo}
We observe that \cite[Theorem~9.1]{DalDesSol11} is based on an explicit representation formula for solutions to the problem of linear elasticity that holds  for  $\isoC$ isotropic (in particular with constant coefficients).   That is why, for the validity of Proposition \ref{prop:regolaritaellittica}
we  have  to resort to the structural condition  $\C(z) = \CV(z) \isoC$.  
\end{remark}

 \begin{theorem}{\cite[Thm.\ 9.1]{DalDesSol11}}
 Let   $ \isoC  \in  \Lin(\Mnn;\Mnn)$ be positive definite, symmetric,  and \emph{isotropic}. Then, for every open set $\Omega' \subset \subset \Omega$
 there exists a constant $C_{\isoC,\Omega'}>0$, only depending on $\isoC$, $\Omega$ and  $\Omega'$, such that if 
 $\pt  \in \Lnn$ and $\ut \in H_\loc^1 (\Omega;\R^n)$ satisfy the equation
 \[
-\mathrm{div} (\isoC \EE(\ut))=  -\mathrm{div} (\isoC \pt) \quad \text{in } \Omega\,,
 \]
then we have the estimate
\begin{equation}
\label{DMDSS-reg}
\|\nabla \ut \|_{1,w,\Omega'} \leq C_{\isoC,\Omega'} \left( \|\pt\|_{L^1(\Omega)}{+} \|\ut\|_{L^1(\Omega)}\right),
\end{equation}
 where  $\|\cdot\|_{1,w,\Omega'}$  denotes the weak $L^1$ norm $\|f\|_{1,w,\Omega'}:=\sup_{t>0} \mathcal{L}^n\big( \{ |f|>t\} \cap \Omega' \big)$. 
\end{theorem}
\par
We are now in a position to prove  our  result.

\begin{proof}[Proof of Proposition  \ref{prop:regolaritaellittica}]
First of all, let us suitably rewrite \eqref{0206231033}:
by Leibniz Formula
\begin{equation}\label{0206231429}
\nabla(\CV(z) u)= \CV(z) \nabla u + \nabla (\CV(z)) \otimes u, \quad \EE(\CV(z) u)= \CV(z) \EE u + \nabla (\CV(z)) \odot u \,,
\end{equation}
so that  \eqref{0206231033} reformulates as 
\begin{equation}\label{0206231430}
- \diver\big( \isoC\EE(\CV(z_k)u_k)  \big)= - \diver \Big(\isoC\big(\CV(z_k)p_k\big)+\isoC\big(\nabla(\CV(z_k)\big) \odot u_k) \Big).
\end{equation}
In particular, we have 
\[
\begin{aligned}
&
- \diver\big( \isoC \EE(\CV(z_k)u_k {-}\CV(z_h)u_h )  \big)
\\
& = - \diver \Big(\isoC\big[ (\CV(z_k)p_k{-}\CV(z_h)p_h) + (\nabla(\CV(z_k)\big) {\odot} u_k {-} \nabla(\CV(z_h)\big) {\odot} u_h) \big] \Big).
\end{aligned}
\]
Let us now fix an open set $\Omega'$
 compactly contained in $\Omega$ and apply estimate \eqref{DMDSS-reg} with choices 
 \[
 \ut= (\CV(z_k)u_k{-} \CV(z_h)u_h), \qquad \pt= \big[ (\CV(z_k)p_k{-}\CV(z_h)p_h) + (\nabla(\CV(z_k)\big) {\odot} u_k {-} \nabla(\CV(z_h)\big) {\odot} u_h) \big]\,.
 \]
 for any two indices $h,k \in \N$ with $h\leq k$. 
  Thus, we obtain 
\begin{equation}\label{0206231710}
\begin{aligned}
&
\|\nabla \big(\CV(z_k)u_k- \CV(z_h)u_h \big)\|_{1,w,\Omega'}
\\ & \leq C_{\isoC,\Omega'}  \Big(\|\CV(z_k)p_k-\CV(z_h) p_h\|_{L^1} 
+ \|\nabla(\CV(z_k)) \odot u_k 
 -\nabla(\CV(z_h)) \odot u_h\|_{L^1} 
 \\
  & \qquad + \|\CV(z_k) u_k- \CV(z_h) u_h\|_{L^1}\Big)
 \\& \leq  C_{\isoC,\Omega'} S   \Big(  \|\CV\|_{\mathrm{Lip}} \|z_k-z_h\|_{L^\infty} +  \|p_k-p_h\|_{L^1} + \|u_h-u_k\|_{L^1} 
+   \|\CV'\|_{\mathrm{Lip}}   \|z_k-z_h\|_{L^\infty} 
\\
 & \qquad + \|\nabla(z_k-z_h)\|_{L^\gamma} +  \|u_k-u_h\|_{L^{\gamma'}}  \Big)\,.
\end{aligned}
\end{equation}
Here, $\gamma >n$ is such that $\Hs(\Omega)$, $m>n/2$ is   \emph{compactly}  embedded into $W^{1,\gamma}(\Omega)$ (that is, $\mathrm{m}-n/2 > 1-n/\gamma$),
and 
\[
S: =  \sup_k\left(\|p_k\|_{L^1}+ \|\nabla z_k\|_{L^\gamma} +   \|u_k\|_{L^{n/(n{-}1)}}   + \|\CV\|_{\mathrm{Lip}}   + \|\CV'\|_{\mathrm{Lip}}\right) 
\]
  (notice that $(u_k)_k$ is bounded in  $L^{n/(n{-}1)}(\Omega;\R^n)$ as it is weakly$^*$ converging in $\BD(\Omega)$).  
  \par
Since $z_k \weakto z$ in $\Hs(\Omega)$,  we have $ \|z_k-z_h\|_{L^\infty}\to 0$ and $ \|\nabla(z_k-z_h)\|_{L^\gamma} \to 0$ as $h,k\to\infty$; likewise, 
the weak$^*$ convergence of $(u_k)_k$ in  $\BD(\Omega)$ yields strong convergence in $L^p(\Omega;\R^n)$ for every $1\leq p<\frac{n}{n-1}$, and thus $\|u_k-u_h\|_{L^{\gamma'}} \to 0$. Finally, observe that, since  $(p_k)_k \subset L^2(\Omega;\MD)$, 
$\|p_k {-}p_h\|_{\calM_{\mathrm{b}}(\Omega)} = \|p_k{-}p_h\|_{L^1(\Omega)}  $ and thus $\|p_k{-}p_h\|_{L^1(\Omega)}  \to 0 $ as $k,h\to\infty$. All in all, 
%
  from \eqref{0206231710} we conclude that  
  $\Big(\nabla\big( \CV(z_k) u_k \big) \Big)_k$ is a Cauchy sequence with respect to the convergence in measure in $\Omega$ and, by \eqref{0206231430} and the uniform convergence of $z_k$, we get the same for the sequence $\Big(\CV(z_k)\nabla  u_k \Big)_k$. 
     Now, from $z_k \weakto z$ in $\Hs(\Omega)$ we deduce that 
  $\CV(z_k) \to \CV(z)$ in $\rmC(\overline\Omega)$. Exploiting the fact that $\CV(z_k(x))\geq c_{\CV}>0$ and $\CV(z(x)) \geq c_{\CV}>0$ for all $x\in \overline\Omega$
  thanks to \eqref{C_dMdSS}, 
  from the property that  $\Big(\CV(z_k)\nabla  u_k \Big)_k$ is a Cauchy sequence w.r.t.\ convergence in measure in $\Omega$ we deduce that
 also $(\nabla u_k)_k$ is.  
\end{proof}

\subsection{An auxiliary duality result}
\label{ss:gen-duality}
Let $(\gBa,\gHi,\gBa^*)$ be a Hilbert triple (i.e.,  the Hilbert space $\gBa$ fulfills $\gBa \subset \gHi$ densely and continuously) and consider the extended $\gHi$-norm on $\gBa^*$
\[
\mathfrak{f}_{\gHi}(\xi): = \begin{cases}
\| \xi \|_{\gHi} & \text{if } \xi \in \gHi,
\\
+\infty & \text{otherwise}\,. 
\end{cases}
\]
Let $\gR: \gHi \to [0,+\infty]$ be a positively $1$-homogeneous functional, with $\gR(0)=0$. 
Consider the convex conjugate and the subdifferential of $\gR$ in the $\gBa^*$-$\gBa$ duality, i.e.
\[
\begin{aligned}
& 
\gR^* : \gBa^* \to [0,+\infty], && \gR^*(w): = \sup_{v \in \gBa} \left( \langle w, v \rangle_{\gBa} {-} \gR(v) \right),
\\
& 
\partial \gR : \gBa \rightrightarrows \gBa^* && \partial \gR(v) = \{ w\in \gR^* \, : \gR(\eta) - \gR(v) \geq  \langle w, \eta{-} v \rangle_{\gBa}  \ \forall\, \eta \in \gBa\}\,.
\end{aligned}
\] 
\begin{lemma}
\label{l:gen-duality}
In the above setup, there holds
\begin{equation*}
\inf_{w \in \partial \gR(0)} \mathfrak{f}_{\gHi}(\xi{-}w) = \sup \left\{ \pairing{}{\gBa}{\xi}{\eta} - \gR(\eta)\, : \ \eta \in \gBa\,, \  \| \eta \|_{\gHi} \leq 1 \right\}\,.
\end{equation*}
\end{lemma}
\begin{proof}
By $1$-homogeneity of $\gR$  we have that 
\[
\gR^*(w) = I_{\partial \gR(0)}(w)  = \begin{cases} 0  & \text{if } w \in \partial \gR(0),
\\
+\infty & \text{otherwise}
\end{cases}
\qquad \text{for all } w \in \gBa^*\,.
\]
Analogously, the convex conjugate $\mathfrak{f}_{\gHi}^* : \gBa \to [0,+\infty]$ is given in terms of the indicator function of the unit ball $B_{1}^{\gHi}(0)$ in $\gHi$
\[
\mathfrak{f}_{\gHi}^*(\eta) = I_{B_{1}^{\gHi}(0)}(\eta)  = \begin{cases} 0  & \text{if } \|\eta\|_{\gHi} \leq 1,
\\
+\infty & \text{otherwise}
\end{cases}
\qquad \text{for all } \eta \in \gBa\,.
\]
Then,
\[
\begin{aligned}
\sup \left\{ \pairing{}{\gBa}{\xi}{\eta} - \gR(\eta)\, : \ \eta \in \gBa\,,\ \| \eta \|_{\gHi} \leq 1 \right\}  & = \sup_{\eta \in \gBa}
\left\{ \pairing{}{\gBa}{\xi}{\eta} - \gR(\eta) -  I_{B_{1}^{\gHi}(0)}(\eta) \right\} 
\\
 & (\gR{+}I_{B_{1}^{\gHi}(0)})^*(\xi)  
 \\
 & \stackrel{(1)}{=}  \left( \gR^*{\stackrel{\inf}{\circ}} I_{B_{1}^{\gHi}(0)}^*\right)(\xi) 
 \\
 & 
 \stackrel{(2)}{=}  \inf_{w \in \gR^*} \left( I_{B_{1}^{\gHi}(0)}^*(\xi{-}w) + \gR^*(w) \right) 
 \\
 & 
 \stackrel{(3)}{=}  \inf_{w \in \gR^*} \left( \mathfrak{f}_{\gHi}(\xi{-}w) + I_{\partial \gR(0)}(w) \right)  
\end{aligned}
\]
where {\footnotesize (1)} follows from, e.g., the duality result from \cite[Theorem 1, p.\ 178]{IofTih79TEPe}, which is applicable since 
$0 \in \mathrm{dom}(\gR) \cap \mathrm{dom}(I_{B_{1}^{\gHi}(0)})  $ and $I_{B_{1}^{\gHi}(0)}$ is continuous at $0$.  Furthermore, 
 {\footnotesize (2)} follows  the definition of the $\inf$-convolution $\gR^* {\stackrel{\inf}{\circ}} I_{B_{1}^{\gHi}(0)}^*$, while  {\footnotesize (3)} 
 ensues from the fact that $ I_{B_{1}^{\gHi}(0)}^*= \mathfrak{f}_{\gHi}^{**} = \mathfrak{f}_{\gHi}$. 
The proof is concluded. 
\end{proof}
\begin{remark}
\label{rmk:proof-duality}
Clearly,  the representation formula \eqref{duality-formulae-1} 
for  $\congdistname p$
immediately follows with 
$\gBa=\gHi = L^2(\Omega;\MD) $, $\mathfrak{f}_{\gHi}= \| \cdot\|_{L^2(\Omega;\MD)}$, and $\gR = \mathcal{H}(z,\cdot)$.
Analogously, \eqref{duality-formulae-2} for  $\congdistname z$
 ensues with 
$\gBa = \Hs(\Omega)$, $\gHi = L^2(\Omega)$, $\mathfrak{f}_{\gHi}= \mathfrak{f}_{L^2(\Omega)}$ and $\gR= \mathcal{R}$. 
\end{remark}

\subsection{A result on the convergence of the stresses}
\label{ss:appE}
The following lemma is inspired by \cite[Proposition~6.1]{FraGia2012}.
\begin{lemma}\label{0404232220}
Let $\Omega$ be a Lipschitz domain and $(\sigma_k )_k \in \Lnn$ be such that $\sigma_k \weakto \sigma$ in $\Lnn$, $\diver \sigma_k \to \diver \sigma$ in $L^2(\Omega; \Rn)$, and $(\sigma_k)_{\dev} \to \sigma_{\dev}$ in $L^2(\Omega;\MD)$. Then $\sigma_k \to \sigma$ in $\Lnn$.
\end{lemma}


\begin{proof}
We recall (cf.\ \cite{Necas66}) 
that for a given Lipschitz domain $\Omega$ and every $p\in (1,\infty)$, if $u\in W^{-1,p}(\Omega)$ with $\nabla u \in W^{-1,p}(\Omega; \Rn)$, then $u\in L^p(\Omega)$ with
\begin{equation}\label{0404232236}
\|u\|_p \leq C(\Omega) (\|u\|_{W^{-1,p}} + \|\nabla u\|_{W^{-1,p}}).
\end{equation}
Decomposing $\sigma_k$, $\sigma$ as
\begin{equation}\label{04042249}
\sigma_k= \widehat{\sigma}_k \mathrm{I} + (\sigma_k)_{\dev}, \quad
\sigma=\widehat{\sigma} \mathrm{I} +\sigma_{\dev}
\end{equation}
we have that
\begin{equation}\label{04040232241}
\nabla \widehat{\sigma}_k= \diver(\widehat{\sigma}_k \mathrm{I})=\diver \sigma_k-\diver((\sigma_k)_{\dev})\to \nabla \widehat{\sigma} \qquad\text{in } H^{-1}(\Omega;\Rn),
\end{equation}
in view of the embedding of $L^2(\Omega;\R^n)$ into $H^{-1}(\Omega;\R^n)$.
Moreover, since $L^2(\Omega;\Mnn)$ is compactly embedded into $H^{-1}(\Omega;\Mnn)$, it holds that $\sigma_k \to \sigma$ in $H^{-1}(\Omega;\Mnn)$. We have also $\widehat{\sigma}_k \to \widehat{\sigma}$ in $H^{-1}(\Omega)$ 
 by \eqref{04042249} and since  $(\sigma_k)_{\dev} \to \sigma_{\dev}$ in $L^2(\Omega;\MD)$.
Therefore, using \eqref{0404232236} with $p=2$, $u=\widehat{\sigma}_k-\widehat{\sigma}$ and employing \eqref{04040232241},  we deduce that
\begin{equation*}\label{0404232251}
\widehat{\sigma}_k \to \widehat{\sigma} \quad \text{in }L^2(\Omega).
\end{equation*}
We conclude with \eqref{04042249} and since $(\sigma_k)_{\dev} \to \sigma_{\dev}$ in $L^2(\Omega;\MD)$.
\end{proof}

\subsection{Approximation of the work of the external forces}
\label{ss:5-prelim}
\noindent Throughout this section,  we let $ (\partg k\ell)_{\ell=0}^{\cardipart k} $, with 
\[
0=\partg k 0<\partg k 1 <\ldots <\partg k {\cardipart k} =S \quad \text{ and } \quad \max_{\ell =1,\ldots, \cardipart k} (\partg k\ell {-}\partg k {\ell-1}) \to 0 \text{ as } k\to\infty,
\]
 be a generic partition of the interval $[0,S]$. 
\par
Our first result concerns the approximation of the integral 
$\int_0^S \langle \serifsigma, \EE( \sfw') \rangle_{L^2(\Omega)} \dd s$
that contributes to the work of the external forces. 
\begin{lemma}
\label{l:approx-sigma}
There holds  
\begin{equation}\label{2112200941}
\begin{aligned}
& \lim_{k\to \infty} \sum_{\ell=1}^{\cardipart k} \langle \serifsigma(\partg k{\ell-1}), \EE (\sfw(\partg k{\ell})-\sfw(\partg k{\ell-1})) \rangle_{L^2(\Omega)} 
\\
&  = \int_0^S \langle \serifsigma(s), \EE {\sfw}'(s) \rangle_{L^2(\Omega)} \ds 
= \lim_{k\to \infty} \sum_{\ell=1}^{\cardipart k} \langle \serifsigma(\partg k{\ell}), \EE (\sfw(\partg k{\ell})-\sfw(\partg k{\ell-1})) \rangle_{L^2(\Omega)}\,.
\end{aligned}
\end{equation}
In particular,
\begin{equation*}
\lim_{k\to \infty} \sum_{\ell=1}^{\cardipart k} \tfrac12  \langle \serifsigma(\partg k{\ell-1}){+}  \serifsigma(\partg k{\ell}), \EE (\sfw(\partg k{\ell})-\sfw(\partg k{\ell-1})) \rangle_{L^2(\Omega)} = \int_0^S \langle \serifsigma(s), \EE {\sfw}'(s) \rangle_{L^2(\Omega)} \ds \,.
\end{equation*}
\end{lemma}
\begin{proof}
We define $\ul \serifsigma_k : [0,S]\to \Lnn$ via $\ul \serifsigma_k(0) = \serifsigma(0)$, $\ul \serifsigma_k(\partg k{\cardipart k}) = \serifsigma(S)$, and  
\begin{equation*}
\ul \serifsigma_k(s) \equiv \serifsigma(\partg k{\ell-1}) \quad\text{in }(\partg k{\ell-1}, \partg k{\ell})\,.
\end{equation*}
Clearly, we have that 
 \begin{equation*}
 \sum_{\ell =1}^{\cardipart k} \langle \serifsigma(\partg k{\ell-1}), \EE (\sfw(\partg k{\ell})-\sfw(\partg k{\ell-1})) \rangle_{L^2(\Omega)} = \int_0^S  \langle\ul \serifsigma_k(s), \EE {\sfw}'(s) \rangle_{L^2(\Omega)} \ds\,.
 \end{equation*}
 Now,  since the mapping  $[0,S] \ni s\mapsto\serifsigma(s) \in \Lnn $ is weakly continuous 
 thanks to
  \eqref{additional-regularity}, it is immediate to check that, as $k\to\infty$, 
 \[
 \langle\ul \serifsigma_k(s), \EE {\sfw}'(s) \rangle_{L^2(\Omega)} \longrightarrow \langle \serifsigma(s), \EE {\sfw}'(s) \rangle_{L^2(\Omega)} \qquad\foraa\, s \in (0,S)\,.
 \]
 Then, taking into account that $\serifsigma \in L^\infty(0,S;\Lnn)$ and $\sfw \in H^1(0,S;H^1(\Omega;\R^n))$ thanks to \eqref{dir-load}, 
  by the dominated convergence theorem  we immediately conclude that 
 \[
 \lim_{k\to\infty}  \int_0^S  \langle\ul \serifsigma_k(s),  \EE  {\sfw}'(s) \rangle_{L^2(\Omega)} \ds = \int_0^S \langle \serifsigma(s), \EE {\sfw}'(s) \rangle_{L^2(\Omega)} \ds\,,
 \]
 which gives the first of   \eqref{2112200941}. 
Similarly, with   the sequence $\ol \serifsigma_k: [0,S]\to \Lnn$
defined by  $\ol \serifsigma_k(s) \equiv \serifsigma(\partg k{\ell})  $ in  $ (\partg k{\ell-1}, \partg k{\ell})$
 in place of $\ul \serifsigma_k$ we obtain the second of  \eqref{2112200941}.
\end{proof}
Mimicking the very same arguments as in the proof of Lemma  \ref{l:approx-sigma}  we also obtain the following results.
\begin{lemma}
\label{l:approx-Fw'}
There holds  \begin{equation*}
\lim_{k\to \infty} \sum_{\ell=1}^{\cardipart k} \tfrac12  \langle \sfF(\partg k{\ell-1}){+}  \sfF(\partg k{\ell}), \sfw(\partg k{\ell})-\sfw(\partg k{\ell-1}) \rangle_{\BD(\Omega)} = \int_0^S \langle \sfF(s), {\sfw}'(s) \rangle_{\BD(\Omega)} \ds \,.
\end{equation*}
\end{lemma}
\begin{lemma}
\label{l:approx-F'w}
There holds  \begin{equation*}
\begin{aligned}
&
\lim_{k\to \infty} \sum_{\ell=1}^{\cardipart k}  \langle \sfF(\partg k{\ell}){-}  \sfF(\partg k{\ell}-1),  \frac12 [  (\sfu(\partg k{\ell-1}){+}\sfw(\partg k{\ell-1}))  {+}
(\sfu(\partg k{\ell}){+}\sfw(\partg k{\ell}))]  
 \rangle_{\BD(\Omega)} 
 \\
 & = \int_0^S \langle \sfF'(s), \sfu(s) {+}\sfw(s) \rangle_{\BD(\Omega)} \ds \,.
 \end{aligned}
\end{equation*}
\end{lemma}

\bigskip

{\small
\section*{Acknowledgments} 
The authors have been funded by the Italian Ministry of University and Research through different projects:
MIUR - PRIN project 2017BTM7SN 
\emph{Variational Methods for stationary and evolution problems with singularities and interfaces},
MIUR - PRIN project 2020F3NCPX \emph{Mathematics for Industry 4.0}, 
MIUR - PRIN project 2022HKBF5C 
\emph{Variational Analysis of complex systems in Materials Science, Physics and Biology}.
G.L.\ and R.R.\ have been partially supported by the  Gruppo Nazionale per  l'Analisi Matematica, la  Probabilit\`a  e le loro Applicazioni (GNAMPA) of the Istituto Nazionale di Alta Matematica (INdAM).
}

\bigskip

\end{document}